\documentclass[twoside,leqno]{article}
\usepackage{amsmath}
\usepackage{amsthm}
\usepackage{amssymb}
\usepackage{amscd}
\usepackage{graphicx}
\usepackage{epsfig}

\usepackage{latexsym}
\usepackage{amsfonts}
\input xy
\xyoption{all} \tolerance=500
=1in \textwidth 16cm \topmargin 5mm\textheight24cm \advance\topmargin by -\headheight\advance\topmargin by
-\headsep

\numberwithin{equation}{section} \allowdisplaybreaks

\newtheorem{theorem}{Theorem}[section]

\newtheorem{proposition}{Proposition}[section]
\newtheorem{corollary}{Corollary}[section]

\theoremstyle{definition}
\newtheorem{definition}{Definition}[section]
\newtheorem{example}{Example} [section]

\newtheorem{remark}{Remark}[section]

\font\black=cmbx10 \font\sblack=cmbx7 \font\ssblack=cmbx5 \font\blackital=cmmib10  \skewchar\blackital='177
\font\sblackital=cmmib7 \skewchar\sblackital='177 \font\ssblackital=cmmib5 \skewchar\ssblackital='177
\font\sanss=cmss10 \font\ssanss=cmss8 scaled 900 \font\sssanss=cmss8 scaled 600 \font\blackboard=msbm10
\font\sblackboard=msbm7 \font\ssblackboard=msbm5 \font\caligr=eusm10 \font\scaligr=eusm7 \font\sscaligr=eusm5
 \font\fraktur=eufm10 \font\sfraktur=eufm7 \font\ssfraktur=eufm5

\font\bsymb=cmsy10 scaled\magstep2
\def\all#1{\setbox0=\hbox{\lower1.5pt\hbox{\bsymb
       \char"38}}\setbox1=\hbox{$_{#1}$} \box0\lower2pt\box1\;}
\def\exi#1{\setbox0=\hbox{\lower1.5pt\hbox{\bsymb \char"39}}
       \setbox1=\hbox{$_{#1}$} \box0\lower2pt\box1\;}

\def\tx#1{{\fam0\relax#1}}

\newfam\bifam
\textfont\bifam=\blackital \scriptfont\bifam=\sblackital \scriptscriptfont\bifam=\ssblackital

\newfam\blfam
\textfont\blfam=\black \scriptfont\blfam=\sblack \scriptscriptfont\blfam=\ssblack

\newfam\bbfam
\textfont\bbfam=\blackboard \scriptfont\bbfam=\sblackboard \scriptscriptfont\bbfam=\ssblackboard

\newfam\ssfam
\textfont\ssfam=\sanss \scriptfont\ssfam=\ssanss \scriptscriptfont\ssfam=\sssanss
\def\sss#1{{\fam\ssfam\relax#1}}

\newfam\clfam
\textfont\clfam=\caligr \scriptfont\clfam=\scaligr \scriptscriptfont\clfam=\sscaligr

\newfam\frfam
\textfont\frfam=\fraktur \scriptfont\frfam=\sfraktur \scriptscriptfont\frfam=\ssfraktur

\def\hpb#1{\setbox0=\hbox{${#1}$}
    \copy0 \kern-\wd0 \kern.2pt \box0}
\def\vpb#1{\setbox0=\hbox{${#1}$}
    \copy0 \kern-\wd0 \raise.08pt \box0}

\def\pmb#1{\setbox0\hbox{${#1}$} \copy0 \kern-\wd0 \kern.2pt \box0}
\def\pmbb#1{\setbox0\hbox{${#1}$} \copy0 \kern-\wd0
      \kern.2pt \copy0 \kern-\wd0 \kern.2pt \box0}
\def\pmbbb#1{\setbox0\hbox{${#1}$} \copy0 \kern-\wd0
      \kern.2pt \copy0 \kern-\wd0 \kern.2pt
    \copy0 \kern-\wd0 \kern.2pt \box0}
\def\pmxb#1{\setbox0\hbox{${#1}$} \copy0 \kern-\wd0
      \kern.2pt \copy0 \kern-\wd0 \kern.2pt
      \copy0 \kern-\wd0 \kern.2pt \copy0 \kern-\wd0 \kern.2pt \box0}
\def\pmxbb#1{\setbox0\hbox{${#1}$} \copy0 \kern-\wd0 \kern.2pt
      \copy0 \kern-\wd0 \kern.2pt
      \copy0 \kern-\wd0 \kern.2pt \copy0 \kern-\wd0 \kern.2pt
      \copy0 \kern-\wd0 \kern.2pt \box0}

\newenvironment{pf}{{\noindent{\it Proof. }}}{\ \rule{2mm}{2.5mm}\medskip}

\mathchardef\za="710B  %\alpha
\mathchardef\zb="710C  %\beta
\mathchardef\zg="710D  %\gamma
\mathchardef\zd="710E  %\delta
\mathchardef\zve="710F %\epsilon
\mathchardef\zz="7110  %\zeta
\mathchardef\zh="7111  %\eta
\mathchardef\zvy="7112 %\theta
\mathchardef\zi="7113  %\iota
\mathchardef\zk="7114  %\kappa
\mathchardef\zl="7115  %\lambda
\mathchardef\zm="7116  %\mu
\mathchardef\zn="7117  %\nu
\mathchardef\zx="7118  %\xi
\mathchardef\zp="7119  %\pi
\mathchardef\zr="711A  %\rho
\mathchardef\zs="711B  %\sigma
\mathchardef\zt="711C  %\tau
\mathchardef\zu="711D  %\upsilon
\mathchardef\zvf="711E %\phi
\mathchardef\zq="711F  %\chi
\mathchardef\zc="7120  %\psi
\mathchardef\zw="7121  %\omega
\mathchardef\ze="7122  %\varepsilon
\mathchardef\zy="7123  %\vartheta
\mathchardef\zf="7124  %\varomega
\mathchardef\zvr="7125 %\varrho
\mathchardef\zvs="7126 %\varsigma
\mathchardef\zf="7127  %\varphi
\mathchardef\zG="7000  %\Gamma
\mathchardef\zD="7001  %\Delta
\mathchardef\zY="7002  %\Theta
\mathchardef\zL="7003  %\Lambda
\mathchardef\zX="7004  %\Xi
\mathchardef\zP="7005  %\Pi
\mathchardef\zS="7006  %\Sigma
\mathchardef\zU="7007  %\Upsilon
\mathchardef\zF="7008  %\Phi
\mathchardef\zW="700A  %\Omega

\newcommand{\be}{\begin{equation}}
\newcommand{\ee}{\end{equation}}
\newcommand{\ra}{\rightarrow}

\newcommand{\bea}{\begin{eqnarray}}
\newcommand{\eea}{\end{eqnarray}}
\newcommand{\beas}{\begin{eqnarray*}}
\newcommand{\eeas}{\end{eqnarray*}}
\def\*{{\textstyle *}}
\newcommand{\R}{{\mathbb R}}

\newcommand{\C}{{\mathbb C}}
\newcommand{\Z}{{\mathbb Z}}
\newcommand{\N}{{\mathbb N}}
\newcommand{\K}{{\mathbb K}}

\newcommand{\we}{\wedge}
\newcommand{\nn}{\nonumber}
\newcommand{\ot}{\otimes}

\newcommand{\pa}{\partial}
\newcommand{\ti}{\times}
\newcommand{\A}{{\cal A}}
\newcommand{\Li}{{\cal L}}

\newcommand{\ad}{{\rm ad}}

\newcommand{\Ll}{{\pounds}}

\def\ran{\rangle}

\def\Ker{\operatorname{Ker}}
\def\im{\operatorname{Im}}
\def\DO{\operatorname{DO}}

\def\ute{{\bar{\zt}}}
\def\uP{{\bar{P}}}
\def\cA{{\cal A}}

\def\cE{{\cal E}}

\def\cM{{\cal M}}

\def\cJ{{\cal J}}
\def\cT{{\cal T}}

\def\cX{{\cal X}}
\def\Po{{\stackrel{o}{P}}}

\def\Pe{{\stackrel{e}{P}}}

\newcommand{\ao}{\mathbf{1}^n}
\def\wh{\widehat}
\def\wt{\widetilde}
\def\ol{\overline}

\def\Sec{\sss{Sec}}

\def\la{\langle}
\def\ran{\rangle}
\def\fp{\mathfrak{p}}
\def\sD{{\sss D}}

\def\sJ{{\sss J}}

\def\sT{{\sss T}}

\def\xd{\tx{d}}
\def\xi{\tx{i}}

\def\dt{\xd_{\sss T}}

\def\dTs{\xd_{\sss T}^*}
\def\dts{\sT^*}
\newcommand{\BB}[2]{\ensuremath{\{\!\!\{ #1 , #2 \}\!\!\}_\cM}}
\newcommand{\BBP}[2]{\ensuremath{\{\!\!\{ #1 , #2 \}\!\!\}}}

\begin{document}

\title{Graded contact manifolds \\ and 
contact Courant algebroids
\thanks{Research
supported by the Polish Ministry of Science and Higher Education under the grant N N201 416839.} }
\author{Janusz Grabowski}
\date{}

\maketitle
\begin{abstract}
We develop a systematic approach to contact and Jacobi structures on graded supermanifolds. In this framework, contact structures are interpreted as symplectic principal $\R^\ti$-bundles. Gradings compatible with the $\R^\ti$-action lead to the concept of a graded contact manifold, in particular a linear (more generally,  $n$-linear) contact structure. Linear contact structures are proven to be exactly the canonical contact structures on first jets of line bundles. They provide linear Kirillov (or Jacobi) brackets and give rise to  the concept of a Kirillov algebroid, an analog of a Lie algebroid, for which the corresponding cohomology operator is represented not by a vector field (de Rham derivative) but a first-order differential operator. It is shown that one can view Kirillov or Jacobi brackets as homological Hamiltonians on linear contact manifolds. Contact manifolds of degree 2 are studied, as well as contact analogs of Courant algebroids. We define lifting procedures that provide us with constructions of canonical examples of the structures in question.



\end{abstract}
%\eject
\setcounter{section}{0}

\section{Introduction}
The role played in geometry by symplectic structures cannot be overestimated. The theory of symplectic and Poisson
manifolds, together with the theory of (pseudo)Riemannian ones, are two legs of the contemporary differential
geometry. It is not only because of the fact that each other really important geometric structure has its
symplectic origins, but also because of the prominent position of symplectic forms and Poisson brackets in
physics, especially in classical mechanics, quantum field theory, and quantization. Actually,
supergeometric versions of symplectic structures appeared first in the physics literature before being recognized by
mathematicians. In particular, graded symplectic supermanifolds have been known to physicists since the 1970s,
providing framework for the so-called {\it BRST formalism}. Its main geometrical background consists of an odd
symplectic structure such that the fields and their respective anti-fields are conjugate with respect to the
corresponding odd Poisson bracket $(\cdot, \cdot)$ (the `antibracket'), and an action functional $S$ obeying
what is called the {\it classical master equation} $\{ S,S\}=0$.

In the early 1980s Batalin and Vilkovisky \cite{BV} developed a generalization of the BRST procedure in terms
of so called {\it Batalin-Vilkovisky algebras} which allows us, in principle, to handle symmetries of
arbitrary complexity, while a detailed mathematical study of the classical BRST algebra and its quantization
was undertaken in the late 1980s by Kostant and Sternberg \cite{KS} who related the BRST cohomology to the Lie
algebra cohomology.

In mid 1990s, in turn,  Alexandrov, Kontsevich, Schwarz, and Zaboronsky \cite{AKSZ} found a simple and elegant
procedure (referred to as the {\it AKSZ formalism} now) for constructing solutions to the classical
master equation.  Their approach uses mapping spaces of supermanifolds equipped with  a compatible integer
grading and a differential (see also \cite{Roy1}).

On the other hand, in 1986 Courant and Weinstein \cite{CW,Co} developed the concept of a Dirac structure providing a geometric setting for Dirac's theory of constrained mechanical systems which was then algebraically generalized and applied to integrable systems by Dorfman \cite{Do}. The integrability condition for a Dirac structure makes use of
a natural bracket operation on sections of the {\it Pontryagin bundle} $\cT M=\sT M\oplus_M\sT^\ast M$, the direct
sum of the tangent and the cotangent bundles, known nowadays as the {\it Courant bracket}. The Courant bracket, originally skew-symmetric, does not satisfy the Leibniz rule with respect to the multiplication by functions nor the Jacobi identity.
Its non-symmetric version we will use in the sequel, going back to Dorfman, fulfils the Jacobi identity for the price of skew-symmetry.
The nature of the Courant bracket itself remained unclear until several years later when it was observed by Liu, Weinstein and Xu \cite{LWX} that $\cT M$ endowed with the Courant bracket plays the role of a `double object' in the sense of Drinfeld \cite{Dr} for
a pair of Lie algebroids (see \cite{MX}) over $M$. Let us recall that, in complete analogy with Drinfeld's Lie
bialgebras, in the category of Lie algebroids there also exist `bi-objects', {\it Lie bialgebroids},
introduced by Mackenzie and Xu \cite{MX} as linearizations of Poisson groupoids. It is well known that every
Lie bialgebra has a double which is a Lie algebra that is not true for a general Lie bialgebroid. Instead,
Liu, Weinstein and Xu showed that the double of a Lie bialgebroid is a more complicated structure they call a
{\it Courant algebroid}, the Pontryagin bundle $\cT M$ with the Courant bracket being a special case.

Only recently become clear (see \cite{Roy0, Roy}) that, again, any Courant algebroid is actually a certain
Hamiltonian system on a graded symplectic supermanifold $\cM$. The standard Courant algebroid $\cT M=\sT
M\oplus_M\sT^\ast M$ corresponds in this way to the symplectic supermanifold $\cM=\sT^*\zP\sT M$ with the
Hamiltonian associated with the de Rham vector field $\xd$ on $\zP \sT M$. Note that all the theory of Lie
algebroids or, equivalently, linear Poisson structures can be put into the framework of Courant algebroids.

Our aim in this paper is to develop and study contact versions of these formalisms, therefore including
contact structures on graded supermanifolds. Of course, these structures are well known (see e.g. \cite{Sch0, Sch} or their appearance in the classification of simple Lie superalgebras \cite{Ka}), but these deserve a more detailed study in the context of Courant algebroids.

We show that canonical contact structures, playing the role of the cotangent bundle $\sT^*\cM$ in the symplectic case, are associated not with the trivial bundles $\R\ti\cM$ but with, generally nontrivial, line
bundles $L\to\cM$. As a result, we must work not with derivations,
vector fields, and Poisson brackets, but with first-order differential operators acting on sections of line
bundles and {\it Jacobi brackets} \cite{DLM,Li} or, better to say, {\it Kirillov brackets}, as Kirillov was
the first who studied local Lie brackets on sections of line bundles \cite{Ki}.

Our understanding of a supermanifold will be standard, i.e., we will use an atlas of local coordinates, some of
which are even and some of which are odd. For details we refer to \cite{DM,Rog, Tuy, Var} and references therein. In
this paper we will work also with {\it graded manifolds} in the sense of Voronov  \cite{Vor1} which are supermanifolds with an additional (usually $\Z^n$ or $\N^n$) gradation in the structure sheaf, so that we can choose local coordinates which have
certain weights besides the parity. In such cases there is often a confusion about the relation between
the additional grading (weight) and the $\Z_2$-grading (parity) responsible for the sign rule. We will follow almost literally the Voronov's setting, sharing his point of view that the grading (weight) is not directly related to parity, unless explicitly assumed, and work mainly with $\N^n$-gradings. However, we will not automatically assume, as it is done by Voronov \cite{Vor1}, that the local coordinates of non-zero degree are "cylindrical".   Besides the main reference \cite{Vor1}, some information about graded manifolds one can find, for instance, in first chapters of the book \cite{Tuy} and Mehta \cite{Meh}. Certain  important particular cases can be also found in the works of Kontsevich \cite{Kon}, Roytenberg
\cite{Roy0,Roy1}, \v Severa \cite{Sev}, Voronov \cite{Vor2}, and Grabowski and Rotkiewicz \cite{GR, GR1}.
Since we will not develop a general theory of graded manifolds, we will only point out main subtleties that
differ graded manifolds from standard manifolds, on one hand, and from pure supermanifolds, on the other.

In a standard description, a contact structure on a manifold $M$ is viewed as certain `maximally
non-integrable' one-codimensional subbundle in the tangent bundle $\sT M$. This subbundle is locally the
kernel of a 1-form on a manifold $M$ which induces a rank-one subbundle in $\sT^* M$.  Two such forms are viewed equivalent if they have the same kernel, i.e.,
they differ by a factor which is an invertible function. On the other hand, with every contact or, more
generally, Jacobi structure one associates its {\it symplectification} (resp., {\it Poissonization})
\cite{Li,GIMPU, Sch0}. In our approach, we will often identify the symplectification, which is a
certain symplectic submanifold in the (canonically symplectic) cotangent bundle $\sT^*M$, with the contact structure itself.

Viewing contact structures as particular symplectic manifolds is very convenient in the
context of the graded (super)geometry in which we will work.
Actually, the symplectification will be
understood as a principal $\R^\ti$-bundle $P$ over $\cM$ equipped with an (even) symplectic form which is
homogeneous with respect to the $\R^\ti$-action. Such an object we will call a {\it symplectic principal $\R^\ti$-bundle}. The corresponding {\it Legendre} (called also {\it Lagrange}) {\it  bracket} is a local Lie bracket (we will call such brackets {\it Kirillov brackets}) on sections of the line bundle $L^*$, dual to the line bundle $L$ associated with $P$. We simply interpret sections of $L^*$ as 1-homogeneous functions on $P$ and use the fact that they are closed with respect to the symplectic Poisson bracket. In the case of an even contact form $\za$  on a manifold $\cM$, this principal $\R^\ti$-bundle $P$ is just the one-dimensional subbundle $\la\za\ran$ in $\sT^*\cM$ generated by $\za$, with the 0-section removed. We consider as well {\it odd contact structures} but  in the introduction we will speak about the even case not to produce an additional complexity. Of course, one can consider as well homogeneous Poisson structures on $P$ which give rise to {\it principal Poisson $\R^\ti$-bundles}. In any case, the corresponding Kirillov brackets are, like the Legendre brackets for contact structures, restrictions of principal Poisson brackets to 1-homogeneous functions.

The main observation in this context is that the canonical linear contact structure on the first jet bundle
$\sJ^1L$ of a line bundle $L\to\cM$ can be identified with the cotangent bundle $\sT^*(L^*)^\ti$ of the
principal $\R^\ti$-bundle $(L^*)^\ti$ associated with the dual $L^*$. The cotangent bundle $\sT^*(L^*)^\ti$ has a canonical symplectic embedding into a line subbundle $C_L$ of $\sT^*\sJ^1L$, thus the corresponding Legendre bracket is defined on sections of $C^*_L$. If $L$ is trivial, $L=\R\ti\cM$, then $\sJ^1L=\R\ti\sT^*\cM$ with the canonical contact form $\za=\xd z-p_a\xd x^a$, and the Legendre bracket is the `standard Legendre bracket' of functions on $\R\ti\sT^*\cM$,
$$\{ F,G\}_\za =\frac{\pa F}{\pa{p_a}}\frac{\pa G}{\pa
x^a}-\frac{\pa F}{\pa x^a}\frac{\pa G}{\pa{p_a}}+\frac{\pa F}{\pa z}\left(G-p_a\frac{\pa G}{\pa{p_a}}\right)-\left(F-\frac{\pa F}{\pa{p_a}}p_a\right)\frac{\pa G}{\pa z}\,.
$$
The cotangent bundle
$\sT^*(L^*)^\ti$ is canonically also a symplectic principal $\R^\ti$-bundle over the first jet bundle $\sJ^1L$. Actually, it is a double bundle, as it is simultaneously a vector bundle, and both structures are compatible in the sense that the Euler vector field of the vector bundle structure is 1-homogeneous with respect to the $\R^\ti$-action; we will call such structure a {\it linear principal $\R^\ti$-bundle} and the whole contact structure, when a homogeneous
symplectic form of the bi-degree $(1,1)$ is given, a {\it linear contact structure}. One of important tools will be therefore the theory of $n$-graded (super)manifolds and $n$-graded principal $\R^\ti$-bundles, with {\it double} or, more generally, {\it $n$-tuple  vector bundles} as basic examples. A natural
definition refers to compatibility condition expressed in terms of the commutation of the corresponding
{\it homogeneity structures} represented by the homotheties induced by the weight or Euler vector fields, as
described in \cite{GR,GR1}.

More precisely, an $n$ graded (super)manifold is a (super)manifold $\cM$ equipped
with $n$ homogeneity structures, i.e. smooth actions $h^i:\R\ti\cM\to\cM$, $i=1\dots,n$, of the
monoid $(\R,\cdot)$ of multiplicative reals. The compatibility condition says that all the {\it homotheties} commute, $h^i_t\circ h^j_s=h^j_s\circ h^i_t$, where $h^i_t=h^i(t,\cdot)$. Each homogeneity structure induces an $\N$-gradation on $\cM$ represented by a certain weight vector field, so we obtain an $\N^n$-gradation. For $n=1$, these are just {\it N-manifolds} in the sense of \v{S}evera and Roytenberg, if the weight of a function represents its parity. One difference with a principal $\R^\ti$-action is that $h^i_0$ represents a projection onto a submanifold in $\cM$,
the base of the corresponding fibration, that is not present in the case of an $\R^\ti$-bundle, but we can
consider the same compatibility condition between $\R^\ti$-actions and homogeneity structures. This leads to
the concept of an {\it $n$-graded principal $\R^\ti$-bundle}. Such a bundle, equipped with a 1-homogeneous (with respect
to the $\R^\ti$-action) symplectic form of weight $k\in\N^n$, becomes an {\it $n$-graded contact structure of
degree $k$}. If, additionally, the total weight represents the parity, we speak about a
{\it contact $k$-manifold}. Since there are canonical procedures of lifting weight vector fields and
$\R^\ti$-actions to the tangent and cotangent bundles, we can produce $n$-graded contact structures out of
$(n-1)$-graded principal $\R^\ti$-bundles, in particular, linear contact structures out of pure principal
$\R^\ti$-bundles. For example, starting with the trivial bundle $\R^\ti\ti\cM$, we obtain the canonical contact
structure on $\sT^*\cM\ti\R$.

Analogously one can define {\it $n$-graded}, in particular, {\it linear  principal Poisson $\R^\ti$-bundles}
which, according to a general knowledge, correspond to certain `algebroids', called by us {\it
Kirillov algebroids}. A Kirillov algebroid can be equivalently defined as an `invariant' Lie
algebroid structure on a vector bundle $E$ over a principal $\R^\ti$-bundle $P_0$. If $P_0$ is trivial,
$P_0=\R^\ti\ti\cM$, then we actually deal with what was called a {\it Jacobi algebroid} in \cite{GM1} (or a
{\it generalized Lie algebroid} in \cite{IM}). The corresponding cohomology can be defined analogously to the
Lie algebroid case by means of a homological `$Q$-operator' and, in the spirit of Va{\u\i}ntrob \cite{Va}, a
Kirillov algebroid can be viewed as a `Q-line bundle'. This time, however, $Q$ is not a vector field but a
first-order differential operator. Of course, we can recover in this way some cohomology known in the Jacobi setting, e.g. the {\it Lichnerowicz-Jacobi cohomology}, as particular cases of our construction.

Note that canonical linear  principal Poisson $\R^\ti$-bundles arise as the tangent lifts of Poisson principal $\R^\ti$-bundles (Kirillov brackets). The corresponding Kirillov algebroid can be viewed then as associated with the Kirillov bracket. In particular, a Jacobi bracket on $\cM$ induces in this way a Lie algebroid structure on $\sT^*\cM\ti\R$, as has been observed by Kerbrat and Souici-Benhammadi \cite{KSB} (see also \cite{Vs,GM1}).

A natural contact analog of the theorem stating that any linear symplectic structure (an NQ-manifold of degree
1 in the terminology of \v Severa \cite{Sev}) is equivalent to a cotangent bundle $\sT^*\cM$ equipped with its canonical
symplectic form, is Theorem \ref{tr2} which implies that any linear even (resp., odd) contact structure is isomorphic with the canonical contact structure on the first jet bundle of a line bundle or, equivalently, with the cotangent bundle $\sT^\ast \uP$ (resp., $\zP \sT^\ast \uP$) of a principal $\R^\ti$-bundle $\uP$. This is generalized to $n$-linear
contact structures in the spirit of \cite[Theorem 8.1]{GR}. The supergeometric setting is crucial in interpreting
Kirillov brackets on a line bundle $L$ as bi-homogeneous Hamiltonians on $\zP\sT^*(L^*)^\ti$: 1-homogeneous with respect to the $\R^\ti$-action and quadratic with respect to the linear structure, which are homological with respect to the Poisson (or Legendre) bracket.

A natural problem is then the description of {\it contact 2-manifolds} (contact N-manifolds of degree 2 in
the terminology of Roytenberg and \v Severa)  which turn out to be
symplectic principal $\R^\ti$-bundles of degree 2. A contact analog of the Roytenberg's result \cite[Theorem
3.3]{Roy} says that contact 2-manifolds are in one-to-one correspondence with pseudo-Euclidean principal
$\R^\ti$-bundles (Theorem \ref{cN2}). A canonical example is $\sT^*[2]F[1]$, where $F$ is a linear principal
$\R^\ti$-bundle, more particularly $\sT^*[2]\sT[1](\R^\ti\ti M)$.

Next, in view of the Roytenberg's description of Courant algebroids \cite{Roy1}, 
{\it contact Courant algebroids} are defined as contact 2-manifold equipped with an
$\R^\ti$-homogeneous cubic homological Hamiltonian. The `classical' description is proven to be the following:
a contact Courant algebroid is a vector bundle $\cE$ over a manifold $M$ equipped with

 - a Loday (Leibniz) bracket $\{\cdot,\cdot\}^{\! 1}$ on sections of $\cE$,

 - a pseudo-Euclidean product $\la\cdot,\cdot\ran^{\! 1}$ with values in a line bundle $L$,

 - a vector bundle morphism $\zr^{\! 1}:\cE\ra\DO^1(L,L)$, associating with any section $X$ of $\cE$ a
first-order differential operator $\zr^{\! 1}(X)$ acting on sections of $L$ such that
$$\la \{ X,Y\}^{\! 1},Y\ran^{\! 1}=\la X,\{ Y,Y\}^{\! 1}\ran^{\! 1}\quad\text{and}\quad\zr^{\! 1}(X)\la Y,Y\ran^{\! 1}=2\la\{ X,Y\}^{\! 1},Y\ran^{\! 1}$$
for all sections $X,Y$ of $\cE$. The canonical homogeneous Courant algebroid structure on
$\cT (\R^\ti\ti M)$ gives rise to a canonical Courant-Jacobi algebroid on
$(\R\ti\sT M)\oplus_M(\R^*\ti\sT^* M)$ considered already in \cite{GM2}.

We should admit that some of our observations concerning contact forms and contact structures, especially in the pure even case, can be found spread over the existing literature.
The existence of a correspondence between Jacobi manifolds and degree-one contact NQ-manifolds was previously
mentioned by \v{S}evera \cite{Sev}. More recently, Antunes and Laurent-Gengoux \cite{AL-G} studied Jacobi
structures from the supergeometric point of view. Additionally, Jacobi and contact structures on
supermanifolds were considered in a series of papers by Bruce \cite{Bru, Bru2, Bru3}. Their approaches, however, are more traditional than what we propose here. After writing these
notes we realized also that some of these questions, in the simplest version concerning contact forms and Jacobi
brackets, have been recently studied by Mehta \cite{Meh1}. He mentioned also that developing a general
theory of contact NQ-manifolds in the degree 2 case should provide a natural generalization of Courant
algebroids, together with a `Courantization' process, and that this approach may be useful in studying
Jacobi-Dirac and generalized contact structures \cite{IW, PW}. This is exactly what we propose in this note and
our main results include:

-- the development of a concept of a (symplectic) principal $\R^\ti$-bundle, its compatibility with linear structures and its phase lift;

-- showing that linear symplectic principal $\R^\ti$-bundles are exactly the canonical contact structures of first jets of line bundles and a universal character of the latter leading to a natural concept of super-Jacobi structures and Kirillov  super-brackets;

-- a natural association of Lie algebroids with Poisson principal $\R^\ti$-bundles with a particular example of the Lie algebroid of a Jacobi bracket;

-- the concepts and descriptions of a Kirillov algebroid (with a Jacobi algebroid as a particular example) and a contact Courant algebroid and the corresponding cohomology.

\medskip
Our paper is organized as follows. In section 2 we present rudiments  of the graded differential geometry, define weight vector fields, their homogeneity completeness, and homogeneity structures, as well as even and odd Poisson and Jacobi brackets, and we define lifting procedures for graded structures. Section 3 is devoted to introducing $n$-graded principal $\R^\ti$-bundles which are fundamental objects in our framework. Contact forms on supermanifolds and contact structures are defined in Section 4, where also their equivalence with symplectic principal $\R^\ti$ bundles is proven. Moreover, local forms of  even and odd contact forms on supermanifolds are derived.
Graded, especially linear contact structures are introduced in Section 5, together with a fundamental example associated with first jets of line bundles which is studied deeper in Section 6. In turn, Section 7 is devoted to Jacobi and, more generally, Kirillov brackets. Linear principal Poisson structures as corresponding to contact analogs of Lie algebroids, {\it Kirillov algebroids}, are investigated in Section 8. As a by-product, we present in Section 9 a full description of contact $n$-vector bundles. In particular, all linear contact structures turn out to be just the canonical ones associated with first jet bundles. In Section 10 we provide basic concepts concerning the cohomology of Kirillov algebroids. In Section 11 and 12 we introduce and study contact 2-manifolds and and contact Courant algebroids, contact analogs of symplectic manifolds of degree 2 and Courant algebroids.
\section{Rudiments of graded differential geometry}

\subsection{Graded manifolds and weight vector fields}
A general theory of $\Z$-graded manifolds, which we will follow in this paper, have been developed in \cite{Vor1}.  Of course, we can allow us to work with $\K^n$-graded manifolds, $\K=\Z,\N$. In simple words, a {\it
$\K^n$-graded manifold}  will be understood as a supermanifold $\cM$ with an additional $\Z^n$ (resp., $\N^n$)
grading in the structure sheaf. Besides the degree (parity) $g\in\Z_2=\{ 0,1\}$, we have therefore defined
also the weight $w\in\K^n$ which should be {\it compatible} with the degree in the sense that the subsheaf of
homogeneous functions $F$ of degree $i=(i_1,\dots,i_n)\in\K^n$, $w(F)=i$, splits into subspaces of odd,
$g(F)=1$, and even $g(F)=0$, functions. One can also think that we actually have a $\K^n\ti\Z_2$-grading in
the structure sheaf. Of course, the algebra generated by homogeneous functions should be sufficiently large.
We will assume that there is an atlas of so called {\em affine charts} whose coordinate functions have defined
weight and parity (i.e., they are $\K^n\ti\Z_2$-homogeneous) and the weight and parity are respected by the
change of coordinates.  Thus the space $\cA^i(\cM)$ of functions $f$ which are {\em homogeneous of weight $i$}
(we will write also $w(f)=i$) is well defined and the {\em algebra of homogeneous functions} or a {\it
homogeneous algebra} $\cA(\cM)=\bigoplus_{i\in \K^n}\cA^i(\cM)$ is dense in the algebra $C^\infty(\cM)$ of
smooth functions with respect to a reasonable $C^\infty$-topology. In particular, homogeneous functions span
all finite jets. Morphisms in the category of graded manifolds are smooth maps respecting the grading. The
subalgebra in $\cA(\cM)$ consisting of homogeneous functions in $\N^n$ will be called the {\it polynomial
algebra} of $\cM$.

\medskip
The example of a vector bundle shows that the degree is not directly related to the parity, but we will consider
as well the case where the grading {\it respects the parity}, i.e. where homogeneous functions of weight
$i=(i_1,\dots,i_n)\in\K^n$ are even if the total weight $\vert i\vert =\sum_s\vert i_s\vert$ is even, while
functions with odd total weight are odd. For such structures, called {\it $\K^n$-manifolds}, the
$\Z_2$-grading is induced from the $\K^n$-grading. A $\K^n$-manifold $\cM$ possesses a coordinate atlas (of so
called {\em affine charts}) with local coordinates $(x^a)$ to which {\em weights}
$w_a=w(x^a)=(w_1(x^a),\dots,w_n(x^a))\in\K^n$ are assigned (and respected by the change of coordinates), and
induce the commutation rules
$$x^ax^b=(-1)^{\vert w_a\vert\cdot\vert w_b\vert}x^bx^a.$$ In other words,
$g(x^a)=(-1)^{\vert w_a\vert}$.

Of course, some of the above concepts can be extended  to gradings more general than those with degrees
indexed by $\Z^n$ or $\N^n$. Note that $\N^n$-manifold is a particular case of an $\Z^n$-manifold and
h-complete 1-manifolds of degree $k\in\N$ are N-manifolds of degree $k$ as defined in \v Severa \cite{Sev} and
Roytenberg \cite{Roy1}. Recall that an N-manifold is of degree $k\in \N$ if there is an atlas with coordinates
of weights $\le k$ and that N-manifolds of degree $k\in\N$ we will call just {\it $k$-manifolds}. In what
follows the term {\it manifold} will mean {\it supermanifold}. In this sense an `ordinary' manifold will be an
{\it even manifold}.

Let us observe that the $\K^n$-grading can be also conveniently encoded by means of of the collection of {\em
weight vector fields} which are jointly {\it diagonalizable}, i.e. there is an atlas of charts with local
coordinates in which
\be\label{euler} \zD_\cM^s=\sum_aw^s_ax^a\pa_{x^a}\,,\quad s=1,\dots,n\,,\ee
where $w^s_a=w^s(x^a)\in\K$. Indeed, $f\in\A^i(\cM)$ if and only if $\zD_\cM^s(f)=i_sf$. Each coordinate is
even or odd, so the weight vector fields are even. Each weight vector field $\zD_\cM^s$ defines a submanifold
$\cM_s$ of $\cM$ on which it vanishes. It is defined locally by constrains $x^a=0$ for $w_a\ne 0$ and local
coordinates on $\cM_s$ are induced by those local coordinates on $\cM$ which are of weight 0.

\begin{definition} We say that an $\N^n$-graded manifold is of {\it degree $k\in\N^n$} if local coordinates can be
chosen such that their weights $w^s$ satisfy $w\preceq k$, and of {\it total degree $r\in\N$} if local coordinates can be
chosen such that all the total weights $|w_a|=|w(x^a)|=|w^1_a|+\cdots+|w^n_a|$ do not exceed $r$. Here, the
partial order $\preceq$ is defined by $l\preceq k\Leftrightarrow\forall j\ [l_j\le k_j]$.
\end{definition}

The list of weight vector fields can be completed by the {\it parity diffeomorphism}
\be\label{parity} Par(x^a)=g(x^a)x^a\ee preserving all weight vector fields.

The weight vector fields give rise to one-parameter groups of automorphisms $\zF^s_t$, $s=1,\dots,n$, of the
algebra of homogeneous functions, defined by $\zF^s_u(f)=e^{ui_s}f$ for $f\in\A^i(\cM)$. We can equivalently
write down this action as an action $h^s$ of multiplicative group of positive reals by putting
$h_t^s=\zF^s_{\ln(t)}$. In the case of an $\N^n$-gradation, these automorphisms can be naturally extended to
actions $h^s:\R\ti\A(\cM)\ni(t,f)\mapsto h^s_t(f)\in\A(\cM)$ of the monoid $(\R,\cdot)$ of multiplicative
reals. All this action can be collected into one action $h:\R^n\ti\A(\cM)\ni(t,f)\mapsto h_t(f)\in\A(\cM)$ of
the monoid $(\R^n,\cdot)$. (i.e., $\R^n$ with coordinate-wise multiplication) on the homogeneous algebra via
the algebra homomorphisms,
\be\label{monoid}
h_t(f)=\prod_st_s^{i_s}f,\quad f\in\A^i(\cM)\,.
\ee
In other words, $h_t$ is the composition of pairwise commuting maps $h^s_{t_s}$, $s=1,\dots,n$, where
$h^s_{t_s}$ acts on $\A^i$, $i\in\N^n$, as the multiplication by $t_s^{i_s}$.
\begin{definition} We say that the weight vector field $\zD_\cM^s$ {\it homogeneitically
complete} ({\it h-complete} for short), if the actions $h^s$ can be extended to a smooth action of the monoid
$(\R,\cdot)$ on the supermanifold $\cM$. This means that $h^s_{t}$ is actually a diffeomorphism of $\cM$ if
$t\ne 0$, and it is a smooth surjection on certain submanifold $\cM_s$ if $t= 0$. Such actions we will call
{\it homogeneity structures} (cf. \cite{GR1}). Homogeneitically complete vector fields
$\zD=\sum_aw_ax^a\pa_{x^a}$ with $w_a=0,1$ we will call {\it Euler vector fields}. We say that an
$\N^n$-graded manifold $(\cM,\zD^1_\cM,\dots,\zD^n_\cM)$ is {\it homogeneitically complete} (shortly: {\it
h-complete}), if all its weight vector fields (\ref{euler}) are h-complete, so they define homogeneity structures
$h^1,\dots,h^n$. Homogeneitically complete $\N^n$-manifolds $(\cM,h^1,\dots,h^n)$ of degree $k\in\N^n$  (total
degree $r$) we will call {\it $n$-graded manifolds} (or {\it $n$-graded bundles}) {\it of degree $k$} ({\it
total degree $r$}). If the grading respects the parity, an $n$-graded manifold of degree $k\in\N^n$ we will
call simply a {\it manifold of degree $k$}, a {\it $k$-manifold}, or a {\it $k$-bundle}.
\end{definition}
\begin{remark}\label{r1.2} The reason for calling h-complete graded manifolds `bundles' is that the submanifold $\cM_s$
of zeros of the weight vector field $\zD^s$ is in this case a base of the locally trivial fibration
$h^s_0:\cM\ra\cM_i$. As the homogeneity structures pairwise commute, we have the whole diagram of fibrations:
\be\label{mfib}h_0^{s_1}\circ\cdots\circ h_0^{s_l}:\cM\ra\cM_{s_1,\dots,s_l}=\cM_{s_1}\cap\cdots\cap\cM_{s_l}
\ee
with fibers being {\it homogeneity spaces} (cf. \cite{GR,GR1}).
\end{remark}

\medskip\noindent It is obvious that weight vector fields on an $\N^n$-graded
manifold pairwise commute. Adapting the methods of \cite{GR1} to supermanifolds, one can show that for
h-complete weight vector fields the converse is also true: if they commute they are jointly diagonalizable.
\begin{theorem}
Pairwise commuting actions $h^1,\dots,h^n$ of the monoid $(\R,\cdot)$ on a manifold $\cM$ induce an
$\N^n$-gradation on $\cM$. In particular they are jointly diagonalizable.
\end{theorem}
As Euler vector fields correspond exactly to vector bundle structures on $\cM$ (see \cite{GR}), an {\it
$n$-vector bundle} can be equivalently characterized as an  $n$-graded manifold of the degree $\ao\in\N^n$ (or
the total degree $1$) \cite{GR}. If $\cM$ is just an ordinary (purely even) manifold, then this concept of an
$n$-vector bundle coincides with that of Pradines \cite{Pr1} and Mackenzie \cite{Mac1a, Mac2} (see also
\cite{KU,Vor2}). Moreover, $n$-manifolds are exactly superanalogs of {\it $n$-tuple homogeneity structures}
(or {\it $n$-tuple graded bundles})  \cite{GR1}.

\begin{remark}\label{r1} Note that, in general, the algebra of homogeneous functions $\A(\cM)$ does
not determine $\cM$. For, consider an ordinary vector bundle $E$ over $M$ viewed as a purely even manifold
$\cM=E$. The vector bundle structure induces a natural $\N$-grading in the subalgebra of $C^\infty(E)$ of
functions which are polynomial along fibers: homogeneous functions of degree $k$ are $k$-homogeneous with
respect to the corresponding Euler vector field $\zD_E$, $\zD_E(f)=kf$. In other words, basic functions have
degree $0$, and functions linear along fibers (i.e., functions associated with sections of the dual bundle
$E^\ast$) have degree 1, and together generate the algebra of polynomial functions. These polynomial
functions, when restricted to an open strip over the zero-section, generate the same polynomial algebra, but
this open strip with restricted polynomials is not diffeomorphic with the vector bundle, as there is no smooth
diffeomorphism of the open strip onto the whole vector bundle respecting degrees of homogeneous functions.
Note that Voronov, Roytenberg, and \v Severa assume from the very beginning that coordinates of degree different from 0 are "cylindrical", i.e. the weight vector field is complete. In general. however, there is no point to exclude manifolds with a general gradation which can be uncomplete.
\end{remark}

The weight defined for functions can be canonically extended to tensor fields on $\cM$: a tensor
field $\zL$ is of weight $i=(i_1,\dots,i_n)$ if $\Ll_{\zD_\cM^s}\zL=i_s\cdot\zL$, $s=1,2,\dots, n$, where
$\Ll$ denotes the Lie derivative. We have in particular  $w(\zD_\cM^s)=0$ which means exactly that the weight
vector fields commute. Similarly, the parity diffeomorphisms can be extended to a map on tensor fields,
mapping a tensor field $K$ to $\overline{K}$. If $K=K^0+K^1$ is the decomposition of the tensor field into the
even and the odd part, then $\ol{K}=K^0-K^1$. Note that, for a vector field $X$ and a one-form $\za$, we have
$\ol{i_X\za}=i_{\ol{X}}\ol{\za}$.

\medskip
One can naturally view the tangent and the cotangent bundles of a $\K^n$-graded manifold $\cM$ as $\K^{n+1}$-graded (see \cite{Vor1}). On a $\K^n$-graded manifold we can also shift the grading with the use of a new weight vector field
$\zD=\sum_aw_ax^a\pa_{x^a}$, which is jointly diagonalizable with the old weight vector fields, by passing to a new collection of weight vector fields: $\zD$ and $\zD_\cM'^{s}=\zD_\cM^s+\zD$, $s=1,\dots,n$. The
corresponding graded supermanifold we will denote $\cM(\zD)$. If $\zD$ is an Euler vector field, we can also
change the parity according to the weights induced by $\zD$. The new atlas is consistent and we get a new
graded supermanifold in which the grading remains the same but the parity has been changed according to the
rule $g'(x^a)=g(x^a)+w^a\ (mod\ 2)$. The corresponding graded supermanifold we will denote $\zP_\zD\cM$. If
the vector bundle structure represented by $\zD$ is clear, we will use the standard notation $\Pi\cM$ instead
of $\zP_\zD\cM$. Of course, we can change both: the grading into the one given by the weight vector fields
$\zD_\cM'^{s}$ and the parity according to the above rule. The corresponding graded supermanifold we will
denote $\cM[\zD]$. The latter operation produces $\K^n$- manifolds from $\K^n$-manifolds. If $\zD$ is the
Euler vector field of a fixed vector bundle structure, we usually write $\cM(r)$ and $\cM[r]$ instead of
$\cM(r\zD)$ and $\cM[r\zD]$, respectively.

We can also construct new gradings using the observation that  the sum of (h-complete) weight vector fields is
a (h-complete) weight vector field again. In particular, any $n$-graded manifold $(\cM,h^1,\dots,h^n)$ is
canonically a $1$-graded manifold $(\cM,h)$: we just add all the weight vector fields and get a weight vector
field corresponding to the composition of all homotheties, $h_t=h^1_t\circ\cdots\circ h^n_t$. In this way we
can view double vector bundles (like $\sT\sT M$ or $\sT^*\sT M$) as 2-manifolds (N-manifolds of degree 2).

\subsection{Brackets}
\begin{definition} An {\it even} (resp., {\it odd}) {\it Jacobi bracket}
in an associative superalgebra $\cA$ with unit 1 is an even (resp., odd) bilinear operation
$\{\cdot,\cdot\}:\cA\ti\cA\ra\cA$ such that
\bea\label{s0} &g\left(\{ \zf,\zc\}\right)=g(\zf)+g(\zc)+ k\ (mod\ 2)\,,\\
\label{s1}&\{ \zf,\zc\}=-(-1)^{(g(\zf)+k)(g(\zc)+k)}\{ \zc,\zf\}
\,,\\
&\{\{ \zf,\zc\},\zg\}=\{ \zf,\{ \zc,\zg\}\}-(-1)^{(g(\zf)+k)(g(\zc)+k)}\{ \zc,\{ \zf,\zg\}\}\,,\label{s2}\\
&\{ \zf,\zg\zc\}=\{\zf,\zg\}\zc+(-1)^{(g(\zf)+k)g(\zg)}\zg\{ \zf,\zc\} -\{\zf,1\}\zg\zc\,, \label{s3}
\eea
where $k=0$ (resp., $k=1$). If $1$ is a central element, we speak about an even (resp., odd) {\it Poisson
bracket}. If $\cA$  is $\K^n$ graded and $k\in\K^n$, we say that the Jacobi bracket is {\it of weight} $k$ if
$w\left(\{ \zf,\zc\}\right)=w(\zf)+w(\zc)+ k$.
\end{definition}
\noindent The property (\ref{s1}) we call {\it super-anticommutativity} and (\ref{s2}) {\it super-Jacobi
identity}. They say together that the bracket is a (super)Lie bracket. The property (\ref{s3}), {\it
generalized super-Leibniz rule}, just tells us that the adjoint map $ad_\zf=\{\zf,\cdot\}:\cA\ra\cA$ is a
first-order differential operator on $\cA$ of parity $g(\zf)+k\ (mod\ 2)$.

It follows (cf. \cite[theorem 12]{GM2}) that the Jacobi bracket is represented by a first-order bidifferential
operator
\be\label{jac} \zL+\zG\cdot I+fI\cdot I\,,
\ee
where $\zL$ is an even (resp., odd) bi-derivation which is (anti)symmetric in an appropriate way, $\zG$ is and
even (resp., odd) derivation, $I$ is the identity understood as a zero-order differential operator, $f$ is 0
(resp., an odd) function, and `$\cdot$' is an appropriate exterior product of first-order linear differential
operators. In the pure even case, $f=0$ and the manifold equipped with the pair $(\zL,\zG)$ is traditionally
called a {\it Jacobi manifold} \cite{Li}.

In fact (see \cite[Theorem 4.2]{Gr92}), if $\cA$ has no nilpotents, then any Lie bracket on $\cA$ given by a
bi-differential operator must be of first order, so is a Jacobi bracket. This is an algebraic generalization of
a geometric result due to Kirillov \cite{Ki}: any local Lie bracket on sections of a line bundle over an even
manifold is of first order with respect to each argument.
This justifies our definition, this time for arbitrary (super)manifolds.
\begin{definition}\label{de1}
An even (resp., odd) Lie bracket defined on sections of an even line bundle $L$ over a manifold $\cM$ which
is a first-order differential operator with respect to each argument will be called an even (resp., odd) {\it
Kirillov bracket}.
\end{definition}
\noindent Obviously, in a local trivialization of the line bundle any Kirillov bracket is represented by a
Jacobi bracket. The Kirillov bracket is even or odd if the local Jacobi bracket is even or odd, respectively.

\begin{remark}\label{rem1} By definition, if $[\cdot,\cdot]_L$ is a Kirillov bracket on sections of a line bundle $L$ over $\cM$ and $\zs$ is a section of $L$, then the adjoint operator
$\ad_\zs=[\zs,\cdot]_L$ is a first-order differential operator from $L$ into $L$, thus a section of the
corresponding vector bundle $\DO^1(L,L)$. There is a canonical isomorphism $\DO^1(L,L)=\sT L^\ti/\R^\ti$,
where $\R^\ti$ is the group of multiplicative real, $L^\ti$ is the principal $\R^\ti$-bundle associated with
$L$ (cf. Section \ref{secprinc}), i.e. $\DO^1(L,L)$ is the {\it Atiyah algebroid} of the principal
$\R^\ti$-bundle $L^\ti$ and sections of $\DO^1(L,L)$ can be viewed as invariant vector fields on $L^\ti$. This
vector bundle is canonically a Lie algebroid with the (super)commutator bracket and the anchor given by
passing to the principal part of a first-order differential operator (which is a vector field on $\cM$) or,
equivalently, as the projection of the invariant vector field on $L^\ti$ onto $\cM$.
\end{remark}

\begin{remark} It is clear that Jacobi brackets are Kirillov brackets on trivial line bundles. In our opinion, there is no particular reason to study Jacobi brackets rather than Kirillov brackets.
First of all, we must stress that even a Jacobi bracket is genuinely a module bracket and the product $\zg\zc$
in (\ref{s2}) should be understood as the left-module product (where $\zg$ is an element of the algebra and
$\zc$ is an element of the left-regular $\cA$-module) rather than the algebra product. This is because of the
role of this product under isomorphism of the Jacobi brackets which do respect the left-regular module
structure and do not respect the algebra structure (see \cite{Gr00}). In particular, the unity 1 regarded as a
module element can be replaced by another invertible element (non-vanishing section) that cannot be done by an
algebra isomorphism. The fact that the product $\zg\zc$ refers to the module and not to the algebra structure
is, of course, hard to be seen in the algebra (trivial bundle) case and becomes obvious only when passing to
general line bundles.

From this point of view, the concept of Jacobi bracket is an auxiliary concept which serves well locally but
has to be naturally extended for an arbitrary line bundle. The concept of a first-order differential operator
works well for modules of sections of such bundles and the distinction of a particular section (constant) is
not needed. This is what differs this situation from the case of Poisson brackets, where we deal with
derivations of associative algebras which are operators vanishing on constants.
\end{remark}

Given a supermanifold $\cM$ we can consider the cotangent bundle $\sT^\ast\cM$ with local Darboux coordinates
$(x^a,p_b)$, where the degree of $x^a$ and that of $p_a$ being the same as the degree of $x^a$ on $\cM$. Note
that vector fields on $\cM$ are represented by functions on $\sT^*\cM$ linear in the fibers, so linear in
$p_b$'s. In this case the canonical symplectic form $\zw_\cM=\xd p_a\xd x^a$ on $\sT^\ast\cM$ is even and
induces an even Poisson bracket $\{\cdot,\cdot\}_\cM$ on $C^\infty(\sT^*\cM)$ which is closed on the
superalgebra $\A(\sT^*\cM)$ of polynomial functions on $\sT^*\cM$. The latter algebra can be identified with
the algebra of symmetric multivector fields on $\cM$, the algebra of symbols of differential operators.

If we reverse the parity of the momenta, $g(p_a)=g(x^a)+1 (mod\ 2)$, then we work on the manifold
$\zP\sT^\ast\cM$ with the canonical symplectic form $\zw_\cM^\zP=\xd p_a\xd x^a$ which is now odd and induces
an odd Poisson bracket $\{\cdot,\cdot\}_\cM^\zP$on $C^\infty(\zP\sT^*\cM)$ which is closed on the superalgebra
$\A(\zP\sT^*\cM)$ of polynomial functions on $\zP\sT^*\cM$. The latter algebra can be identified with the
Grassmann algebra $\cX(\cM)=\bigoplus_{k=0}^\infty\cX^k(\cM)$ of skew-symmetric multivector fields on $\cM$
with an obvious $\N$-grading. The odd Poisson bracket $\{\cdot,\cdot\}_\cM^\zP$ on $\cX(\cM)$ we will call the
{\it Schouten bracket} on $\cM$, while the bracket $\{\cdot,\cdot\}_\cM$ on $\A(\sT^*\cM)$ we will call the
{\it symmetric Schouten bracket} on $\cM$. The following is well known.
\begin{theorem}\label{derived1}
Any even (resp, odd) Poisson bracket on $C^\infty(\cM)$ is the derived bracket $\{\cdot,\cdot\}_\cJ$  induced
from the symplectic Poisson bracket $\{\cdot,\cdot\}_\cM^\zP$ (resp. $\{\cdot,\cdot\}_\cM$) by a certain even
(resp. odd) Poisson tensor ${\cJ}$ viewed as a quadratic Hamiltonian in $\A(\zP\sT^*\cM)$ (resp.,
$\A(\sT^*\cM)$), so that
\be\label{evenpbr} \{ F,G\}_\cJ=\{\{F,{\cJ}\}_{{\cM}}^{\zP},
G\}_{{\cM}}^{\zP}
\ee
in the even case and
\be\label{oddpbr} \{ F,G\}_\cJ=\{\{F,{\cJ}\}_{{\cM}},
G\}_{{\cM}}
\ee
in the odd case.
\end{theorem}
Note that if $\cM$ is a vector bundle over $M$ with the Euler vector field $\zD$ and the Poisson tensor $\cJ$
on $\cM$ is linear, i.e. $\cJ$ is homogeneous of degree -1, $\Ll_\zD\cJ=-\cJ$, then the Poisson bracket
$\{\cdot,\cdot\}_\cJ$ is closed on linear function on $\cM$, thus defines a Lie bracket $[\cdot,\cdot]_\cJ$ on
sections of the dual bundle $E=\cM^*$ by
\be\label{Lad}\zi_{[e,e']_\cJ}=\{\zi_e,\zi_{e'}\}_\cJ\,.
\ee
Here, $\zi_e$ is the linear function on $\cM=E^*$ corresponding to $e\in\Sec(E)$. This Lie bracket admits an
{\it anchor}, i.e. a vector bundle morphism $\zr:E\to\sT M$ covering the identity such that
\be\label{Lad1}[e,fe']_\cJ=f[e,e']_\cJ+\zr(e)(f)e'
\ee
for any basic function $f\in C^\infty(M)$. In other words, the linear Poisson tensor $\cJ$ makes the vector
bundle $E$ into a {\it Lie algebroid}.

\subsection{Tangent and phase lifts}

In the standard approach, for any supermanifold $\cM$ with local coordinates $(x^a)$ the tangent bundle
$\sT\cM$ with adapted local coordinates $(x^a,\dot{x}^b)$, and the cotangent bundle $\sT^\ast\cM$ with adapted
local coordinates $(x^a,p_b)$ are viewed as supermanifolds with the parity of $\dot{x}^a$ and $p_a$ being the
same as the parity of $x^a$.

We will regard the graded superalgebra (Grassmann algebra) $\zW(\cM)=\bigoplus_{k=0}^\infty\zW^k(\cM)$ of
differential forms on $\cM$ as the superalgebra of polynomial functions on $\Pi\sT\cM$, so that the degree of
local coordinate $x^a\in\zW(\cM)$ is the same as in $\cM$, and the degree of $\dot x^a=\xd {x^a}\in\zW(\cM)$
is $g(x^a)+1\ (mod\ 2)$. This will be the default superalgebra structure on $\zW(\cM)$ we will use.

Note that both  Grassmann algebras,  $\cX(\cM)$ and $\zW(\cM)$ are canonically subalgebras of the algebra
$C^\infty(\sT^\ast\Pi\sT \cM)\simeq C^\infty(\sT^\ast\Pi\sT^* \cM)$. The canonical symplectic form
$\zw_{\Pi\sT^* \cM}$ induces an even Poisson bracket $\BBP{\cdot}{\cdot}_{\Pi\sT^* \cM}$ on
$C^\infty(\sT^\ast\Pi\sT^* \cM)$ called sometimes the {\it big bracket}.

\medskip
If $(\cM,\zD^1_\cM,\dots,\zD^n_\cM)$ is a $\K^n$-graded supermanifold, then $\sT\cM$ is canonically
$\N\ti\K^n$-graded. The corresponding weight vector fields on $\sT\cM$ are: $\zD_{\sT\cM}$, i.e. the Euler
vector field of the vector bundle structure $\zt_\cM:\sT\cM\ra\cM$, and the so called {\it tangent lifts}
$\dt\zD^s_\cM$ of the weight vector fields $\zD^s_\cM$ on $\cM$. In local coordinates,
$$\zD_{\sT\cM}=\dot{x}^a\pa_{\dot{x}^a},\quad \dt\zD^s_\cM=\sum_aw^s_ax^a\pa_{x^a}+
\sum_aw^s_a\dot{x}^a\pa_{\dot{x}^a}\,.$$

This graded manifold we will denote $\sT(1)\cM$ (or $\sT[1]\cM$ if we reverse also the parity of momenta). It
follows that if $\cM$ is an $n$-manifold of degree $r$, then $\sT(1)\cM$ is canonically an $(n+1)$-manifold of
degree $(1,r)$. Indeed, if the weight vector fields $\zD^s_\cM$ are complete and correspond to pair-wise
commuting actions $h^s$ of the monoid $(\R,\cdot)$ on $\cM$ (homogeneity structures), then the vector fields
$\dt\zD^s_\cM$ correspond to pairwise commuting homogeneity structures $\sT h^s$, $(\sT h^s)_t=\sT h^s_t$, on
$\sT\cM$, and the Euler vector field $\zD_{\sT\cM}$ corresponds to the multiplication $h:\R\ti\sT\cM\to\sT\cM$
by scalars in the vector bundle $\sT\cM$, $h_t(v)=t.v$, that commutes with all the actions $\sT h^s$,
$s=1,\dots,n$. As the tangent lift of an Euler vector field is an Euler vector field, so $\sT[1]\cM$ (as well
as $\sT\cM$) is canonically an $(n+1)$-vector bundle if $\cM$ is an $n$-vector bundle (see \cite{GR}).

Similarly, the cotangent bundle $\sT^\ast\cM$ is canonically $\N\ti\K^n$-graded. The corresponding weight
vector fields on $\sT^*\cM$ are: $\zD_{\dts\cM}$ and the so called {\it cotangent lifts} $\dTs\zD^s_\cM$ of
the weight vector fields $\zD^s$ on $\cM$. The cotangent lift $\dTs\zD^s_\cM$ is, by definition, the
Hamiltonian vector field of the linear function on $\sT^\ast\cM$ represented by the vector field $\zD^s_\cM$.
In local Darboux coordinates $(x,p)$,
$$\zD_{\sT^\ast\cM}=p_a\pa_{p_a},\quad \dTs\zD^s_\cM=\sum_aw^s_ax^a\pa_{x^a}-
\sum_aw^s_ap_a\pa_{p_a}\,.$$ The cotangent lift of a h-complete weight vector field is not h-complete any more
as it includes negative weights. If we start with an $\N^n$-graded manifold of degree $r$, a solution depends
on using $r\zD_{\dts\cM}$ instead of $\zD_{\dts\cM}$ and the {\it $r$-phase lifts} $\sT^\ast(r)\zD^s$ of
weight vector fields
\be\label{plift} \dts(r)\zD^s_\cM=\dTs\zD^s_\cM+r\zD_{\sT^\ast\cM}=\sum_aw^s_ax^a\pa_{x^a}+
\sum_a\left(r-w^s_a\right)p_a\pa_{p_a}
\ee
instead of the cotangent lifts. Hence, we end up with an $\N^{n+1}$-graded manifold which we will denote as
$\sT^\ast(r)\cM$, or $\sT^\ast[r]\cM$ if we reverse the parity of momenta for $r$ being odd. If $\cM$ is an
$n$-manifold of degree $r$, these lifts together with the Euler vector field $\zD_{\sT^\ast\cM}$ make
$\sT^\ast\cM$ into an $(n+1)$-manifold of degree $r$. The homogeneity structures $\left(\sT^*(r)h^s\right)$
associated with $\dts(r)\zD^s_\cM$ are defined by
\be\label{hmt}\left(\sT^*(r)h^s\right)_t=t^r.\left(\sT h^s_{t^{-1}}\right)^\ast
\ee
for $t\ne 0$ and extend canonically to the whole $\R$. Note that the canonical symplectic form
$\zw_\cM^{(r)}=\xd p_a\xd x^a$ on $\sT^*(r)\cM$ (or $\zw_\cM^{[r]}=\xd p_a\xd x^a$ on $\sT^*[r]\cM$) has the
weight $r^{n+1}=(r,\dots,r)\in\N^{n+1}$.

The 1-phase lifts, which we will call simply {\it phase lifts}, have  the property that they produce Euler
vector fields from Euler vector fields, so $\dts\cM$ is canonically an $(n+1)$-vector bundle if $\cM$ is an
$n$-vector bundle (see \cite{GR}). In fact, in this case $\dts\cM$ is a {\it symplectic $(n+1)$-vector
bundle}, since the canonical symplectic form $\zw_\cM$ on $\dts\cM$ is homogeneous of degree 1 with respect to
phase lifts and with respect to the Euler vector field $\zD_{\sT^\ast\cM}$.

\section{Graded principal $\R^\ti$-bundles}\label{secprinc}

Let $\R^\ti=\R\setminus\{ 0\}$ be the multiplicative group on non-zero real numbers, viewed as an ordinary
manifold (purely even supermanifold). To every principal $\R^\ti$-bundle $P$ over $\cM$ with the
$\R^\ti$-action
\be\label{action}h:\R^\ti\ti P\ra P\,, \quad (t,p)\mapsto h_t(p)\,,
\ee

they correspond canonical line bundles over $\cM$: the even line bundle $\Pe$ with the typical fibre
$\R=\R^{1\mid 0}$ and the odd line bundle $\Po$ with the typical fibre $\R^{0\mid 1}$. With $P^e$ and $P^o$ we
will denote the dual line bundles, $P^e=(\Pe)^*$, $P^o=(\Po)^*$.

On the other hand, using the transformation rules for local trivializations, with every line bundle $L$ over
$\cM$ one can associate canonically a principal $\R^\ti$-bundle $P=L^\ti$ over $\cM$. These operations on
bundles are mutual inverses: $\stackrel{e}{L^\ti}=L$ and $\stackrel{o}{L^\ti}=L$, for an even and odd bundle
$L$, respectively. If $L$ is even, one can obtain $L^\ti$ just by removing the zero-section:
$L^\ti=L\setminus\{ 0_\cM\}$. Moreover, the fundamental vector field $\zD_P$ of the $\R^\ti$-action $\R^\ti\ti
P\ni(t,p)\mapsto h_t(p)\in P$ on $P=L^\ti$ is just the Euler vector field $\zD_L$ of the vector bundle $L$,
restricted to $L^\ti$. With some abuse of terminology, we will call it the {\it Euler vector field} of the
$\R^\ti$-principal bundle and the $\R^\ti$-action the {\it homogeneity structure} of the $\R^\ti$-principal
bundle. For the standard coordinate $t$ on $\R$, thus $\R^\ti$, used as coordinate in fibers for a fixed local
trivialization and extended by coordinate in the base manifolds (such coordinates in $L^\ti$ we will call {\it
homogeneous}, we have $\zD_L=t\pa_t$.

\begin{definition} An {\it $n$-graded principal $\R^\ti$-bundle of degree $k$} over $\cM$ is a principal $\R^\ti$-bundle $P$ over $\cM$ with a principal action $h^0$ of the group $\R^\ti$ on $P$, equipped simultaneously with a structure of an $n$-graded manifold $(P,h^1,\dots,h^n)$ of degree $k$, such that the action $h^0$ commutes with the homogeneity structures $h^i$,
\be\label{cr}h^0_t\circ h^i_s=h^i_s\circ h^0_t\,,\quad t,s\in\R^\ti\,,\quad i=1,\dots,n\,.
\ee
 In particular, if $h^1,\dots,h^n$ are homogeneity structures of an $n$-vector bundle, an  $n$-graded principal $\R^\ti$-bundle  $(P,h^0,h^1,\dots,h^n)$ of degree $\ao$ we will call an {\it $n$-linear principal $\R^\ti$-bundle}.

Analogous object called simply a {\it principal $\R^\ti$-bundle of degree $k$} we obtain assuming additionally
that the parity is defined by the total weight.
\end{definition}
\begin{remark}
Of course, $h^0_t\circ h^i_s=h^i_s\circ h^0_t$ for $t,s\in\R^\ti$ implies also $h^0_t\circ h^i_0=h^i_0\circ
h^0_t$ and, in turn, that the diffeomorphisms $h^0_t$ preserve the weight vector fields $\zD^1,\dots,\zD^n$
associated with the homogeneity structures $h^1,\dots,h^n$, i.e. $(h^0_t)_*(\zD^i)=\zD^i$.  Hence, the
fundamental vector field $\zD^0$ of the principal $\R^\ti$-action commutes with $\zD^i$, i.e.
$[\zD^0,\zD^i]=0$ for $i=1,\dots,n$ and it is easy to see that they are jointly diagonalizable. On the other
hand, the group $\R^\ti$ is not connected and the fundamental vector field $\zD^0$ determines only the action
of its connected component $\R^\ti_+$. To know the whole action we have to know additionally the symmetry
$h_{-1}$, so (\ref{cr}) means that $[\zD^0,\zD^i]=0$ and $(h_{-1})_*(\zD^i)=\zD^i$. In any case, however, an
$n$-graded principal $\R^\ti$-bundle is canonically an $\Z\ti\N^n$-graded manifold with the homogeneous
algebra
$$\A(P)=\bigoplus_{j\in\Z\ti\N^{n}}\A^j(P)\,,$$
where $j_0$ of $(j_0,\dots,j_n)\in\Z\ti\N^{n}$ refers to the degree of homogeneity with respect to the
$\R^\ti$-action.
\end{remark}

\medskip\noindent A local form of a graded principal $\R^\ti$-bundle describes the following.
\begin{theorem}\label{t2} Any  $n$-graded principal $\R^\ti$-bundle $(P,h^0,h^1,\dots,h^n)$ over $\cM$ induces a canonical $n$-graded manifold structure on $\cM$ and admits an atlas of $\R^\ti$-invariant charts on $P$ with homogeneous local coordinates $(t,x^1,\dots,x^m)$, $t\in\R^\ti$, such that $h^0_s(t,x^a)=(st,x^a)$ and the h-complete weight vector fields read as
\be\label{lvsp} \zD^i=\sum_{a=1}^mw^i_ax^a\pa_{x^a}\,,
\ee
for some $w^s_a\in\N$. In other words, $P$ is locally the product $\R^\ti\ti U$, where $U$ is an open
$n$-graded submanifold of $\cM$, with the obvious $n$-graded principal $\R^\ti$-bundle structure.
\end{theorem}
\begin{proof}
It is straightforward to see that the commutation relations (\ref{cr}) imply that $h^1,\dots,h^n$ induce on
$P/\R^\ti=\cM$ a reduced $n$-manifold structure $\bar{h}1,\dots,\bar{h}^n$. Consider a local trivialization
$\R^\ti\ti U$ of $P$, where $U$ is an $\bar{h}1,\dots,\bar{h}^n$-invariant open subset of $\cM$ and
coordinates $(t,x^a)$, $a=1,\dots,m$. We can choose coordinates $(x^a)$ being homogeneous, so that
\be\label{lvsp=} {\zD}^i=f^i(t,x)\pa_t+\sum_{a=1}^mw^i_ax^a\pa_{x^a}\,.
\ee
As $(h^0_s)_*(\zD^i)=\zD^i$, we get that $sf^i(t/s,x)=f(t,x)$ for all $s,t\in\R^\ti$, so $f^i(t,x)=g^i(x)t$.
Hence, $h^i_s(t,x)=(s^{g^i(x)}t,\bar{h}^i(x))$. Since $h^i$ is a smooth map $\R\ti\R^\ti\ti U
\ni(s,t,x)\mapsto h^i_s(t,x)\in\R^\ti \ti U$, considering the  limit $\lim_{s\to 0}h^i_s(t,x)$ we conclude
that $g(x)=0$.
\end{proof}

\begin{example} Like for vector bundles, the tangent $\sT P$ and the cotangent bundle $\sT^* P$ of an $\R^\ti$-principal bundle $P$ are canonically $\R^\ti$-principal bundles, with the tangent $\sT h$ and the phase $\sT^* h$ lift of the $\R^\ti$-action,
respectively:
\be\label{lifts}(\sT h)_t=\sT(h_t)\,,\quad (\sT^* h)_t=t.(\sT h_{t^{-1}})^*\,,\quad t\ne 0\,.
\ee
In the coordinates $(t,x^a)$ on $P$ as above and the adopted coordinates $(t,x^a,p_0,p_a)$ in $\sT^* P$, we
have $h_s(t,x^a)=(st,x^a)$, so $(\sT^*h)_s(t,x^a,p_0,p_a)=(st,x^a,p_0,sp_a)$ that commutes with the action
$h^1_u(t,x^a,p_0,p_a)=(t,x^a,up_0,up_a)$ of $\R$ by homotheties. We will study closer the case of the
cotangent bundle in Section \ref{sfjb}.

Formally the same formulae hold true when we are lifting a principal $\R^\ti$-bundle structure to the
structures $\zP\sT h$ and $\zP\sT^* h$ on $\zP\sT P$ and $\zP\sT^* P$, respectively. As $\sT\cM$ and
$\sT^*\cM$ are simultaneously vector bundles and the both structures are compatible, $\sT P$ and $\sT^* P$ (as
well as $\zP\sT P$ and $\zP\sT^* P$) are canonically  linear principal $R^\ti$-bundles.
\end{example}
\begin{remark}
We will not especially indicate  the lifting of vector bundle structures and the $\R^\ti$-action to the cotangent bundle and write simply $\sT^* P$ instead writing $\sT^*(1)P$ or similar for an $n$-linear principal $\R^\ti$-bundle $P$, just assuming implicitly that on $\sT^*P$ the phase lifts are taken by default.
\end{remark}

\noindent Suppose now that $(P,h^0,h^1)$ is a 1-graded principal $\R^\ti$-bundle with a principal bundle and a
1-graded manifold fibrations $\zt_0:P\ra P_0$ and $\zt_1=h^1_0:P\ra P_1$, respectively, and with the Euler
vector fields $\zD^0,\zD^1$. According to Theorem \ref{t2}, $P_0=P/\R^\ti$ is canonically a 1-graded manifold
with the  reduced homogeneity structure $\bar{h}^1$ and the bundle structure $ul{\zt}_1:P_0\to\cM$.

Similarly, the $\R^\ti$-action $h^0$ is projectable onto the base $P_1$ of the 1-graded bundle structure and
induces there a principal $\R^\ti$-bundle structure over a base $P_1/\R^\ti$ which can be identified with
$\cM$. Recall that $P_1$ is canonically embedded in $P$ (cf. Remark \ref{r1.2}). Thus the picture is analogous
to the case of a double vector bundle (see \cite{GR})
\be\label{FD0}\xymatrix{
P\ar[rr]^{\zt_0} \ar[d]^{\zt_1} && P_0\ar[d]^{\bar{\zt_1}} \\
P_1\ar[rr]^{\bar{\zt_0}} && \cM }
\ee
except for the fact that one structure is a principal $\R^\ti$-bundle instead of a vector bundle, thus $P_0$
is not canonically embedded in $P$. In this picture however, we have no core, so that the map
\be\label{nocore} \zt=(\zt_1,\zt_0):P\ra P_1\ti_\cM P_0\ee
is a fibration with one-point fibers, thus a diffeomorphism.

Indeed, using the local coordinates $(t,x^a)$ in $P$ described in Theorem \ref{t2} we can write the map $\tau$
as the identity, since coordinates on $P_0$ are $(x^a)$, coordinates on $P_1$ are $(t,x^A)$, and coordinates
on $\cM$ are $(x^A)$, where with $(x^A)$ we denoted coordinates on $P_0$ of degree 0.

Thus we get the following.
\begin{theorem}\label{ii}
Any 1-graded {principal $\R^\ti$-bundle}  $(P,h^0,h^1)$ induces the commutative diagram \ref{FD0} of morphisms
of bundle structures and is therefore equivalent to the product $P_1\ti_{\cM} P_0$ of the principal
$\R^\ti$-bundle $P_1=h^1_0(P)$ over $\cM=P_1/\R^\ti$ and the 1-graded bundle $P_0=P/\R^\ti$ fibred over $\cM$,
with the obvious 1-graded principal $\R^\ti$-bundle structure on the product.
\end{theorem}
\begin{remark}\label{iia}
A particular example of an 1-graded {principal $\R^\ti$-bundle} from the above theorem is a linear {principal
$\R^\ti$-bundle} for which $h^1$ represents homotheties in the vector bundle $\zt_1:P\to P_1$. Here, $P_1$
represents just the 0-section in the vector bundle.
\end{remark}

\section{Contact forms and contact structures}
\subsection{Contact forms}
Note that sections of a vector bundle over $\cM$ form canonically a $C^\infty(\cM)$-bimodule: we usually think
on left-modules, but the right-module structures come automatically according to the standard rules of
super-commutation. Since the left- and the right-module structures are generally different, we encounter some
problems when dealing with contractions. The standard (left) contraction $i_X\zb$ of a vector field $X$ with a
$k$-form $\zb$ is bilinear with respect to the left-module structure on vector fields and the right-module
structure on forms. In particular, for a vector field $X$ and a one-form $\za$ the contractions $i_X\za$ and
$i_\za X$ give, in general, different results. Of course, we can consider as well the right contraction
$i'_X\zb$ which is bilinear with respect to the left-module structure on forms and the right-module structure
on vector fields. The left- and the right-module structures will be indicated by the superscripts `$l$' and
`$r$'. For example, $\zW^1_r(\cM)$ is canonically the dual module of $\cX^1_{\,l}(\cM)$.

Similarly, when speaking about subbundles and quotient bundles as well as direct sums or products we will
usually refer to certain locally free sub-bimodules or quotient bimodules. If a, say, left-submodule is not
generated by even or odd sections, it is, in general, not a bimodule of sections of a super-vector bundle,
even if it is locally free.

Recall that the parity diffeomorphism of $\cM$ which is the identity on even functions and minus the identity
on odd functions can be extended to arbitrary tensors: we will write $\ol{K}$ for the parity diffeomorphisms
applied to a tensor field $K$, i.e. $\ol{K}=K^0-K^1$, where $K^0,K^1$ are the even and the odd part of $K$,
respectively.

Let us observe that if $\za$ is an arbitrary 1-form on $\cM$, not assumed to be even, so it cannot be viewed
as a smooth map $\za:\cM\ra\sT^\*\cM$ between supermanifolds, the pull-back $\za^*(\zb)$ of a differential
form on $\sT^\*\cM$ which is basic or homogeneous of degree 1 is nevertheless well defined. In particular, the
pull back $\za^*(\zs_\cM)$ of the canonical Liouville one-form $\zs_\cM$ on $\sT^\*\cM$ makes sense and
$\za^*(\zs_\cM)=\za$. Similarly, the pull-back of the canonical symplectic form $\zw_\cM=\xd\zs_\cM$ on
$\sT^\ast\cM$ reads $\za^*(\zw_\cM)=\xd\za$. Indeed, in local coordinates $(x^a)$ in $\cM$ and the adapted
Darboux coordinates $(x^a,p_b)$ in $\sT^\ast\cM$ we have $\za=\sum_a\za_a(x)\xd x^a$, so
$$\za^*(\zs_\cM)=\za^*\left(\sum_ap_a\xd x^a\right)=\sum_a\za_a(x)\xd x^a=\za,$$
as $\za^*$ is, by convention, the identity map on basic functions and intertwines the de Rham differential. In
the sequel we will extensively use this observation.

A 1-form $\za\in\zW^1(\cM)$ we will call {\em nowhere-vanishing} if there is a vector field $Y\in\cX^1(\cM)$
such that $i_Y\za=1$. It is easy to see that, equivalently, the one-form $\za$ is nowhere-vanishing if and
only if there is $Y'\in\cX^1(\cM)$ such that $i'_{Y'}\za=i_\za Y'=1$. This implies that $\za$ and $Y$ generate
free 1-dimensional submodules $\la\za\ran$ and $\la Y\ran$ in $\zW^1_l(\cM)$ and $\cX^1_{\,l}(\cM)$,
respectively, and that we have the splittings
$$\cX^1_{\,l}(\cM)=\la Y\ran\oplus \Ker(\za),\quad \zW^1_l(\cM)=\la \za\ran\oplus\Ker(Y')\,,$$
where $\Ker(\za)=\{ X\in\cX^1_{\,l}(\cM):i_X\za=0\}$ and $\Ker(Y')=\{ \zb\in\zW^1_{l}(\cM):i_\zb Y'=0\}$ are
the left kernels of $\za$ and $Y'$, respectively. Of course, the above splittings depend on the choice of $Y$
and $Y'$.

A two form $\zw$ is called {\it non-degenerate} if the contraction $\cX^1_{\,l}\ni X\mapsto
i_x\zw\in\zW^1_l(\cM)$ defines a module isomorphism, and {\it symplectic} if $\zw$ is non-degenerate and
closed. Since on a supermanifold a 1-form need not to be a nilpotent element in the Grassmann algebra of
differential forms and, in general, there is no module of `top forms', it is clear that a contact form $\za$
has to be defined differently that via non-vanishing of $\za(\xd\za)^n$ like in the standard differential
geometry.

\begin{theorem}\label{t10} Let $\za$ be a nowhere-vanishing 1-form
on a supermanifold $\cM$. The following are equivalent.
\begin{description}
\item{(a)} $\la\za\ran$ and $\im(\xd\za)=\{\zb\in\zW^1(\cM):\exists X\in\cX^1(\cM)\ [\zb=i_X\xd\za]\}$ are
complementary submodules of $\zW^1_{l}(\cM)$,
\be\label{decomposition}\zW^1_{l}(\cM)=\la\za\ran\oplus\im(\xd\za)\,,\ee and
the contraction $X\mapsto i_X\xd\za$ constitutes a $C^\infty(\cM)$-module isomorphism between $\Ker(\ol{\za})$
and $\im(\xd\za)$; \item{(b)} $\zW^1_{l}(\cM)=\la\za\ran\oplus\im(\xd\za)$ and
$$\#_\za:\cX^1_{\,l}(\cM)\ra\zW^1_l(\cM),\quad \#_\za(X)=i_X\ol{\za}\cdot\za+i_X\xd\za\,,$$
is a $C^\infty(\cM)$-module isomorphism; \item{(c)} the two-form
\be\label{sf}\zw=\xd(t\cdot\za)=\xd t\cdot\za+t\cdot\xd\za\ee is symplectic on the trivial
$\R^\ti$-principal bundle $\R^\ti\ti\cM$.
\end{description}
\end{theorem}
\begin{pf}
(a)$\Rightarrow$(b) Since $\za$ is nowhere-vanishing, we can find a vector field $Y$ with $i_Y\za=1$. Then,
$i_{\ol{Y}}\ol{\za}=1$, so $\ol{\za}$ is nowhere-vanishing as well. According to the splitting
$\cX^1_{\,l}(\cM)=\la\ol{Y}\ran\oplus\Ker(\ol{\za})$ and the fact that the contraction $X\mapsto i_X\xd\za$
constitutes a $C^\infty(\cM)$-module isomorphism between $\Ker(\ol{\za})$ and $\im(\xd\za)$, the vector field
$\ol{Y}$ can be uniquely chosen from $\Ker(\xd\za)$. We will call it the {\it left Reeb vector field} of
$\za$. The map $\#_\za$ is clearly a $C^\infty(\cM)$-module isomorphism. This is because $\#_\za(X)$ acts on
$\Ker(\ol{\za})$ as an isomorphism on $\im(\xd\za)$, and on $\la\ol{Y}\ran$ as an isomorphism on $\la\za\ran$.

(b)$\Rightarrow$(c) The two-form $$\zw=\xd(t\cdot\za)=\xd t\cdot\za+t\cdot\xd\za,\quad t\ne 0$$ is clearly
closed, so we will show that it is non-degenerate. Any vector field on $\R^\ti\ti\cM$ is of the form
$\wt{X}=f\pa_t+gX$, where $X$ is a vector field on $\cM$ and $f,g$ are smooth function on $\R^\ti\ti\cM$. We
have
\be\label{syf}i_{\wt{X}}\zw=f\cdot\za-g\cdot i_X(\ol{\za})\xd t+t\cdot g\cdot i_X\xd\za\,,\ee so, in view
of the decomposition (\ref{decomposition}), $i_{\wt{x}}\zw=0$ implies $f=0$, $g\cdot i_X(\za)=0$ and $g\cdot
i_X\xd\za=0$, thus $\wt{X}=0$. It remains to show that the contraction with $\zw$ is `onto'. Take
$\wt{\zb}\in\zW^1(\R^\ti\ti\cM)$, $\wt{\zb}=F\xd t+G\zb$, where $\zb$ is a one-form on $\cM$ and $F,G$ are
smooth function on $\R^\ti\ti\cM$. Take $X,\ol{Y}\in\cX^1(\cM)$ such that $\#_\za(X)=\zb$ and
$\#_\za(\ol{Y})=\za$. Then, as easily seen, we get $\wt{\zb}$ by contracting $\zw$ with $(t^{-1}G)\cdot
X+(t^{-1}G-F)\cdot\ol{Y}+(G\cdot i_X\ol{\za})\cdot\pa_t$.

\smallskip\noindent
(c)$\Rightarrow$(a) According to (\ref{syf}), the contraction $i_{\wt{X}}\zw$ is a one-form on $\cM$ (a basic
one-form) if and only if $X\in\Ker(\ol{\za})$ and $f,t^{-1}g$ are smooth function on $\cM$. We can therefore
write such $\wt{X}$ uniquely as $\wt{X}=f(x)\pa_t+t^{-1}X$, where $X\in\Ker(\ol{\za})$. Since the form $\zw$
is non-degenerate, the map $f(x)\pa_t+t^{-1}X\mapsto f(x)\za+i_X\xd\za$ is an isomorphism of the module
$C^\infty(\cM)\oplus \Ker(\ol{\za})$ onto $\zW^1_l(\cM)$. In particular, we have the splitting
(\ref{decomposition}) and a module isomorphism $\Ker(\ol{\za})\ni X\mapsto i_X\za\in\im(\xd\za)$.
\end{pf}

\medskip\noindent
\begin{definition} A nowhere-vanishing one-form $\za$ on a manifold $\cM$ is
called a {\em contact form}, if it satisfies one of the equivalent conditions of Theorem \ref{t10}.
\end{definition}
\begin{example}\label{e1} On $\R^{1\mid 1}$ with the even coordinate $x$ and the odd coordinate
$\zvy$ the one-form $\za=\xd x+\zvy\xd\zvy$ is an even contact form.
\end{example}
\begin{proposition} Given an invertible smooth even function $\zc$ on a manifold $\cM$,
a one-form $\za$ on $\cM$ is a contact form if and only if \ $\zc\cdot\za$ is contact.
\end{proposition}
\begin{pf} According to Theorem \ref{t10}(c), it suffices to show that if $\xd(t\cdot\za)$
is non-degenerate, then $\xd(t\cdot\zc\za)$ is non-degenerate. But $\xd(t\cdot\zc\za)=\xd(t\zc\cdot\za)$ is
the pull-back $\Psi^*(\xd(t\cdot\za))$ of $\xd(t\cdot\za)$ with respect to the diffeomorphism $\Psi:
\R^\ti\ti\cM\ra\R^\ti\ti\cM$, $\Psi(t,x^a)=(t\cdot\zc(x),x^a)$, so $\xd(t\cdot\zc\za)$, exactly like
$\xd(t\cdot\za)$, is non-degenerate.is non-degenerate.
\end{pf}

\noindent We will say that the contact forms $\za$ and $\zc\cdot\za$ are {\it equivalent}.
\begin{example}\label{e4} Note that in the above proposition we cannot take $\zc$ being an arbitrary
invertible function. Indeed, if in the example \ref{e1} we put $\zc=1+\zvy$, then $\zc\cdot\za=(1+\zvy)\xd
x+\zvy\xd\zvy$ is no longer a contact form. The point is that $\zc$ here is not even.
\end{example}

\medskip\noindent
In what follows we will reduce ourselves to the geometric case of contact forms with a given parity. It is
well known that on even manifolds contact forms can be written locally as $\za=\xd z-p_a\xd x^a$. This can be
easily extended to the super-case as follows (see \cite{Sch0,Sch} for the complex setting). The proof refers
to the results describing  the local form of even and odd symplectic structures \cite{Khu, Kos,Sch0,Sha}.
\begin{theorem}\label{t56}
(a) Every even contact form $\za$ on a supermanifold $\cM$ can be locally written as
\be\label{lfe}\za=\xd z -p_a\xd x^a+\frac{\epsilon_j}{2}\zvy^j\xd\zvy^j,\quad\epsilon_j=\pm 1\ee
for certain local coordinates $(z,x^a,p_b,\zvy^j)$ on $\cM$ in which $(z,x^a,p_b)$ are even and $(\zvy^j)$ are
odd.

(b) Every odd contact form $\za$ on a supermanifold $\cM$ can be locally written as
\be\label{lfo}\za=\xd \zx -\zvy^a\xd x^a\ee for certain local coordinates $(x^a,\zvy^b,\zx)$ on $\cM$
in which $(x^a)$ are even and $(\zvy^b,\zx)$ are odd.
\end{theorem}
\begin{pf} It is easy to see that the above one forms are contact forms.
Suppose now that a contact form $\za$ is even (resp., odd) and let $Y$ be the Reeb vector field for $\za$,
i.e. $i_Y\za=1$ and $i_Y\xd\za=0$. In particular, $Y$ is even (resp., odd) nowhere vanishing and involutive
$[Y,Y]=0$ (which is a non-trivial condition in the odd case). This implies (cf. \cite{Bruz,Sha}) that we can
choose local coordinates $(y^0,y^i)$, $i>0$, such that $Y=\pa_0=\frac{\pa}{\pa {y^0}}$ and $\za=\xd y^0-\za'$.
Since $i_{\pa_0}\za'=0$ and $i_{\pa_0}\xd\za'=-i_{\pa_0}\xd\za=0$, we conclude that $\za'$ depends only on
coordinates $y^i$, $i>0$, and, due to properties of the contact form, $\xd\za'$ is a symplectic form in
coordinates $y^i$. Now, we will use the local description of a symplectic form (Darboux Theorem). Since now
the cases split, let us assume that $\za$ is odd (the proof in the even case is similar). We can therefore
assume that $(y^i)=(x^a,\zvy^b)$, $a,b=1,\dots,r$, where $x^a$ are even and $\zvy^a$ are odd, and that
$\xd\za'=\xd\zvy^a\xd x^a$. Since $\xd(\zvy^a\xd x^a)=\xd\za'$, the form $\za''=\za'-\zvy^a\xd x^a$ is closed,
so locally exact, $\za''=\xd f$ for a certain odd function $f$ in coordinates $(x^a,\zvy^b)$. Now, we can
introduce new coordinates $(x^a,\zvy^b,\zx)$, where $\zx=y_0+f(x^a,\zvy^b)$, in which $\za$ takes the form
(\ref{lfo}).
\end{pf}

\noindent Note that the above formulae make sense only for contact forms with a determined parity. In general,
there is much more freedom in choosing a contact form. For instance, if $\za$ is an even contact form and $F$
is any odd function, then $\za+\xd F$ is again a contact form but we cannot built a new coordinate adding $F$
to $z$, as the function $z+F$ has not a fixed parity. In what follows we will work only with even or odd
contact forms.

\bigskip
If $\za$ is a nowhere-vanishing even (resp., odd) one-form on a manifold $\cM$, then it spans a trivial
one-dimensional even vector subbundle $[\za]$ in $\sT^\* \cM$ (resp., in $\zP\sT^\* \cM$). With $\za$ we
associate a canonical smooth embedding $I_\za:\R\ti\cM\ra\sT^\ast\cM$ (resp., $I_\za:\R\ti\cM\ra\zP\sT^\*
\cM$) inducing an isomorphism of $\R\ti\cM$ with $[\za]$. If $(x^a)$ are local coordinates on $\cM$ in which
$\za=f_b(x)\xd x^b$ and $(x^a,p_a)$ are the corresponding Darboux coordinates, then the canonical embedding
$I_\za$ reads
\be\label{emb}I_\za(t,x^a)=(x^a,t\cdot f_b(x))\,.\ee

We can also consider the principal $R^\ti$-bundle $[\za]^\ti$ which is an open submanifold in $[\za]$
described locally by the condition $t\ne 0$. It is straightforward to see that the restriction of the
canonical symplectic form $\zw_\cM$ (resp., $\zw_\cM^\zP$) to $[\za]^\ti$ identified with $\R^\ti\ti\cM$ is
represented by $I^*_\za(\zw_\cM)$  (resp., $I_\za^*(\zw_\cM^\zP)$) and coincides with (\ref{sf}). Thus, in
view of Theorem \ref{t10}, we get the following.
\begin{theorem}\label{t10a} Let $\za$ be a nowhere-vanishing even (resp., odd) 1-form
on a manifold $\cM$. Then, $\za$ is  a contact form if and only if the trivial principal $\R^\ti$-bundle
$[\za]^\ti\simeq\R^\ti\ti\cM$ is, via the embedding $I_\za$, a symplectic submanifold of $\sT^\*\cM$ (resp.,
$\zP\sT^\*\cM$).

In this case the contact form $\za$ can be reconstructed from the symplectic form $\zw=I^*_\za(\zw_\cM)$
(resp., $\zw=I^*_\za(\zw_\cM^\zP)$) thanks to the formula
\be\label{reconstruction}i_{\zD}\zw(t,x)=t\cdot\za(x)\,,\ee
where $\zD=\zD_{[\za]^\ti}=t\pa_t$ is the fundamental vector field of the $\R^\ti$-principal bundle
$[\za]^\ti$. Moreover, this symplectic form is homogeneous with respect to the $\R^\ti$-action, i.e.
\be\label{hmg} (h_t)^*(\zw)=t\zw\,,\quad t\ne 0\,,
\ee
so that
\be\label{reconstruction1}\Ll_{\zD}\zw=\xd(t\cdot\za)=\zw\,.\ee
\end{theorem}
\subsection{Contact structures}
\begin{definition}
An $\R^\ti$-principal bundle $(P,h)$ equipped with a 1-homogeneous symplectic form $\zw$, i.e. a symplectic
form satisfying (\ref{hmg}), we will call a {\it symplectic principal $\R^\ti$-bundle}. More generally, a
principal  $\R^\ti$-bundle $(P,h)$ equipped with a homogeneous Poisson structure $\cJ$ of degree -1, i.e. a
Poisson structure satisfying
\be\label{hmgp} (h_t)_*(\cJ)=-t\cJ\,,\quad t\ne 0\,,
\ee
we will call a {\it principal Poisson $\R^\ti$-bundle}.
\end{definition}
\noindent As a direct consequence of Theorem \ref{t10a} we obtain the following.
\begin{theorem}\label{t01} Let $C$ be an even line subbundle (vector subbundle of rank 1) in $\sT^\ast\cM$ (resp., $\zP\sT^\ast\cM$). The following are equivalent.

\begin{description}
\item{(a)}  $C$ is locally generated by contact (resp., odd contact) one-forms.

\item{(b)}  $C^\ti$ is a symplectic submanifold of $\sT^\ast\cM$ (resp., $\zP\sT^\ast\cM$).
\end{description}
Moreover, any even (resp., odd) symplectic principal $\R^\ti$ bundle over $\cM$ has a canonical symplectic
embedding into $\sT^\ast\cM$ (resp., $\zP\sT^\ast\cM$) as a principal $\R^\ti$-bundle of the form $C^\ti$. In
other words the association $C\mapsto C^\ti$ establishes a one-to-one correspondence between line subbundles
in $\sT^\ast\cM$ (resp., $\zP\sT^\ast\cM$) locally generated by contact forms and even symplectic (resp., odd
symplectic) principal $\R^\ti$-bundles over $\cM$.
\end{theorem}
\begin{proof}
We will work with the even case; the proof in the odd case is parallel. In view of Theorem \ref{t10a}, the
only nontrivial part is to show that any symplectic principal $\R^\ti$-bundle $(P,\zw,h)$ over $\cM$ is of the
form $C^\ti$ for a for $C$ as in (a).

If $\zD$ is the Euler vector field associated with $h$, then the 1-form $\wt{\za}=i_{\zD}\zw$ on $P$ is
semi-basic and defines a map $\Psi:P\ra\sT^\ast\cM$ which is an isomorphism of the principal $\R^\ti$-bundle
$P$ onto $C^\ti\subset\sT^\ast\cM$, where $C$ is the line subbundle in $\sT^\ast\cM$ spanned by the image of
$\wt{\za}$. Indeed, let us write in local bundle coordinates $(t,x)$ in $P$, with $\zD=t\pa_t$, the symplectic
form $\zw$ as $\zw=\xd t\we{\za}+\zw'$ for a certain semi-basic one-form ${\za}$ and a semi-basic two-form
$\zw'$. Since $i_\zD\zw'=0$ and  $i_{\zD}\za=0$, we get $\wt{\za}=t\cdot\za$. As $\zD$ is 0-homogeneous,
$i_{\zD}\zw$ is 1-homogeneous, so $\za$ is homogeneous of degree 0, thus $\R^\ti$-invariant. Being
simultaneously semi-basic it is actually basic, so it can be regarded as a 1-form on $\cM$. Moreover,
$\xd(t\cdot\za)=\zw$, so that $\za$ is a contact form. Hence, $\Psi(t,x)=t\cdot\za(x)$ is a principal bundle
isomorphism which does not depend on the choice of homogeneous coordinates $(t,x)$, since a change in the
choice of the local trivialization of $P$ results in multiplication of $\za$ by an invertible function on
$\cM$ and in multiplication of $t$ by the inverse of this function.
\end{proof}

\begin{definition} A {\it contact structure} (resp., {\it an odd contact structure}) on a
manifold $\cM$ is an even line subbundle $C$ in the cotangent bundle $\sT^\ast\cM$ (resp., $\zP\sT^\ast\cM$)
generated locally by contact forms. Equivalently, such a contact structure can be viewed also as an even
(resp., odd) symplectic principal $\R^\ti$-bundle.
\end{definition}
\begin{remark} The symplectic manifold $(C^\ti,\zw)$ is usually called the {\it
symplectization} of the contact structure $C$ (cf. \cite{Sch0,Sch}). We prefer, finding it much more elegant
and fruitful, to view the symplectization as the contact structure itself and to identify contact structures
with symplectic principal $\R^\ti$-bundles. We will use interchangeably both descriptions of a contact
structure in the sequel.

Note also that the vector bundle $C^\ast$, dual to $C$, is clearly the quotient $\sT\cM/C^0$, where $C^0$ is
the annihilator of $C$, so locally $C^0$ is $\Ker(\za)$ for any contact form $\za$ generating $C$. Hence,
local trivializations of $C^\ast$ are represented by cosets of Reeb vector fields of the corresponding contact
forms. The subbundle $C^0$ in $\sT\cM$ was historically the original concept of a contact structure.
\end{remark}

\section{Graded contact structures}
\subsection{Contact structures of weight $k$}
\begin{definition} Let $\cM$ be a $n$-graded manifold. An even line subbundle $C$ in $\sT^\ast\cM$
(resp., $\zP\sT^\ast\cM$) is a {\it contact} (resp, an {\it odd contact}) {\it structure of weight $k\in\N^n$}
on $\cM$ if it is locally generated by contact one-forms of weight $k$. If $\cM$ is an $n$-vector bundle, a
contact structure of weight $\ao$ on $\cM$ we will call an {\it $n$-linear contact structure} or a {\it
contact $n$-vector bundle}. A $k$-manifold $\cM$ equipped with a contact structure of weight $k$ we will call
a {\it contact $k$-manifol}.
\end{definition}
\begin{remark}\label{r1a} Note that the object in question is well defined, i.e. the weight does not depend on the choice of the
homogeneous contact form. Indeed, if $\za'$ is another homogenous contact form defining the same contact
structure, then $\za'=f\cdot\za$ for certain homogeneous and invertible function, thus function of weight 0.
Note also that an $n$-graded manifold admits a contact structure of weight $k$ only if its degree is $\le k$.
Indeed, let $\za$ be a local contact form generating the structure. Since $\za$ and $i_{\pa_{x^a}}\xd\za$
generate the cotangent bundle, since $w(\pa_{x^a})\le 0$ and $w(\za)=w(\xd\za)=k$, all homogeneous local
coordinates $x^a$ have weight $\le k$.
%In particular, the $k$-phase lift
%$\sT^*(k)\cM=\left(\sT^*\cM,\sT^*(k_1)\bar{\zD}^1,\dots,\sT^*(k_n)\bar{\zD}^n\right)$ is an $\N^n$-graded manifold.

\end{remark}

\begin{definition}\label{d1} An (even or odd) {\it $n$-graded symplectic principal $\R^\ti$-bundle of weight $k\in\N^n$} on $\cM$ is
an $n$-graded principal $\R^\ti$-bundle $(P,h^0,h^1,\dots,h^n)$ over $\cM=P/\R^\ti$, equipped additionally
with an (even or odd, respectively) $\R^\ti$-homogeneous symplectic form $\zw$ of weight $k$,
\be\label{wsf}
(h_t^0)^*\zw=t\zw\,, \quad (h^i_t)^*\zw=t^{k_i}\zw\quad \text{for}\quad t\in\R^\ti\quad i=1,\dots,n\,.
\ee
%Similarly, {\it $n$-graded symplectic  principal $\R^\ti$-bundles of degree $k$ and weight $r$} and {\it  symplectic %$k$-principal $\R^\ti$-bundles of weight $r$},  $(\cM,\zw,h^0,h^1,\dots,h^n)$, are equipped with an $\R^\ti$-homogeneous  %symplectic form $\zw$ of weight $r$,
%\be\label{wsf0}
%(h_t^0)^*\zw=t\zw\quad \text{and} \quad (h^s_t)^*\zw=t^{r_s}\zw\quad \text{for}\quad s=1,\dots,n\quad \text{and}\quad %t\in\R^\ti\,\,.
%\ee
%In particular,  $n$-linear symplectic principal $\R^\ti$-bundles of weight $\ao$ we will call simply {\it $n$-linear %symplectic principal $\R^\ti$-bundles}.
An {\it   $n$-linear (even or odd) symplectic principal $\R^\ti$-bundle} $(P,\zw,h^0,\dots,h^n)$ is an
$n$-linear principal $\R^\ti$-bundle $(P,h^0,\dots,h^n)$ equipped with an (even or odd, respectively)
symplectic form $\zw$ that $\zw$ is 1-homogeneous, $(h^i_t)^*\zw=t\zw$ for $t\ne 0$, with respect to all
actions $h^i$ (in particular, with respect to the Euler vector fields $\zD_0,\dots,\zD_n$).

More generally, a {\it  $n$-linear (even or odd) principal Poisson $\R^\ti$-bundle} $(P,\cJ,h^0,\dots,h^n)$ is
a similar structure with the role of  the homogeneous symplectic form played by a Poisson tensor $\cJ$ on $P$
which is homogenous of degree $-1$ with respect to all homogeneity structures, i.e.
$(h^i_t)_*\cJ=\frac{1}{t}\cJ$ for $t\ne 0$ (in particular, $\Ll_{\zD_i}\cJ=-\cJ$ for the Euler vector fields
$\zD_i$), $i=0,\dots,n$ .
\end{definition}
\begin{example}\label{ex1} Let $(\bar{P},\bar{\cJ},\bar{h}^0)$ be a principal Poisson $\R^\ti$-bundle. It is easy to see that the complete tangent lift $\dt\bar{\cJ}$ of $\bar{\cJ}$ (cf. \cite{GU,GU3}) is homogeneous of bi-degree $(-1,-1)$ on $\sT \bar{P}$. In this way we obtain a  linear principal Poisson $\R^\ti$-bundle $(P,\cJ,h^0,h^1)$, the {\it complete tangent lift} of $(\bar{P},\bar{\cJ},\bar{h}^0)$, where $P=\sT \bar{P}$, $\cJ=\dt\bar{\cJ}$, $h^0=\sT\bar{h}^0$, and $h^1$ representing the canonical vector bundle structure $\sT \bar{P}\to \bar{P}$.
\end{example}

\begin{theorem}\label{t01a}
The correspondence $C\mapsto C^\ti$ establishes a one-to-one correspondence between even (resp., odd) contact
structures of weight $k\in\N^n$ on an $n$-graded manifold $\cM$ and even  (resp., odd) {$n$-graded  symplectic
principal $\R^\ti$-bundles of weight $k\in\N^n$} over $\cM$. In other words, contact structures of weight
$k\in\N^n$ can be viewed as  {$n$-graded  symplectic principal $\R^\ti$-bundles of weight $k$}.
\end{theorem}
\begin{proof} As the proof is completely analogous in the odd case, to fix our attention let us assume that the case is
even. Let $C\subset\sT^\ast\cM$ be a contact structure of weight $k\in\N^n$ on an $n$-graded manifold
$(\cM,\bar{h}^1,\dots,\bar{h}^{n})$. On the cotangent bundle $\sT^\ast\cM$ consider the phase lifts
$h^i=\sT^\ast(k_i)\bar{h}^i$, $i=1,\dots,n$, which induce an $n$-graded manifold structure on $\sT^\ast\cM$
with weight vector fields $\zD^1,\dots,\zD^n$ (cf. Remark \ref{r1a}). It is easy to see that if $\za$ is a
1-form homogeneous of weight $k$ on $\cM$, then the vector fields $\zD^i$, $i=1,\dots,n$, are tangent to the
one-dimensional subbundle $[\za]\subset\sT^\ast\cM$ spanned by $\za$. This implies that if
$C\subset\sT^\ast\cM$ is a contact structure of weight $k$, then the vector fields $\zD^i$ induce on $C^\ti$ a
structure of an $n$-graded principal $\R^\ti$ bundle. Moreover, since the canonical symplectic form
$\zw_{\sT^*\cM}$ is 1-homogeneous with respect to the canonical homogeneity structure $h^0$ in the vector
bundle $\sT^\ast\cM$ and of weight $k$ with respect to $\zD^1,\dots,\zD^n$, its restriction $\zw$ to $C^\ti$
is $\R^\ti$-invariant and homogeneous of weight $k$.

Conversely, if $(P,h^0,h^1,\dots,h^n)$ is an $n$-graded symplectic principal $\R^\ti$-bundle over $\cM$ of
weight $k$, then $h^i$ commuting with $h^0$ projects to homogeneity structures $\bar{h}^i$, $i=1,\dots,n$, on
$\cM$ that turns $\cM$ into an $n$-graded manifold (Theorem \ref{t2}). Moreover, if $\zD^0$ is the Euler
vector field associated with the $\R^\ti$-action $h^0$, then the 1-form $\wt{\za}=i_{\zD^0}\zw$ on $P$ defines
a map $\Psi:P\ra\sT^\ast\cM$ which is a symplectic isomorphism of the principal $\R^\ti$-bundle $P$ onto
$C^\ti\subset\sT^\ast\cM$, where $C$ is the line subbundle in $\sT^\ast\cM$ spanned by the image of $\wt{\za}$
(see the proof of Theorem \ref{t01}). In the local coordinates $(t,x)$ described in Theorem \ref{t2},
$\wt{\za}(t,x)=t\cdot\za(x)$ for a contact form $\za$ and the image of $\Psi$ is locally generated by $\za$.
As $\zw$ is of weight $k$, $t\cdot\za=i_{t\pa_t}\zw$, and $t$ is of weight 0, the contact form is of weight
$k$.
\end{proof}

\subsection{Canonical contact structures}\label{sec1}
To present some examples of canonical contact structures, let us start with an even line bundle $\zt:L\ra\cM$
over $\cM$ (for an odd line bundle the construction is similar). We have therefore an open covering of $\cM$
by coordinate charts $U_i$ with coordinates $(x_i^a)$ and an open covering of $L$ by coordinate charts
$V_i=\zt^{-1}(U_i)$, diffeomorphic to $\R\ti U_i$, with coordinates $(z_i,x^a_i)$, where $z_i$ is even and
varying through all $\R$. The change of coordinates takes the form
\be\label{cchange}x_j=\zf_j^i(x_i),\quad z_j=\zc_j^i(x_i)\cdot z_i\,,\ee
for certain diffeomorphism $\zf_j^i$ and an even smooth nowhere-vanishing function $\zc_j^i$.

The {\it first jet bundle} $\sJ^1L$ is a vector bundle over $\cM$ of first jets of local sections of $L$
defined as follows. In a given trivialization $V_i\simeq\R\ti U_i$ any section $z_i=f(x_i)$ is represented by
the function $f$ and $\sJ^1L$ over $V_i$ is diffeomorphic to $\R\ti\sT^\ast V_i$ with coordinates
$(z_i,x_i^a,\mathfrak{p}_b^i)$. The coordinate change in $L$ results in the coordinate change in $\sJ^1L$ of
the form
\be\label{cchange1}
x_j=\zf_j^i(x_i),\quad z_j=\zc_j^i(x_i)\cdot z_i,\quad \mathfrak{p}^j=\left(\zc_j^i(x_i)\cdot \mathfrak{p}^i+
\frac{\pa\zc_j^i}{\pa x_i}(x_i)\cdot z_i\right)\frac{\pa x_i}{\pa x_j}\,.
\ee
The first jet bundle $\sJ^1L$ is a vector bundle with local linear coordinates $z_i$ and $\mathfrak{p}^i_a$.
Every section $f$ of $L$ induces a section $j^1(f)$ of this bundle, its {\it first jet prolongation}, given
locally by $z_i=f(x_i)$, $\mathfrak{p}^i_a=\frac{\pa f}{\pa x^a_i}(x_i)$. A similar construction of $\sJ^1L$
can be done when starting with an odd line bundle $L$.

The first jet bundle $\sJ^1L$ is a tool for representing first-order differential operators on $L$ by vector
bundle morphisms. Indeed, for any vector bundle $E$ over $\cM$, vector bundle morphisms
$D\in\operatorname{Hom}_\cM(\sJ^1L,E)$ covering the identity on $\cM$ represent first-order differential
operators $\wt{D}\in\Sec(\DO^1(L,E))$ in the obvious way: $\wt{D}(f)=D(j^1(f))$. We will often use the
identification
$$\operatorname{Hom}_\cM(\sJ^1L,E)\simeq (\sJ^1L)^\ast\ot_\cM
E\simeq\DO^1(L,E)\,.
$$
In particular, $\DO^1(L,L)=(\sJ^1L)^\ast\ot_\cM L$.

There is an even line subbundle $C_L$ in $\sT^\ast \sJ^1L$ whose sections $\zb$ have the property that the
pull-back $j^1(f)^*(\zb)$ is 0 for all $f\in\Sec(L)$. In local coordinates these sections have the form
$\zc(x_i)\cdot(\xd z_i-\mathfrak{p}_a^i\xd x^a_i)$, so the line bundle $C_L$ is generated locally by contact
forms $\xd z_i-\mathfrak{p}_a^i\xd x^a_i$, whence local coordinates $(t_i,z_i,x^a_i,\mathfrak{p}_a^i)$ on
$C_L$. It follows that $C_L$ is an even contact structure locally of the form (\ref{lfe}) with no $\zvy^i$
present.

Since $\sJ^1L$ is a vector bundle over $\cM$, the cotangent bundle $\sT^*\sJ^1L$ is a double vector bundle.
The non-obvious vector bundle structure is associated with a canonical fibration $\zt_1:\sT^*\sJ^1L\to
(\sJ^1L)^*$. The latter can be reduced to a vector bundle structure
\be\label{lcst}\zt_1:C_L^\ti\to (L^*)^\ti\,,\quad (t_i,z_i,x^a_i,\mathfrak{p}_a^i)\mapsto (t_i,x^a_i)\,,\ t\ne 0\,,
\ee
which makes $C_L^\ti$ into a linear contact structure, as $\za_i=\xd z_i-\mathfrak{p}_a^i\xd x^a_i$ is linear.
With $\zP C_L^\ti$ we will denote the parity shift associated with this vector bundle structure.

The local maps $(z_i,x_i^a,\mathfrak{p}_b^i)\mapsto (z_i,x_i^a)$ give rise to a vector bundle morphism
$\sJ^1L\ra L$ whose kernel has local coordinates $(x_i^a,\mathfrak{p}_b^i)$ with transformation rules
$$x_j=\zf_j^i(x_i),\quad \mathfrak{p}^j=\zc_j^i(x_i)\cdot \mathfrak{p}^i\cdot\frac{\pa x_i}{\pa x_j}\,,
$$
so it can be identified with $\sT^\ast\cM\ot_\cM L$. Thus, we get the {\it jet bundle exact sequences} of
vector bundle morphisms:
\be\label{jbes}0\ra\sT^\ast\cM\ot_\cM L\ra \sJ^1L\ra L\ra 0\ee
if $L$ is even and
\be\label{jbesodd}0\ra\zP\sT^\ast\cM\ot_\cM  L\ra \sJ^1L\ra L\ra 0\ee
\begin{example} {\bf (The case of the trivial line bundle)}

\smallskip\noindent
Let $L=\R^{1|0}\ti\cM$ be the trivial even line bundle over $\cM$. Then, $\sJ^1L=\R^{1|0}\ti\sT^*\cM$ with
local coordinates $(z,x^a,p_a)$and the {jet bundle exact sequence} of vector bundle morphisms reads as
$$0\ra\sT^\ast\cM\ra \R^{1|0}\ti\sT^*\cM\ra \R^{1|0}\ti\cM\ra 0\,,$$
and the contact structure $C_L$ is a trivial line subbundle in $\sT^*\left(\R^{1|0}\ti\sT^*\cM\right)$
generated by the contact form $\za=\xd z-p_a\xd x^a$. Thus, $C_L^\ti=\R^\ti\ti\R^{1|0}\ti\sT^*\cM$ with local
coordinates $(t,z,x^a,p_a)$, $t\ne 0$, with the symplectic form $\zw=\xd t\xd z-p_a\xd t\xd x^a-t\xd p_a\xd
x^a$ which is linear with respect to the vector bundle structure
$C_L^\ti\ni(t,z,x^a,p_a)\mapsto(t,x^a)\in\R^\ti\ti\cM$.

If $L=\R^{0|1}\ti\cM$ is the trivial odd line bundle over $\cM$, then $\sJ^1L=\R^{0|1}\ti\zP\sT^*\cM$ with
local coordinates $(z,x^a,\bar{p}_a)$ ($\bar{p}_a$ has the parity opposite to $p_a$) and the {jet bundle exact
sequence} of vector bundle morphisms reads as
$$0\ra\zP\sT^\ast\cM\ra \R^{0|1}\ti\zP\sT^*\cM\ra \R^{0|1}\ti\cM\ra 0\,,$$
and the contact structure $C_L$ is a trivial even line subbundle in
$\zP\sT^*\left(\R^{1|0}\ti\zP\sT^*\cM\right)$ generated by the contact form $\za=\xd z-\bar{p}_a\xd x^a$.
Since in this case $z$ is odd and $x^a$ and $\bar{p}_a$ have opposite parities, the contact form is odd. Thus,
$C_L^\ti=\R^\ti\ti\R^{0|1}\ti\zP\sT^*\cM$ with local coordinates $(t,z,x^a,\bar{p}_a)$, $t\ne 0$, with the
symplectic form $\zw=\xd t\xd z-\bar{p}_a\xd t\xd x^a-t\xd \bar{p}_a\xd x^a$ which is odd and linear with
respect to the vector bundle structure $C_L^\ti\ni(t,z,x^a,\bar{p}_a)\mapsto(t,x^a)\in\R^\ti\ti\cM$.

\end{example}

\begin{definition} The even (resp., odd) contact structure $C_L$ on $\sJ^1L$ associated with an even (resp., odd)
line bundle $L$ we will call {\it canonical}. The corresponding symplectic principal $\R^\ti$-bundle $C^\ti_L$
will also be called, with some abuse of terminology, a {\it canonical contact structure}.
\end{definition}

\section{First jet bundles revisited}\label{sfjb}
Let $L$ be an even line bundle over $\cM$. Note that the canonical contact structure $C_L$ is a {double vector
bundle}. Indeed, as $\sJ^1L$ is a vector bundle, $\sT^\ast \sJ^1L$ is a double vector bundle and $C_L$ is a
double vector subbundle of the latter. This is because the double vector bundle structure in $\sT^\ast \sJ^1L$
is determined by the commuting Euler vector fields: the Euler vector field $\wt{\zD}_0=\zD_{\sT^\ast \sJ^1L}$
of the vector bundle structure over $\sJ^1L$ and the Euler vector field $\wt{\zD}_1=\sT^\ast\zD_{\sJ^1L}$,
i.e. the phase lift of the Euler vector field $\zD_{\sJ^1L}$ of the vector bundle structure on $\sJ^1L$, and
both Euler vector fields are tangent to $C_L$. In local affine coordinates,
$$\wt{\zD}_0=\zD_{\sT^\ast \sJ^1L}=t\pa_t,\quad
\wt{\zD}_1=\sT^\ast\zD_{\sJ^1L}=\mathfrak{p}_a\pa_{\mathfrak{p}_a}+z\pa_z\,,
$$
and the corresponding diagram of the double vector bundle structure reads
\be\label{DB}\xymatrix{
C_L\ar[rr]^{\zt_0} \ar[d]^{\zt_1} && \sJ^1L\ar[d]^{\bar{\zt_1}} \\
{L^\ast}\ar[rr]^{\bar{\zt_0}} && \cM }
\ee
An important observation is that section of the line bundle $L$ are identified with linear functions on
$L^\ast$, i.e., in turn, with $(1,0)$-homogeneous functions with respect to the pair of the Euler vector
fields $(\wt{\zD_0},\wt{\zD_1})$. The Euler vector fields $\wt{\zD}_i$, $i=0,1$, are tangent also to $C_L^\ti$
which becomes a linear $\R^\ti$-principal bundle $(C_L^\ti,\wt{h}^0,\wt{h}^1)$. The corresponding projections
are ${\zt}_0:C_L^\ti\ra \sJ^1L$ and ${\zt}_1:C_L^\ti\ra (L^\ast)^\ti$, as the $\zD_1$-invariant functions can
be identified with functions on $(L^\ast)^\ti$. To see the latter we will use another double bundle, namely
$\sT^\ast(L^\ast)^\ti$. One bundle structure is clearly the vector bundle structure over $(L^\ast)^\ti$
associated with the Euler vector field $\zD_1$ or the corresponding homogeneity structure $h^1$. The other is
associated with the phase lift ${h}^0=\sT^\ast h$ of the natural principal $\R^\ti$-action $h$ on
$(L^\ast)^\ti$. Let $\zD_0$ be the Euler vector field associated with $\wt{h}^0$. In local affine coordinates
$(t,x^a)$, $t\ne 0$, on $(L^\ast)^\ti$ and the corresponding Darboux coordinates $(t,x^a,z,p_a)$ on
$\sT^\ast(L^\ast)^\ti$,
$$\zD_1=z\pa_z+p_a\pa_{p_a},\quad \zD_0=t\pa_t+p_a\pa_{p_a}\,,$$
so that $h^1_u$ is the multiplication by $u$ in the coordinates $(z,p_a)$ and $h^0_u$ is the multiplication by
$u$ in coordinates $(t,p_a)$. It is easy to see that the orbit space of the principal $\R^\ti$ action ${h}^0$
on $\sT^\ast(L^\ast)^\ti$ can be canonically identified with $\sJ^1L$.

Of course, we can start with an arbitrary principal $\R^\ti$-bundle $\bar{\zt}:P\ra \cM$ with the homogeneity
structure $h$ instead of $(L^\ast)^\ti\ra\cM$ and considering the even and odd line bundles $P^{1\mid 0}$ and
$P^{0\mid 1}$ associated with $P$ we get the following.
\begin{theorem}\label{plift0} The phase lift ${h}^0=\sT^\ast h$ of the homogeneity structure of a principal
$\R^\ti$-bundle $\bar{\zt}:P\ra \cM$ defines canonical principal $\R^\ti$-bundle structures on the fibrations
${\zt}:\sT^\ast P\ra \sJ^1P^e$ and ${\zt}^\zP:\zP\sT^\ast P\ra \sJ^1P^o$.
\end{theorem}

The canonical symplectic structures: $\zw=\xd\left(t\cdot(\xd z-\mathfrak{p}_a\xd x^a)\right)$ on $C_L^\ti$
and $\zw_{(L^\ast)^\ti}=\xd z\xd t+\xd p_a\xd x^a$ on $\sT^\ast(L^\ast)^\ti$ are bi-homogeneous,
$(\wt{h}^i_u)_*\zw=u\zw$, $({h}^i_u)_*\zw_{(L^\ast)^\ti}=u\zw_{(L^\ast)^\ti}$, $i=0,1$. Recall that we refer
to such structures as to {linear symplectic principal $\R^\ti$-bundle structures} or linear contact
structures.

It is well known that linear symplectic structures are exactly those of cotangent bundles. Here, the map
$\Psi:C_L^\ti\ra\sT^\ast(L^\ast)^\ti$, written in homogeneous coordinates $(t,z,x^a,\mathfrak{p}_b)$, $t\ne
0$, in $C^\ti_L$ and Darboux coordinates $(t,x^a,z,p_b)$ in $\sT^\ast (L^\ast)^\ti$ as
$$(-t,z,x^a,\mathfrak{p}_b)\mapsto(t,x^a,z,t\mathfrak{p}_b)\,,$$
is a symplectomorphism identifying the canonical symplectic structures and relating the double bundle
structures: $(\Psi)_*(\wt{\zD}_i)=\zD_i$, $i=0,1$. In particular, we have the commutative diagram of the
corresponding fibrations
\be\label{FD}\xymatrix{
C_L^\ti\simeq\sT^\ast (L^\ast)^\ti\ar[rr]^{\zt_0} \ar[d]^{\zt_1} && \sJ^1L\ar[d]^{\bar{\zt_1}} \\
{(L^\ast)^\ti}\ar[rr]^{\bar{\zt_0}} && \cM }
\ee
together with the canonical embeddings $i_2:\sJ^1L\ra C_L^\ti\simeq\sT^\ast (L^\ast)^\ti$ and
$i_1:{(L^\ast)^\ti}\ra C_L^\ti\simeq\sT^\ast (L^\ast)^\ti$, so we can view $\zt_i$ as a projection onto the
submanifold and $\bar{\zt_i}$ as the restriction of $\zt_i$ to the submanifold, $i=0,1$. In the sequel we will
usually identify the above linear symplectic principal $\R^\ti$-bundles.

\medskip
It is clear that we could start as well with an odd line bundle $L$, constructing similarly $\sJ^1L$ and the
odd contact structure $C_L$ on $\sJ^1L$ as a line subbundle in $\zP\sT^\ast \sJ^1L$. In this case, instead of
the canonical Poisson bracket $\{\cdot,\cdot\}_\zw$ associated with the symplectic form
$\zw=\zw_{(L^\ast)^\ti}$ on $\sT^\ast(L^\ast)^\ti\simeq C_L^\ti$, we get the canonical Poisson bracket
$\{\cdot,\cdot\}_\zw^\zP$ associated with the symplectic form $\zw^\zP=\zw^\zP_{(L^\ast)^\ti}$ on
$\zP\sT^\ast(L^\ast)^\ti\simeq C_L^\ti$. If local coordinates are concerned, now $z$ is odd and the parity of
$p_a$ is the opposite to the parity of $x^a$. The diagram (\ref{FD}) takes now the form
\be\label{FD1}\xymatrix{ C_{L}^\ti\simeq\zP\sT^\ast (L^\ast)^\ti\ar[rr]^{\zt_0} \ar[d]^{\zt_1} &&
 \sJ^1L\ar[d]^{\bar{\zt_1}} \\
{(L^\ast)^\ti}\ar[rr]^{\bar{\zt_0}} && \cM }
\ee
Thus, we get the following.
\begin{theorem} Let $L$ be an even (resp., odd) line bundle over $\cM$, let $L^*$ be its dual, and let $(L^\ast)^\ti$
be the corresponding $\R^\ti$-principal bundle with the $\R^\ti$-action $h$. The canonical contact structure
$(C_L^\ti,\zw,\wt{h}^0,\wt{h}^1)$ is canonically isomorphic to $(\sT^\ast
(L^\ast)^\ti,\zw_{(L^\ast)^\ti},h^0,h^1)$ (resp., to $(\zP\sT^\ast
(L^\ast)^\ti,\zw_{(L^\ast)^\ti}^\zP,h^0,h^1)$), where $h^0=\sT^*h$. In simple words, canonical even (resp.,
odd) contact structures are exactly the (parity shifts of) the cotangent bundles of principal
$\R^\ti$-bundles.
\end{theorem}
Note also that if we start with a principal $\R^\ti$-bundle $P$ over $\cM$, then the diagrams (\ref{FD}) and
(\ref{FD1}) take the form
\be\label{FDa}\xymatrix{
\sT^\ast P\ar[rr]^{\zt_0} \ar[d]^{\zt_1} && \sJ^1P^e\ar[d]^{\bar{\zt_1}} \\
{P}\ar[rr]^{\bar{\zt_0}} && \cM }\qquad \xymatrix{
\zP\sT^\ast P\ar[rr]^{\zt_0} \ar[d]^{\zt_1} && \sJ^1P^o\ar[d]^{\bar{\zt_1}} \\
{P}\ar[rr]^{\bar{\zt_0}} && \cM }
\ee

\medskip\noindent  Note that with the above notation we have $\zP C_L^\ti=\zP\sT^\ast(L^\ast)^\ti$ if the line bundle $L$ is even, and $\zP C_L^\ti=\sT^\ast(L^\ast)^\ti$ if $L$ is odd. In other words, we can write $\zP C_L=C_{\zP L}$. Under this parity shift the canonical symplectic form $\zw$ on $C_L^\ti$ goes into the canonical symplectic form $\zw^\zP$ on $\zP C_L^\ti$ which is of the opposite parity. Note however that $(L^\ast)^\ti=(\zP L^\ast)^\ti$, so the principal $\R^\ti$-bundle being the base of the vector fibration remains unchanged.
%The Poisson bracket $\{\cdot,\cdot\}_{\zw^\zP}$
%associated with the symplectic form on $\zP C_L^\ti$ we will call the {\it Schouten bracket on $C_L^\ti$}.

\section{Brackets}
Before considering arbitrary contact structures, let us consider the trivial case of a contact form which,
however, can serve as a local model in general.

Let $\za$ be an even or odd contact one form on $\cM$ and let $\zw=\xd(t\cdot\za)$ be the corresponding
symplectic form on $\wt{\cM}=\R^\ti\ti\cM$. Let $\{\cdot,\cdot\}_\zw$ be the Poisson bracket associated with
the symplectic form $\zw$. Since $\zw$ is homogeneous of degree 1 with respect to the $\R^\ti$-action on
$\R^\ti\ti\cM$ (thus with respect to the fundamental vector field $\zD=t\pa_t$ of the action), the bracket is
homogeneous of degree -1, so for any $F,G\in C^\infty(\cM)$, we have
\be\label{PJ}\{ tF,tG\}_\zw=t\{ F,G\}_\za
\ee
for certain $\{ F,G\}_\za\in C^\infty(\cM)$. Here, of course, $(tF)(t,x)=tF(x)$. The bilinear operation on
$C^\infty(\cM)$, $(F,G)\mapsto\{ F,G\}_\za$, we call the {\it Legendre bracket} associated with the contact
form $\za$. The Legendre bracket is a particular example of a {\it Jacobi bracket} (cf. \cite{GM2}).

It is easy to see that the equation (\ref{PJ}) can be also taken as the definition of $\{\cdot,\cdot\}_\zw$
when $\{\cdot,\cdot\}_\za$ is given. This can be generalized to an arbitrary Jacobi bracket as follows (cf.
\cite{DLM,GIMPU}).
\begin{theorem}\label{tfun} There is a one-to-one correspondence between Jacobi brackets $\{\cdot,\cdot\}$ on $C^\infty(\cM)$
and homogeneous Poisson brackets $\{\cdot,\cdot\}_\cJ$ on $\wt{\cM}=\R^\ti\ti\cM$ of  degree -1, defined via
the identity
\be\label{PJ1}\{tF,tG\}_\cJ=t\{ F,G\}
\ee
valid for any $F,G\in C^\infty(\cM)$.
\end{theorem}
The bracket $\{\cdot,\cdot\}_\cJ$ is usually called the {\it poissonization} of $\{\cdot,\cdot\}$. Under the
poissonization, the tensor (\ref{jac}) on $\cM$ goes into the bivector field
\be\label{jac1} {\cJ}=\frac{1}{t}\zL+\zG\cdot\pa_t+tf\pa_t\cdot\pa_t
\ee
on $\wt{\cM}$ which is homogeneous of degree -1. This bivector field is an even (resp., odd) {\it Poisson
tensor} in the sense that $\{{\cJ},{\cJ}\}^{\zP}_{\wt{\cM}}=0$ in the even case and
$\{{\cJ},{\cJ}\}_{\wt{\cM}}=0$ in the odd case. The even (resp., odd) Poisson tensor $\cJ$ represents a
quadratics function on $\zP\sT^\ast\wt{\cM}$ (resp., $\sT^\ast\wt{\cM}$). Note that in the even case
$\pa_t\cdot\pa_t$ is automatically 0. The properties of even and odd Poisson bracket immediately imply the
following.

\begin{theorem} The tensor $\zL+\zG\cdot I$, where $\zL$ and $\zG$ are even elements from $\cX(\cM)$ of degrees,
respectively, 2 and 1, defines an even Jacobi structure  on $\cM$ if and only if
\be\label {cond1}\{\zL,\zL\}^{\zP}_{{\cM}}=2\zG\cdot\zL\,,\quad
\{\zG,\zL\}^{\zP}_{{\cM}}=0\,.
\ee
The tensor $\zL+\zG\cdot I+fI\cdot I$, where $\zL$, $\zG$, and $f$ are odd elements from $\A(\sT^*\cM)$ of
degrees, respectively, 2, 1, and 0, defines an odd Jacobi structure  on $\cM$ if and only if
\be\label {cond2}
\{\zL,\zL\}_{{\cM}}=2\zG\cdot\zL\,,\  \{\zG,\zL\}_{{\cM}}=2f\zL\,, \ \{\zG,\zG\}_{{\cM}}=2(f\zG-\{
f,\zL\}_{{\cM}})\,,\ \zG(f)=0 \,.
\ee
Here $\{{\cdot},{\cdot}\}^{\zP}_{{\cM}}$  (resp., $\{{\cdot},{\cdot}\}_{{\cM}}$) is the Schouten (resp.,
symmetric Schouten) bracket on $\cM$.
\end{theorem}

\medskip\noindent
Note that the conditions (\ref{cond2}) reduce to that considered in \cite{Bru} for $f=0$.
\begin{example} The triple
$$\left(\zL=\zvy\cdot\pa_x\cdot\pa_x\,,\zG=\zvy\cdot\pa_x\,,f=\zvy\right)
$$
on $\R^{1\mid 1}$ (with the even coordinate $x$ and the odd coordinate $\zvy$) represents an odd Jacobi
structure.
\end{example}

We call a Jacobi bracket {\it non-degenerate} if its poissonization is symplectic, so that the bivector field
${\cJ}$ is invertible. As the poissonization is homogeneous of degree -1 with respect to $\zD=t\pa_t$, the
corresponding symplectic structure $\zw$ is 1-homogeneous and makes $\wt{\cM}$ into a symplectic
$\R^\ti$-principal bundle. Thus $\zw=\xd(t\cdot\za)$ for a contact form $\za$ and the Jacobi bracket is the
Legendre bracket associated with this contact form. In other words, non-degenerate Jacobi brackets are
Legendre brackets.

\bigskip
The Poisson bracket $\{\cdot,\cdot\}_\zw$ associated with the symplectic form $\zw=\xd(t\za)$, $t\ne 0$, where
$\za$ is the even contact form (\ref{lfe}), is determined by the following commutation rules of coordinate
functions (lacking commutations rules not following from the antisymmetry are, by convention, trivially 0):
\be\label{lfe1}\{
z,t\}_\zw=1,\quad\{p_a,z\}_\zw=\frac{p_a}{t},\quad\{\zvy^j,z\}_\zw=\frac{\zvy^j}{t},
\quad\{p_a,x^a\}_\zw=\frac{1}{t},\quad\{\zvy^j,\zvy^j\}_\zw=\frac{\epsilon_j}{t}\,,\ee so they are associated
with the Poisson tensor ${\cJ}$,
\be\label{lfe2}{\cJ}=\pa_t\pa_z+\frac{1}{t}\left[\pa_z\left(p_a\pa_{p_a}+\zvy^j\pa_{\zvy^j}\right)+
\pa_{x^a}\pa_{p_a}-\frac{\epsilon_j}{2}\pa_{\zvy^j}\pa_{\zvy^j}\right]\,,
\ee
viewed as a quadratic function on $\zP\sT^\ast\wt{\cM}$, as the derived bracket of the Schouten bracket
$\{\cdot,\cdot\}_{\wt{\cM}}^{\zP}$ on $\zP\sT^\ast\wt{\cM}$,
\be\label{lfe11}\{\zf,\zc\}_\zw=\{\{\zf,{\cJ}\}_{\wt{\cM}}^{\zP},\zc\}_{\wt{\cM}}^{\zP}\,.\ee
Note that the bracket is even so, for homogeneous $\zf,\zc$, we have
$$\{\zf,\zc\}_\zw=-(-1)^{g(\zf)g(\zc)}\{\zc,\zf\}_\zw\,,$$ and the corresponding Legendre
bracket is an even (graded) Jacobi bracket and reads
\bea\nn \{ F,G\}_\za &=&\frac{\pa F}{\pa{p_a}}p_a\frac{\pa G}{\pa z}-\frac{\pa F}{\pa z}p_a\frac{\pa G}{\pa{p_a}}
+ \frac{\pa F}{\pa{\zvy^j}}\zvy^j\frac{\pa G}{\pa z}-\frac{\pa F}{\pa z}\zvy^j\frac{\pa G}{\pa{\zvy^j}}\\
&&+\epsilon_j\frac{\pa F}{\pa \zvy^j}\frac{\pa G}{\pa \zvy^j}+\frac{\pa F}{\pa{p_a}}\frac{\pa G}{\pa
x^a}-\frac{\pa F}{\pa x^a}\frac{\pa G}{\pa{p_a}} +\frac{\pa F}{\pa z}\cdot G -F\cdot\frac{\pa G}{\pa
z}\,.\label{lfe3}
\eea

\medskip
For the odd contact form (\ref{lfo1}), the symplectic Poisson bracket is determined by
\be\label{lfo1}\{ t,\zx\}_\zw=1,\quad\{\zvy^a,\zx\}_\zw=-\frac{\zvy^a}{t},
\quad\{x^a,\zvy^a\}_\zw=-\frac{1}{t}\,.\ee It is associated with the Poisson tensor
\be\label{lfeo2}{\cJ}=-\pa_t\pa_\zx+\frac{1}{t}\left[\pa_{x^a}\pa_{\zvy^a}-
\pa_{\zvy^a}\zvy^a\pa_{\zx}\right]\,,\ee viewed as a quadratic function on $\sT^\ast\wt{\cM}$, as the derived
bracket of the Poisson symplectic bracket $\{\cdot,\cdot\}_{\wt{\cM}}$ on $\sT^\ast\wt{\cM}$,
\be\label{lf011}\{\zf,\zc\}_\zw=\{\{\zf,{\cJ}\}_{\wt{\cM}},\zc\}_{\wt{\cM}}\,.\ee
Note that the bracket is odd so, for homogeneous $\zf,\zc$, we have
$$\{\zf,\zc\}_\zw=-(-1)^{(g(\zf)+1)(g(\zc)+1)}\{\zc,\zf\}_\zw\,,$$ and the corresponding Legendre bracket is an odd
Jacobi bracket and reads
\bea\nn \{ F,G\}_\za &=&-\frac{\pa F}{\pa{x^a}}\frac{\pa G}{\pa \zvy^a}+
\frac{\pa F}{\pa{\zvy^a}}\zvy^a\frac{\pa G}{\pa\zx}
+F\cdot\frac{\pa G}{\pa\zx}\\
&&-(-1)^{g(F)}\left(-\frac{\pa F}{\pa\zvy^a}\frac{\pa G}{\pa{x^a}}+ \frac{\pa F}{\pa\zx}\zvy^a\frac{\pa
G}{\pa{\zvy^a}}+\frac{\pa F}{\pa\zx}\cdot G\right)\,.\label{lfo3}
\eea

\medskip
If $P$ is a principal  $\R^\ti$-bundle over $\cM$, then we have in general no trivialization $P=\R^\ti\ti\cM$,
so that formula (\ref{PJ1}) makes sense only locally. On the other hand, there is an obvious canonical
identification of 1-homogeneous functions on $P$ with sections of the even line bundle $P^e$ (or the odd line
bundle $P^o$). Conversely, this identification identifies sections $\zs$ of an even (resp., odd) line bundle
$L$ over $\cM$ with 1-homogeneous functions $\zi_\zs$ on $P=(L^*)^\ti$. Now, a generalization of Theorem
\ref{tfun} takes the following form.
\begin{theorem}\label{th1}
Let $L$ be an even line bundle. There is a one-to-one correspondence between even (resp., odd) Kirillov
brackets $[\cdot,\cdot]_L$ on $L$ and even (resp., odd) Poisson tensors $\cJ$ of degree -1 on $(L^\ast)^\ti$
given by the formula
\be\label{le}\{\zi_\zs,\zi_{\zs'}\}_\cJ=\zi_{[\zs,\zs']_L}\,.\ee
\end{theorem}
In particular, any even or odd contact structure $C$ defines a Kirillov bracket $[\cdot,\cdot]_{C^*}$ on
sections of the dual line bundle $C^*$.
\begin{definition} Let $C$ be a contact structure. The Kirillov bracket $[\cdot,\cdot]_{C^*}$ on
sections of the dual line bundle $C^*$ will be called the {\it Legendre bracket} associated with the contact
structure $C$. \end{definition} The Legendre bracket is nondegenerate in the sense that it is represented by a
symplectic Poisson bracket on $C^\ti$. Conversely, if we have a nondegenerate Kirillov bracket on a line
bundle $L$, then the corresponding Poisson structure $\cJ$ on $(L^\ast)^\ti$ is symplectic and homogeneous,
thus represents the contact structure $C_L$. This gives an identification of nondegenerate Kirillov brackets
with Legendre brackets.

As usual, we can interpret an even Poisson tensors $\cJ$ on the principal $\R^\ti$-bundle $P=(L^\ast)^\ti$ as
quadratic Hamilonians on $\zP\sT^\ast P=\zP C_L^\ti$  satisfying the homological condition
\be\label{Shc}\{ \cJ,\cJ\}_{P}^\zP=0\,,
\ee
where  $\{ \cdot,\cdot\}_{P}^\zP$ is the Schouten  bracket on $P$. As the Schouten bracket is of degree $-1$,
it is closed on homogeneous functions of degree 1, i.e.  functions from $\A^{(1,\bullet)}(\zP\sT^* P)$  if we
consider the standard bi-gradation on $\zP\sT^* P$. In the introduced terminology, the Schouten bracket $\{
\cdot,\cdot\}_{P}^\zP$ reduced to $\A^{(1,\bullet)}(\zP\sT^*P)$ is the Legendre bracket
$[\cdot,\cdot]_{(\zP\sT^*P)^e}$ on sections of the line bundle $(\zP\sT^*P)^e$. In the case of an odd Poisson
tensor, the situation is analogous and we replace the Schouten bracket with the symmetric Schouten bracket. If
$P=(L^*)^\ti$, then, according to our conventions, we can write the Legendre bracket  as $[\cdot,\cdot]_{\zP
C_L^*}$ in the even as well as in the odd case.

Using our standard local coordinates $(t,z,x^a,\mathfrak{p}_a)$ on $\zP C_L^\ti$, with the bi-degrees $(1,0)$,
$(0,1)$, $(0,0)$, and $(0,1)$, respectively, we can write locally any Poisson bracket of degree -1 as derived
from a quadratic Hamiltonian of degree 1
$$\cJ=t(f_{ab}(x)\mathfrak{p}_a\mathfrak{p}_b+h_a(x)\mathfrak{p}_az+u(x)z^2)\,.
$$
Note that we reversed the parity, so $z$ is even if and only if $L$ is odd, and $\mathfrak{p}_a$ and $x^a$
have different parities. Therefore the term $u(x)z^2$ is non-zero only for even $L$. In the Darboux
coordinates $(t,z,x^a,{p}_a)$, $p_a=t\mathfrak{p}_a$ we get
$$\cJ=\frac{1}{t}f_{ab}(x){p}_a{p}_b+h_a(x){p_a}z+tf(x)z^2\,.
$$
Since the Hamiltonian vector field of $z$ is $\pa_t$, this coincides with formula (\ref{jac1}) in which the
bivector field $\zL$ represented by $f_{ab}(x){p}_a{p}_b$ and the vector field $\zG$ represented by
$h_a(x){p_a}$. This coincides also with (\ref{jac}) if we take into account that on 1-homogeneous functions
$t\pa_t$ acts as the identity.

Completely analogously to the classical result describing Poisson brackets as derived brackets for Poisson
tensors viewed as quadratic Hamiltonians with respect to the Schouten bracket and to a similar result for
Jacobi brackets \cite{GM1,GM2}, we can now formulate the following description of Kirillov brackets.
\begin{theorem}\label{t-main1}
There is a one-to-one correspondence between Kirillov brackets $[\cdot,\cdot]_L$ on a line bundle $L$ and
1-homogeneous quadratic Hamiltonians  $\cJ\in\A^{(1,2)}(\zP C_L^\ti)$ satisfying the homological condition $[
\cJ,\cJ]_{\zP C_L^*}=0$ with respect to the Legendre bracket on sections of $\zP C_L^*$. The Kirillov bracket
is in this case given as the derived bracket
$$[\zs,\zs']_L=[[\zs,\cJ]_{\zP C_L^*},\zs']_{\zP C_L^*}\quad
$$
of the Legendre bracket $[ \cdot,\cdot]_{\zP C_L^*}$ with respect to $\cJ$.
\end{theorem}

\section{Kirillov algebroids}
It is well known that linear Poisson structures represent Lie algebroids. Our general philosophy suggests to
consider linear principal Poisson $\R^\ti$-bundles and the corresponding `algebroids'.
\begin{definition} A {\it Kirillov algebroid} is a linear principal Poisson $\R^\ti$-bundle. If this principal $\R^\ti$-bundle is trivial, we will speak about a {\it Jacobi algebroid}.
\end{definition}
To look closer on such structures, consider a linear principal Poisson $\R^\ti$-bundle over $P_0$ and the
corresponding fibrations $\zt_i:P\to P_i$, $i=0,1$. Let us denote with $\A^{(k,\,l)}(P)$, $k\in\Z$, $l\in\N$,
the space of those smooth functions on $P$ which are homogeneous of degree $k$ with respect to $h^0$ and
homogeneous of degree $l$ with respect to $h^1$ (or $\zD_1$). We will say  simply that they are
$(k,\,l)$-homogeneous. Recall that we have the decomposition $P=P_1\ti_\cM P_0$, where
$\cM=h^1_0(P_0)=h^1_0(P)/\R^\ti$ (Theorem \ref{ii}). It is clear that
$$\A^{(k,\,l)}(P)=\A^k(P_1)\ot_{C^\infty(\cM)}\A^l(P_0)$$
and that
$$\A^{(1,\,\bullet)}(P)=\bigoplus_{l=0}^\infty\A^{(1,\,l)}(P)$$ is a bi-module over the algebra
$$\A^{(0,\,\bullet)}(P)=\bigoplus_{l=0}^\infty\A^{(0,\,l)}(P)$$ of polynomial functions on the vector bundle $P_0$.
%Locally, $\A^{(0,\bullet)}(P)$ is the algebra of polynomials in $y^i$ with coefficients in $\A^{(0,0)}(P)=C^\infty(\cM)$
%and $\A^{(1,\bullet)}(P)$ is one-dimensional and freely generated over $\A^{(0,\bullet)}(P)$ by the even coordinate
%function $t$.
Note that functions from $\A^{(0,1)}(P)$ can be identified with sections of the vector bundle $P_0^*$.

It is also easy to see that elements of $\A^{(1,1)}(P)$ represent linear sections of the even line bundle
$P^e$ associated with the principal $\R^\ti$-bundle $P$, and elements of $\A^{(1,0)}(P)$ represent sections of
the line bundle $P_1^e$. This line bundle can be also described as the restriction of $P^e$ to the
zero-section $0_{P_0}\simeq\cM$ of the vector bundle $P_0$. Note that, as $P\simeq P_1\ti_\cM P_0$, the
$\A^{(0,0)}(P)$-module $\A^{(1,1)}(P)$ can be identified with the module of sections of the vector bundle
$P_1^e\ot_\cM P_0^*$ (or the vector bundle $P_1^o\ot_\cM P_0^*$). Indeed, any such section $\ze=\zs\ot\ze_0$
defines by duality a function $\zi_\ze$ of degree $(1,1)$ on $P\simeq P_1\ti_\cM P_0$ by
$\zi_\ze=\zi_\zs\zi_{\ze_0}$, where $\zi_\zs$ is a 1-homogeneous function on $P_1=(L^*)^\ti$ represented by
$\zs\in\Sec(L)$ and $\zi_{\ze_0}$ is the linear function on $P_0$ represented by the section
$\ze_0\in\Sec(P_0^*)$.

The Poisson bracket $\{\cdot,\cdot\}=\{\cdot,\cdot\}_\cJ$ associated with the Poisson tensor $\cJ$ on $P$ is
homogeneous of bi-degree $(-1,-1)$, so
\be\label{laaa}\{\A^{(1,1)}(P),\A^{(k,l)}(P)\}\subset\A^{(k,l)}(P)\,.
\ee
In particular, the modules $\A^{(1,\,\bullet)}(P)$ over $\A^{(0,\,\bullet)}(P)$ and $\A^{(1,1)}(P)$ over
$\A^{(0,0)}(P)$ are closed with respect to the Poisson bracket.
\begin{definition}
The restriction of the Poisson bracket to the module $\A^{(1,\,\bullet)}(P)$ we will call the {\it graded
Kirillov bracket} associated with  $(P,\cJ,h^0,h^1)$ and denoted as $[\cdot,\cdot]_\cJ^\bullet$, or simply
$[\cdot,\cdot]^\bullet$ if $\cJ$ is fixed. The restriction of the Poisson bracket to the module
$\A^{(1,1)}(P)$ we will call the {\it Kirillov algebroid bracket} associated with  $(P,\cJ,h^0,h^1)$ and
denoted as $[\cdot,\cdot]_\cJ^1$, or simply $[\cdot,\cdot]^1$ if $\cJ$ is fixed.

\end{definition}
Let us denote $L=L(P)=P_1^e$ if $\cJ$ is even, and $L=L(P)=P_1^o=\zP P_1^e$ if $\cJ$ is odd. If we put
$\cE(P)=L\ot_\cM P_0^*$, then this bracket interpreted as a bracket on sections of $\cE(P)$ is even and
defines a Lie algebroid structure on $\cE(P)$ with the anchor map defined by
$$\zr(\ze)(f)=\{\zi_\ze,f\}\,.$$
This Lie algebroid we will call the {\it Lie algebroid of the linear principal Poisson $\R^\ti$-bundle}
$(P,\cJ,h^0,h^1)$.

We have also
\be\label{ff}\{\A^{(0,1)}(P),\A^{(1,0)}(P)\}\subset\A^{(0,0)}(P)
\ee
that allows us to represent any $\zs\in\A^{(0,1)}(P)\simeq\Sec(P_0^*)$ as a first-order differential operator
$\Upsilon(\zs):\Sec(L)\ra C^\infty(\cM)$. Since
$$\{\A^{(0,0)}(P),\A^{(1,0)}(P)\}\subset\A^{(0,-1)}(P)=\{ 0\}\,,$$ we have $\Upsilon(f\zs)=f\Upsilon(\zs)$ for
$f\in C^\infty(\cM)$. The map $\zs\mapsto \Upsilon(\zs)$ preserves or reverses the parity depending on the
parity of $\cJ$, so it represents a smooth map into $(\sJ^1L)^*$,
\be\label{bum}\Upsilon:P_0^\ast\ra (\sJ^1L)^*\,,
\ee
covering the identity on $\cM$. We have an obvious extension of this map to a vector bundle morphism
\be\label{doaa}\zr^{\! 1}:\cE(P)=L\ot_\cM P_0^\ast\ra (\sJ^1L)^\ast\ot_\cM L\simeq\DO^1(L,L)\,,
\ee
that follows also directly from the inclusion
$$\{\A^{(1,1)}(P),\A^{(1,0)}(P)\}\subset\A^{(1,0)}(P)\,.$$
In other words, $\zi_{\zr^{\! 1}(\ze)(\zs)}=\{\zi_\ze,\zi_\zs\}$.

From the Jacobi identity it follows immediately that the map $\zr^{\! 1}$ induces a homomorphism of the
Poisson bracket on $\A^{(1,1)}(P)$ viewed as the space of sections $\zs$ of $\cE(P)$ into the
(super)commutator in $\DO^1(L,L)$, i.e., for $\ze,\ze'\in\Sec(\cE(P))$ of parity $\zd,\zd'$, respectively, we
have
\be\label{doc}\zr^{\! 1}\left([\ze,\ze']^1\right)=[\zr^{\! 1}(\ze),\zr^{\! 1}(\ze')]=
\zr^{\! 1}(\ze)\circ \zr^{\! 1}(\ze')-(-1)^{\zd\zd'}\zr^{\! 1}(\ze')\circ \zr^{\! 1}(\ze)\,.
\ee
The map $\zr^{\! 1}$ is actually a Lie algebroid morphism.

The dual of the map $\Upsilon$, i.e. the map $\Upsilon^\ast:\sJ^1L\ra P_0$, can be extended to a map
\be\label{fho}\wt{\Upsilon}:C_{L}^\ti\simeq(L^\ast)^\ti\ti_\cM \sJ^1L\ra(L^\ast)^\ti\ti_\cM
 P_0\simeq P\,.
\ee
It is easy to see that, for any section $\ze$ of $\cE(P)$, we have
$\zi_\ze\circ\wt{\Upsilon}=\zi_{\zr^{\! 1}(\ze)}$.
Since the restriction of $(-1,-1)$-homogeneous Poisson bracket to $(1,1)$-homogeneous functions completely
determines the bracket, we infer, in view of (\ref{doc}), that $\wt{\Upsilon}$ is a Poisson map. This is the
contact analog of the map $\zr^\ast:\sT^\ast M\ra E^\ast$, dual to the anchor map $\zr:E\ra\sT M$, defined for
a Lie algebroid on a vector bundle $E$.

Note that we can change the parity in fibers of the vector bundle $\zt_0:P\ra P_1$. In this way we get a new
linear principal Poisson $\R^\ti$-bundle $(\zP P,\cJ^\zP,h^0,h^1)$ over $\cM$ with the reversed parity of the
Poisson tensor and with the double bundle structure described by the diagram
\be\label{FD01}\xymatrix{
\zP P\ar[rr]^{\zt_0} \ar[d]^{\zt_1} && \zP P_0\ar[d]^{\bar{\zt_1}} \\
P_1\ar[rr]^{\bar{\zt_0}} && \cM }
\ee
To be more precise, let us observe that, formally, $\A^{(k,l)}(P)=\A^{(k,l)}(\zP P)$ for $0\le k,l\le 1$. We do
not change the Poisson bracket on these spaces, i.e. $\{\cdot,\cdot\}_\cJ$ and
$\{\cdot,\cdot\}_\cJ^{\zP}=\{\cdot,\cdot\}_{\cJ^\zP}$ coincide on $\A^{(1,1)}(P)=\A^{(1,1)}(\zP P)$. In other
words, the corresponding Poisson tensors are formally the same but define Poisson brackets of the reversed
parity. Of course, the algebras $\A^{(0,\bullet)}(P)$ and $\A^{(1,\bullet)}(\zP P)$ differ, since elements
from $\A^{(0,1)}(P)$ and $\A^{(0,1)}(\zP P)$ commute differently. The bracket $\{\cdot,\cdot\}_\cJ^{\zP}$
restricted to $\A^{(1,\bullet)}(\zP P)$ we will call the {\it graded Kirillov bracket} associated
with $(P,\cJ)$ and denote $[\cdot,\cdot]_{\cJ^\zP}^\bullet$. In this way we get the following.
\begin{theorem}\label{tr1}
Let  $(P,\cJ,h^0,h^1)$ be a linear principal Poisson  $\R^\ti$-bundle over $\cM$ with the double bundle
structure (\ref{FD0}) and let $L=L(P)$ be the line bundle $P_1^e$ in the case when $\cJ$ is even, and the line
bundle $P_1^o$ in the case when $\cJ$ is odd, with sections represented by functions from $\A^{(1,0)}(P)$.
Then, the corresponding Poisson bracket satisfies (\ref{laaa}) and induces:
\begin{description}
\item{(a)} a Lie algebroid bracket $[\cdot,\cdot]^1$ on the vector bundle $\cE(P)=L\ot_\cM P_0^\ast$ over $M$ whose sections are represented by
functions in the module $\A^{(1,1)}(P)$ over $\A^{(0,0)}(P)\simeq C^\infty(\cM)$;

\item{(b)} a morphism $\wt{\Upsilon}:C_{L}^\ti\ra P$ of linear principal Poisson \
$\R^\ti$-bundles covering the identity on $(L^*)^\ti$.
%, i.e. induced by a morphism $\Upsilon^\ast:\sJ^1L\ra P_0$.

\item{(c)} a Lie algebroid morphism
\be\label{lam}\zr^{\! 1}:\cE(P)\ra\DO^1(L,L)
\ee
from $\cE(P)$ into the Lie algebroid $\DO^1(L,L)$ of first-order differential operators from $L$ into $L$ with
the (super)commutator bracket whose principal part gives the anchor of $\cE(P)$.

\item{(d)} a linear principal Poisson \ $\R^\ti$-bundle
$(\zP P,\cJ^\zP,h^0,h^1)$ over $\cM$ with the reversed parity of the Poisson tensor and a graded Kirillov bracket on
the $\A^{(0,\bullet)}(\zP P)$-module $\A^{(1,\bullet)}(\zP P)$.

\end{description}
\end{theorem}
\begin{example}\label{eks0} Let $L$ be line bundle over $\cM$ and $\{\cdot,\cdot\}_L$ be a Kirillov bracket on sections of $L$. According to Theorem \ref{th1}, this Kirillov structure is associated with a principal Poisson $\R^\ti$-bundle structure  $(\bar{P},\bar{\cJ},\bar{h}^0)$ on $\bar{P}=(L^*)^\ti$. Passing to the tangent lifts (cf. Example \ref{ex1}), we get a linear principal Poisson $\R^\ti$-bundle $P=(\sT (L^*)^\ti,\dt\bar{\cJ},\sT \bar{h}^0,h^1)$.
In this case the Lie algebroid $\cE(P)$ is exactly the Lie algebroid associated with a Kirillov bracket (a
local Lie algebra structure) introduced in \cite{GM1}. In particular, if the line bundle $L$ is trivial and we
deal with a Jacobi structure on $\cM$, then $P=\sT(\R^\ti\ti\cM)$, so $P_0=\sT\cM\ti\R$ and  the corresponding
Lie algebroid structure lives on $\cE(P)=L\ot_\cM P_0^*=(\sT \cM\ti\R)^\ast=\sT^\ast \cM\ti\R$ and it is
exactly the Lie algebroid structure on $\sT^\ast \cM\ti\R$ associated with a Jacobi structure on $\cM$ as
first described in \cite{KSB} (cf. also \cite{Vs}).
\end{example}

\medskip
Let us stress that the structure we obtained on $\cE(P)$ is not only a Lie algebroid structure with a Lie
bracket $[\cdot,\cdot]^{\! 1}$, but also a Lie algebroid morphism $\zr^{\! 1}:\cE(P)\ra\DO^1(L,L)$. A part of
this morphism defines the anchor map of the Lie algebroid structure by passing to the principal symbol of a
first-order differential operator, but there is another part not directly seen by the Lie algebroid structure
on $\cE(P)$. Let us observe that there is a natural identification of first-order differential operators $D$
between sections of $L$ with invariant vector fields $\bar{D}$ in the principal $\R^\ti$-bundle $(L^*)^\ti$,
i.e. with sections of the Atiyah algebroid $\sT(L^*)^\ti/\R^\ti$ of $(L^*)^\ti$. The vector field $\bar{D}$ is
uniquely defined by the formula
\be\label{doid} \bar{D}(\zi_\zs)=\zi_{D(\zs)}\,,
\ee
where $\zs$ is any section of $L$ and $\zi_\zs$ is the corresponding homogeneous function on $(L^*)^\ti$.

On the other hand, the linear principal $\R^\ti$-bundle $P$ can be reconstructed from $\cE(P)$. Indeed, denote
the vector bundle dual to $P\to(L^*)^\ti$ with $P^*$, so that $P^*=(L^*)^\ti\ti_\cM P_0^*$. We have a
canonical isomorphism of vector bundles over $(L^*)^\ti$,
\be\label{recc}\mathfrak{H}:(L^*)^\ti\ti_\cM\cE(P)=(L^*)^\ti\ti_\cM\left(L\ot_\cM P_0^*\right)\to (L^*)^\ti\ti_\cM P_0^*=P^*\,,
\ee
given by
\be\label{recc1} (\zs^*,\zs\ot\zn)\mapsto (\zs^*,\zs^*(\zs)\zn)\,.
\ee
This isomorphism identifies sections $\ze$ of $\cE(P)$, viewed as $\R^\ti$-invariant (we will say also {\it
constant}) sections $\bar{\ze}$ of $E(P)=(L^*)^\ti\ti_\cM\cE(P)$, with 1-homogeneous sections $\wt{\ze}$ of
the vector bundle $P^*$, thus homogeneous functions of bi-degree $(1,1)$ on $P$. Note that the Lie algebroid
structure on $\cE(P)$ together with the Lie algebroid morphism (\ref{lam}) define a Lie algebroid structure on
$E=(L^*)^\ti\ti_\cM\cE(P)$, thus $P^*$, with a bracket $[\cdot,\cdot]$ and the anchor $\zr:E(P)\to\sT
(L^*)^\ti$ by
\bea\label{111}
&[\bar{\ze}_1,\bar{\ze}_2]=\overline{[{\ze}_1,{\ze}_2]^{\! 1}}\,,\\
&\zr(\bar{\ze})=\overline{\zr^{\! 1}(\ze)}\,. \label{recon3}
\eea
This Lie algebroid structure on $E$ over $(L^*)^\ti$ is {\it invariant}  with respect to the $\R^\ti$-action
in the sense that the bracket of invariant sections is invariant and the anchor of an invariant section is an
invariant vector field. The Lie algebroid on $P^*$ completely determines the Poisson structure we started
with. We can go back to $\cE(P)$ by considering constant sections only. In this way we get the following.
\begin{theorem}\label{thmain2}
A Kirillov algebroid can be equivalently defined as an invariant Lie algebroid structure on a vector
bundle $E$ over a principal $\R^\ti$-bundle $P_0$. If $P_0$ is trivial, $P_0=\R^\ti\ti\cM$, then we deal with
a Jacobi algebroid.
%or a Lie algebroid structure on a vector bundle $\cE$ over $\cM$ together with a Lie algebroid morphism (\ref{lam}) into %the Lie algebroid of first-order differential operators acting on sections of a line bundle $L$ over $\cM$.
\end{theorem}
\begin{remark}\label{rm1} The isomorphism (\ref{recc}) is a geometrical counterpart of a bi-degree shift in the linear principal bundle $(P^*,h^0,h^1)$ associated with a change of the $\R^\ti$-action. Since the group $\R^\ti$ acts canonically in each vector bundle by non-zero homotheties, we can consider a new principal $\R^\ti$-action $\wt{h}^0$ on $P^*$ putting $\wt{h}^0_t=h^0_t\circ h^1_{-t}$. This produces the shift $\A^{(k,l)}(P^*)\mapsto\A^{(k-l,l)}(P^*)$ in bi-degrees. In particular, linear 1-homogeneous tensors, i.e. tensors of the bi-degree $(1,1)$, become linear invariant tensors, i.e. tensors of bi-degree $(0,1)$. This explains how 1-homogeneous Lie algebroid structures have turned into invariant Lie algebroid structures.
\end{remark}
\begin{example} In the case of a Jacobi algebroid, $P_0=\R^\ti\ti\cM$, we can reduce the Lie algebroid structure to $\cE=E/\R^\ti$ over $\cM$, interpreting sections of $\cE$ as invariant sections of $E$. As the anchor map associates with any  invariant section $\bar{\ze}$ of $E$ an invariant vector field $\zr(\bar{\ze})$ on $\R^\ti\ti\cM$, its projection to $\cM$ gives the anchor $\bar{\zr}^{\! 1}(\bar{\ze})$ for the Lie algebroid $\cE$, while its vertical part, $f(m)\pa_t$,  defines a map $\bar{\ze}\mapsto f$ represented by a 1-form $\zm$ on $\cE$ (a section of $\cE^*$), $\zm(\bar{\ze})=f\in C^\infty(\cM)$.
Since $\zr:E\to\sT(\R^\ti\ti\cM)$ is a Lie algebroid morphism, the 1-form $\zm$ is $\cE$-closed,
$\xd_\cE\zm=0$, and we arrive at the standard definition of a Jacobi algebroid (called also a {\it generalized
Lie algebroid}) \cite{GM1,IM}.
\end{example}

\section{Contact $n$-vector bundles}
In the case of the canonical contact structure $C_L$ and the corresponding symplectic Poisson tensor $\cJ$ on
$P=C^\ti_L$, the maps $\Upsilon$, $\zr^{\! 1}$, and $\wt{\Upsilon}$, defined in the previous section, are all
isomorphisms. Indeed,  in this case $P_0=\sJ^1L$ and $P_1=(L^*)^\ti$, so $L(P)=L$, $\cE(P)=\DO^1(L,L)$,  and
the map $\Upsilon$, thus $\zr^{\! 1}$ and $\wt{\Upsilon}$, are all the identity maps. The space of linear
sections of $C_L^\ast$ can be then identified with the space of first-order linear differential operators
acting on sections of $L$. Fixing a local trivialization of $L$ and the corresponding affine coordinates
$(t,x^a,z,\mathfrak{p}_b)$ in $C_L$, the linear section $\ze$ of $C_L^\ast$ corresponding to the
$(1,1)$-homogeneous function $\zi_{\ze}=t(f_a(x)\mathfrak{p}_a+h(x)z)$ on $C_L$ represents the differential
operator $D_\ze=f_a(x)\pa_{x^a}+h(x)$. These first-order differential operators are sections of the
corresponding bundle $\DO^1(L,L)$ over $\cM$ with the basis of local sections represented by $p_a=t\mathfrak{p}_a$
and $tz$. Moreover, the Legendre bracket $\{\ze,\ze'\}_{C_L}$ of linear sections of given parity corresponds,
{\it via} the identification $\ze\leftrightarrow D_\ze$, to the (graded) commutator $[D_\ze,D_{\ze'}]$ of
differential operators.

Conversely, if $\cJ$ is invertible, i.e. $\zw=\cJ^{-1}$ is a symplectic structure, then the map $\Psi_\cJ$ is
an isomorphism. For, let us observe that the map $\Psi_\cJ$ is derived from a bilinear form
$\la\cdot,\cdot\ran_\cJ$ on $P_0^\ast\ti_\cM \sJ^1L$ defined by means of (\ref{ff}). If the Poisson bracket is
symplectic, then this map is a pairing, so $\Psi_\cJ$ is an isomorphism. Indeed, being of bi-degree $(-1,-1)$
the Poisson bracket $\{\cdot,\cdot\}_\cJ$ is completely determined by its restriction to
$\A^{(1,1)}(P)=\A^{(1,0)}(P)\cdot\A^{(0,1)}(P)$ and it is nondegenerate if and only if
$\la\cdot,\cdot\ran_\cJ$ is nondegenerate. Summarizing our observations, we get the following.

\begin{theorem}\label{tr2} For a linear principal Poisson \ $\R^\ti$-bundle $(P,\cJ,h^0,h^1)$ the following are equivalent:
\begin{description}
\item{(i)} The Poisson tensor $\cJ$ is invertible (symplectic);

\item{(ii)} the map
$$\Upsilon:P_0^\ast\ra \sJ^1L^\ast$$ is a vector bundle isomorphism;

\item{(iii)} the map
$$\zr^{\! 1}:\cE(P)\ra\DO^1(L,L)$$
described by (\ref{doaa}) is an isomorphism of Lie algebroids.
\end{description}
If this is the case, then $(P,\cJ^{-1},h^0,h^1)$ is a linear symplectic principal \ $\R^\ti$-bundle; and the
map (\ref{fho}) constitutes a canonical isomorphism between $(P,\cJ^{-1})$ and the linear symplectic principal
$\R^\ti$-bundle $(C_{L}^\ti,\zw)$ associated with the canonical contact structure on the first jet bundle
$\sJ^1L$.
\end{theorem}
\begin{corollary}\label{c5}
Any linear even (resp., odd) contact structure is isomorphic with a canonical contact structure on the first
jet bundle of an even (resp., odd) line bundle, i.e. with the cotangent bundle $\sT^\ast \uP$ (resp., $\zP
\sT^\ast \uP$) of a principal $\R^\ti$-bundle $\uP$.
\end{corollary}
\begin{remark} The above result is a contact analog of the well-known fact that any linear symplectic manifold can be
identified with the cotangent bundle of a manifold with its canonical symplectic structure.

It is easy to see, in turn, that the space $\A^{(1,\bullet)}(L)$ represents the space of principal symbols of
differential operators in $L$, and that the Poisson bracket $\{\cdot,\cdot\}$ is closed on principal symbols
representing the principal part of the (super)commutator of differential operators. We will not go into
details just referring to \cite{GM2} and the literature cited there. Note only that the observation on the
canonical Poisson structure on $\sT^\ast\cM$ can be obtained from the bracket in principal symbols viewed as
polynomial functions in $\sT^\ast\cM$ and the commutator of differential operators is, up to our knowledge,
due to Vinogradov \cite{Vi}.
\end{remark}
\begin{example}\label{eks}
Consider the contact structure of the trivial even line bundle $L=\R^\ti\ti M$ over  $\cM$, i.e. the canonical
linear principal symplectic $\R^\ti$-bundle $P=\sT^*(\R^\ti\ti \cM)=\R^\ti\ti(\R\ti\sT^* \cM)$. In the
standard Darboux coordinates $(t,z,x^a,p_a)$ of the bi-degrees, respectively, $(1,0),(0,1),(0,0),(1,1)$, the
Poisson tensor takes the form $\cJ=\pa_t\pa_z+\pa_{x^a}\pa_{p_a}$. On the vector bundle $\sT^*(\R^\ti\ti
\cM)/\R^\ti=\sJ^1(\R\ti \cM)=\R\ti\sT^* \cM$ we have coordinates $(z,x^a,\fp_a)$, where $t\fp_a=p_a$. The
vector bundle $\cE(P)$ is in this case $\cE(P)=(\R\ti\sT^* \cM)^*=\R\ti\sT \cM=\DO^1(\cM)$. Any its section
$\ze=g(x)+f^a\pa_{x^a}$ is a first-order differential operator on $\cM$ and corresponds to the function
$\zi_\ze=g(x)tz+f^a(x)p_a=t(g(x)z+f^a(x)\fp_a)$ of the bi-degree $(1,1)$ on $P$. It is easy to see that the
Poisson bracket of such functions corresponds to the commutator of the associated first-order differential
operators. The corresponding Kirillov algebroid is $E(P)=\R^\ti\ti\R\ti\sT\cM$ with local coordinates
$(t,\dot{\mathfrak{t}},x^a,\dot{x}^a)$ and and a vector bundle isomorphism
$\mathfrak{H}:=E(P)\to\sT(\R^\ti\ti\cM)=P^*$ given in the adapted coordinates $(t,x^a,\dot{t},\dot{x}^a)$ in
$\sT(\R^\ti\ti\cM)$ of the bi-degrees $(1,0),(0,0),(1,1),(0,1)$, respectively,  by
$$(t,x^a,\dot{t},\dot{x}^a)=(t,x^a,t\dot{\mathfrak{t}},\dot{x}^a)\,.$$
The Kirillov algebroid is therefore $\sT(\R^\ti\ti\cM)$ with the bracket of vector fields for which
$\R^\ti$-invariant vector fields $g(x)t\pa_t+f^a(x)p_a$ form a Lie subalgebra.
\end{example}

Recall that an $n$-vector bundle $\cM$ with an even (odd) contact structure $C\subset\sT^\ast \cM$ of degree
$\ao$ we will call an even (odd) {\it contact $n$-vector bundle}. It can be also described as an $n$-linear
symplectic principal $\R^\ti$-bundle. In this language Corollary \ref{c5} tells us that even (odd) contact
vector bundles are exactly the canonical contact structures, i.e. they are represented by the cotangent
bundles of principal $\R^\ti$-bundles. We have the following straightforward generalization.
\begin{theorem}\label{nlcs} The cotangent bundle $\sT^\ast \uP$ (resp., $\zP \sT^\ast \uP$) of an $(n-1)$-linear
principal $\R^\ti$-bundle $(\bar{P},\bar{h}^0,\bar{h}^1,\dots,\bar{h}^{n-1})$  is canonically an even (resp., odd)
contact $n$-vector bundle with the homogeneity structures $h^i=\sT^* \bar{h}^i$, $i=0,\dots, n-1$,  and the homogeneity
structure $h^n$ associated with homotheties in the vector bundle $\zp_\uP:=\sT^\ast \uP\to \uP$.

Any even (resp., odd) contact $n$-vector bundle $(P,\zw,h^0,\dots,h^n)$ is of this form, where as $\uP$ can be
taken any of the side vector bundles $h^i_0:P\to \uP_i$, $i=1,\dots, n$. In particular, all contact $n$-vector
bundles $\sT^\ast \uP_i$ (resp., $\zP \sT^\ast \uP_i$) are canonically isomorphic as well as all first jet
bundles $\sJ^1L_i$, where $L_i=\uP_i^e$ (resp., $L_i=\uP_i^o$), $i=1,\dots,n$.
\end{theorem}
\begin{proof}
As the proof is completely analogous in the odd case, let us assume that our case is even. If
$(\bar{P},\bar{h}^0,\bar{h}^1,\dots,\bar{h}^{n-1})$ is an $(n-1)$-linear principal $\R^\ti$-bundle with the
principal $\R^\ti$-bundle structure $\ute_0:\uP\ra P_0$, and the vector bundle structures
$\ute_i:\uP\ra\uP_i$, $i=1,\dots,n-1$, then $P=\sT^\ast\uP$ is canonically an $n$-linear symplectic principal
$\R^\ti$-bundle with the canonical symplectic structure, with the principal $\R^\ti$-bundle structure
associated with $h^0=\sT^* \bar{h}^0$ (see theorem \ref{plift0}), and with the vector bundle structures
$h^1,\dots, h^n$, where $h^i=\sT^\ast \bar{h}^i$, $i=1,\dots,n-1$, and $h^n$ is associated with the vector
bundle structure $\zp_\uP:=\sT^\ast \uP\to \uP$.

Conversely, if $(P,\zw,h^0,\dots,h^n)$ is an $n$-linear symplectic $\R^\ti$-principal bundle, $n\ge 1$, then,
as linear symplectic structures are exactly these of the cotangent bundle (cf. \cite{GR}), the linear
symplectic structure $(P,\zw,h^n)$ is canonically isomorphic with $(\sT^\ast\uP,\zw_\uP,h^{\sT^\ast\uP})$,
where $\uP=\uP_n=\zt_n(P)$ is the base of the $n$th vector bundle structure and $h^{\sT^\ast\uP}$ is the
homogeneity structure of the vector bundle $\zp_\uP:\sT^\ast \uP\to \uP$. Since  $h^i$ commute with $h^n$, the
corresponding Euler vector fields $\zD_i$ are linear on $\sT^\ast\uP$ and project to Euler vector fields
$\bar{\zD}_i$ on $\uP$, $i=1,\dots,n-1$. As the linear vector fields on $\sT^\ast\uP$, making the canonical
symplectic form $\zw_\uP$ homogeneous od degree 1, are completely determined by their values on $\uP$ (cf.
\cite[Proposition 5.1]{GR}), we have $\zD_i=\sT^\ast\bar{\zD}_i$, $i=1,\dots,n-1$.  Hence, $h^i=\sT^\ast
\bar{h}^i$, $i=1,\dots,n-1$.

Similarly, $h^0$ projects to an $\R^\ti$-action $\bar{h}^0$ on $\uP$ that makes
$(\uP,\bar{h}^0,\dots,\bar{h}^{n-1})$ into an  $(n-1)$-linear principal $\R^\ti$-bundle. Since $\zw_\uP$ is
1-homogeneous with respect to $h^0$, it is necessarily the phase lift of $\bar{h}^0$.

Of course, we can as well take one of $\uP_1,\dots,\uP_{n-1}$ instead of $\uP_n$ with the same effect. Writing
the diagram (\ref{FD}) for these realizations, we get
\be\label{FDn}\xymatrix{
P\simeq\sT^\ast \uP_i\ar[rr]^{\zt_0} \ar[d]^{\zt_i} &&
P_0\simeq \sJ^1L_i\ar[d]^{\bar{\zt_i}} \\
\uP_i\ar[rr]^{\bar{\zt_0}} && \cM_i\simeq\uP_i/\R^\ti }
\ee
and the theorem follows.
\end{proof}

\begin{remark} The above theorem is a contact analog of a theorem on symplectic $n$-vector bundles
(\cite[Theorem 6.1]{GR}) stating that any such bundle $(F,\zw)$ is canonically isomorphic with the cotangent
bundle $\sT^\ast F_i$ for each of its side bundles $F_i$, $i=1,\dots,n$. In particular, all $\sT^\ast F_i$,
$i=1,\dots,n$, are canonically symplectically isomorphic. The latter is a generalization of the well-known
fact that, for a vector bundle $E$, the double vector bundles $\sT^\ast E$ and $\sT^\ast E^\ast$ are
canonically symplectically isomorphic.
\end{remark}

\section{Kirillov algebroid cohomology}
Consider again a linear principal Poisson $\R^\ti$-bundle $(P,\cJ,h^0,h^1)$. To fix our attention, let us
assume for a moment that $\cJ$ is even. The Kirillov algebroid bracket on $P^*$ is associated with the
Poisson bracket $\{\cdot,\cdot\}_\cJ$ which, according to Theorem \ref{t-main1}, is the derived bracket
generated by the linear 1-homogeneous Poisson tensor $\cJ$, viewed as a homological 1-homogeneous quadratic
Hamiltonian on
$$\zP C_P^\ti=C_{\zP P}^\ti=\zP \sT^*P\simeq\zP \sT^*\zP P^*\,,$$
from the Schouten bracket $\{\cdot,\cdot\}^{\zP}_{P}$. This Schouten bracket is of tri-degree $(-1,-1,-1)$ in
$\A(\zP \sT^*P)$, so that the quadratic Hamiltonian $\cJ$, being of the tri-degree $(1,1,2)$ and homological
induces the cochain complex
\be\label{chainc}
(\A^{(1,0,\bullet)}(\zP \sT^*P),\xd_\cJ)\,,\quad \xd_\cJ=\{\cJ,\cdot\}^{\zP}_{P}\,.
\ee
of $\A^{(0,0,0)}(\zP \sT^*P)$-modules. Note that $\A^{(0,0,0)}(\zP \sT^*P)=C^\infty(\cM)$, where $\cM$ is the
base of the vector bundle $P_0$, and since we can identify elements of $\A^{(1,0,l)}(\zP \sT^*P)$ with
1-homogeneous basic functions of degree $l$ on $\zP \sT^*P^*$, thus functions from $\A^{(1,l)}(\zP P^*)$, our
cochain complex is isomorphic with $(\A^{(1,\bullet)}(\zP P^*),\xd_\cJ)$. Note that $\A^{(1,\bullet)}(\zP
P^*)$ is canonically an $\A^{(0,\bullet)}(\zP P^*)$-module and $\xd_J$ is a first-order differential operator
on this module. A more vivid description would say that $\xd_\cJ$ is a homological first-order differential
operator acting on polynomial sections of the line bundle $\zP P\to \zP P_0$. If $\cJ$ is odd, we get an
analogous complex $(\A^{(1,\bullet)}(P^*),\xd_\cJ)$.
\begin{definition} Consider the Kirillov algebroid $E(P)\simeq P^*$ associated with a linear principal $\R^\ti$-bundle $(P,\cJ,h^0,h^1)$.
The cohomology of the complex $(\A^{(1,\bullet)}(\zP P^*),\xd_\cJ)$ if $\cJ$ is even, and of the complex
$(\A^{(1,\bullet)}(P^*),\xd_\cJ)$ if $\cJ$ is odd, we will call the {\it cohomology of the Kirillov
algebroid} $E(P)$.
\end{definition}
\begin{example}\label{eks1}
For a purely even $\cM$, consider $P=\sT^*(\R^\ti\ti \cM)=\R^\ti\ti(\R\ti\sT^* \cM)$ with the canonical
Poisson tensor $\cJ=\pa_t\pa_z+\pa_{x^a}\pa_{p_a}$ as in Example \ref{eks}. We have $P^*=\sT(\R^\ti\ti \cM)$
with adapted coordinates $(t,x^a,\dot{t},\dot{x}^a)$ and $\zP\sT^* \zP P^*=\zP\sT^*\zP\sT(\R^\ti\ti \cM)$ with
the adapted Darboux coordinates $(t,x^a,\dot{t},\dot{x}^a,\dot{z},\dot{p}_a,z,p_a)$. If, according to
$$\zP\sT^*\zP\sT(\R^\ti\ti \cM)\simeq \zP\sT^*\sT^*(\R^\ti\ti \cM)=\zP\sT^* P\,,$$
we will regard $(t,x^a,z,p_a,\dot{z},\dot{p}_a,\dot{t},\dot x^a)$ as coordinates in $\zP\sT^* P$, of the
tri-degrees, respectively,
$$(1,0,0),(0,0,0),(0,1,0),(1,1,0),(0,1,1),(1,1,1),(1,0,1),(0,0,1)\,,$$
and the parity indicated by the last term in the bi-degree, the canonical Poisson tensor on $\sT^* P$ takes
the form
$$\pa_{\dot{t}}\pa_z+\pa_t\pa_{\dot{z}}+\pa_{\dot{x}^a}\pa_{p_a}+\pa_{x^a}\pa_{\dot{p}_a}\,,$$
so that $\cJ$ is represented by the quadratic and odd Hamiltonian of the tri-degree $(1,1,2)$
$$\zi_\cJ=\dot{z}\dot{t}+\dot{p}_a\dot{x}^a\,.$$
The corresponding Hamiltonian vector field is therefore
$$\dot{z}\pa_z+\dot{t}\pa_t+\dot{p}_a\pa_{p_a}+\dot{x}^a\pa_{x^a}$$
which, reduced to functions from
$$\A^{(\bullet,0,\bullet)}(\zP\sT^* P)=\A(\zP P^*)\,,$$
i.e. functions in variables $t,x,\dot{t},\dot{x}$ gives
$$\xd_J=\dot{{t}}\pa_t+\dot{x}^a\pa_{x^a}\,.$$
In the coordinates $(t,\dot{\mathfrak{t}},x^a,\dot{x}^a)$ in $E(P)=\R^\ti\ti\R\ti\sT\cM$, where
$\dot{\mathfrak{t}}=t^{-1}\dot{t}$ is odd and of the bi-degree $(0,1)$, this vector field reads as
$$\xd_J=\dot{\mathfrak{t}}t\pa_t+\dot{x}^a\pa_{x^a}\,.$$
Since the variables $\dot{\mathfrak{t}},\dot{x}^a$ are odd, any functions from the $C^\infty(\cM)$-module
$\A^{(1,r)}(\zP P^*)$ is of the form
$$t(f(x)w(\dot{x})+\dot{\mathfrak{t}}g(x)u(\dot{x})\,,$$
where $\zm=f(x)w(\dot{x})$ is of the bi-degree $(0,r)$, thus represents an $r$-form on $\cM$, and
$\zn=u(\dot{x})$ is of the bi-degree $(0,r-1)$, thus represents an $(r-1)$-form. The function
$\dot{\mathfrak{t}}$ represents the closed 1-form $\zF$ on $\R\ti\sT\cM$ being the projection on $\R$ and
$\dot{x}^a\pa_{x^a}$ represents the standard de Rham derivative $\xd$. Moreover,
$$(\dot{\mathfrak{t}}t\pa_t+\dot{x}^a\pa_{x^a})(t(\zm+\dot{\mathfrak{t}}\zn))=
t(\dot{x}^a\pa_{x^a}(\zm)+\dot{\mathfrak{t}}(\zm+\dot{x}^a\pa_{x^a}(\zn))\,$$ Hence,
$$\zW^r(\cM)\oplus\zW^{r-1}(\cM)\simeq \A^{(1,r)}(\zP P^*)\simeq\A^r(\zP(\R\ti\sT M))\,, \quad (\zm,\zn)\simeq\zm+\zF\we\zn\,,$$
and the cohomology operator $\xd_\cJ$ reads as
$$\xd_\cJ(\zm+\zF\we\zn)=\xd\zm+\zF\we(\zm+\xd\zn)\,,$$
that gives a well-know chain complex associated with a closed 1-form $\zF$.
\end{example}\label{eks2}
\begin{example} Starting with an even Kirillov bracket on a line bundle $L$, associated with a principal
Poisson $\R^\ti$-bundle structure $(\uP=(L^*)^\ti,\bar{\cJ},\bar{h}^0)$, we can consider the cohomology of the
linear principal Poisson $\R^\ti$-bundle $P=(\sT^* (L^*)^\ti,\dt\bar{\cJ},\sT \bar{h}^0,h^1)$ (cf. Example \ref{eks0}).
In this case, the $r$-cochains can be identified with homogeneous multivector fields on $(L^*)^\ti$. In the trivial case,
$L=\R\ti\cM$, we deal with a Jacobi structure on $\cM$ and the corresponding chain complex is exactly the complex of
the corresponding Lichnerowicz-Jacobi cohomology (cf. \cite{GM2,GL,LMP0,LMP,LLMP,Li}).
\end{example}

\section{Contact 2-manifolds}
Following the ideas of Roytenberg \cite{Roy0} in the symplectic case, let us consider contact 2-manifolds,
i.e. 2-manifolds equipped with a contact structure of weight 2. According to our general philosophy, we will
prefer an alternative definition (cf. Theorem \ref{t01a}).
\begin{theorem} A contact 2-manifold is canonically associated with a symplectic principal $\R^\ti$-bundle $(P,\zw,h^0,h^1)$ of degree 2.
\end{theorem}
\noindent In the above, $h^0$ represents the free $\R^\ti$-action on the principal bundle $\zt_0:P\to P_0$ and
$h$ is the homogeneity structure associated with an h-complete weight vector field $\bar{\zD})$. The
symplectic form $\zw$ is 1-homogeneous with respect to $h^0$ and is of weight 2 with respect to $\bar{\zD}$,
thus even, as in this case the parity is determined by the weight.

It is well known \cite{Vor1} that with a graded manifold, in particular a 2-manifold $\cM$, we can associate a tower of fibrations
$$\cM=\cM_{(2)}\ra\cM_{(1)}\ra\cM_{(0)}\,,$$
corresponding to the filtration $\A_{(0)}(\cM)\subset\A_{(1)}(\cM)\subset\A_{(2)}=\A(\cM)$ of the polynomial
algebra on $\cM$, where $\A_{(i)}$ is the subalgebra in $\A$ generated by polynomial functions of weight $\le
i$, $i=0,1,2$. Hence, $\cM_{(i)}$ is an $i$-manifold. In particular, $\cM_{(0)}$ is an even manifold and
$\cM_{(1)}=F[1]$ is a vector bundle over $\cM_{(0)}$ with odd fibers. In the case of an $\R^\ti$-principal
2-manifold, the fibrations intertwine the $\R^\ti$-action, so $\cM_{(0)}$ and $\cM_{(1)}$ are also principal
$\R^\ti$-bundles and $\cM\ra\cM_{(1)}\ra\cM_{(0)}$ is a morphism of principal $\R^\ti$-bundles. If,
additionally, $(\cM,\zw)$ is a symplectic $\R^\ti$-principal 2-manifold, the bundle $F[1]$ is a
pseudo-Euclidean with a pseudo-scalar product $\la\cdot,\cdot\ran$ and $(\cM,\zw)$ is the pullback of the
affine fibration
$$\sT^\ast[2]F[1]\to \left(F\oplus_{\cM_{(0)}}F^*\right)[1]\,,$$
equipped with the pulled-back canonical symplectic form, with respect to the isometric bundle embedding
$F[1]\ra(F\oplus_{\cM_{(0)}} F^\ast)[1]$ associated with $\la\cdot,\cdot\ran$. The pseudo-Euclidean product
induces a Poisson structure of on the principal $\R^\ti$-bundle $\cM_{(1)}=F[1]$ which makes it into a linear
odd principal Poisson $\R^\ti$-bundle, as the pseudo-Euclidean product is an $\R^\ti$-homogeneous linear odd
Poisson bracket. Such linear odd principal Poisson $\R^\ti$-bundles we will call {\it pseudo-Euclidean
principal $\R^\ti$-bundles}.

\begin{example}\label{ex2}
Let $(\uP,\bar{h}^0,\bar{h}^1)$ be a linear principal $\R^\ti$-bundle with the vector bundle structure
$\bar{\zt}_1:\uP\to\uP_1$. According to Theorem \ref{nlcs}, the cotangent bundle $P=\sT^*\uP$ equipped with
the cotangent lifts of the homogeneity structures $h^i=\sT^*\bar{h}^i$, $i=0,1$, and the canonical homogeneity
structure $h^2$ of the vector bundle structure $\pi_\uP:\sT^*\uP\to\uP$ is a contact 2-vector bundle with
respect to the canonical symplectic form $\zw=\zw_\uP$. Since $\zw$ is 1-homogeneous with respect tu $h^i$,
$i=0,1,2$, we can produce a 1-graded contact manifold of weight 2, $(P,\zw,h^0,h)$, out of it by composing the
two vector bundle homogeneity structures, $h=h^1\circ h^2$. Since the parity should be determined by the
weight, $\uP_1$ should be a purely even principal $\R^\ti$-bundle over a purely even manifold $M$. As the
linear principal $\R^\ti$-bundle $\uP$ can be written as the product $\uP=\uP_1\ti_M\uP_0$, where $\uP_0\to M$
is a vector bundle with odd linear functions, so that $\uP_0=F_0[1]$ should be the weight shift of a purely even
vector bundle $\zz:F_0\to M$.  Therefore
\be\label{pom}\uP=\uP_1\ti_MF_0[1]
\ee
and finally,
\be\label{e2} P=\sT^*[2]\left(\uP_1\ti_MF_0[1]\right)=\sT^*[2]\uP_1\ti_{\sT^*[2]M}\sT^*[2]F_0[1]\,,
\ee
so that $P_{(0)}=\uP_1$ and $P_{(1)}=E[1]=\uP_1\ti_M(F_0\oplus_M F_0^*)[1]$ is a vector bundle over $\uP_1$. For
trivial bundles $\uP_1=\R^\ti\ti M$ and $F_0=V\ti M$ with coordinates $(u,x^a)$ and $(\zx^i,x^a)$, thus locally
in general, we can write
$$P=\sT^*[2]\R^\ti\ti\sT^*[2]M\ti\sT^*[2]V\,.$$
Denoting the corresponding coordinates as $(u,z,x^a,p_a,\zx^i,\zh_i)$, $t\ne 0$, where $u,x^a$ are of weight
0, $\zx_i,\zh^i$ are of weight 1, and $z,p_a$ are of weight 2, we can write the symplectic form (which is of
weight 2) as
$$\zw=\xd z\xd u+\xd p_a\xd x^a+\xd\zh_i\xd\zx^i$$
and the corresponding Poisson tensor as
\be\label{pt} \cJ=\pa_z\pa_u+\pa_{p_a}\pa_{x^a}+\pa{\zh_i}\pa_{\zx^i}\,.
\ee

The pseudo-Euclidean structure on the bundle $E[1]$ is represented by the form $\xd\zh_i\xd\zx^i$. The
$\R^\ti$-action is the phase lift of the action on $\R^\ti$ by translations, so that
$$h^0_{t'}(u,z,x^a,p_a,\zx^i,\zh_i)=(t'u,z,x^a,t'p_a,\zx^i,t'\zh_i)\,.$$
In the coordinates $(t,z,x^a,\mathfrak{p}_a,\zx^i,\wt{\zh}_i)$, where $t=-u$, $u\mathfrak{p}_a=p_a$ and
$u\wt{\zh}_i={\zh}_i$, the action looks simpler,
$$h^0_{t'}(t,z,x^a,\mathfrak{p}_a,\zx^i,\wt{\zh}_i)=(t't,z,x^a,\mathfrak{p}_a,\zx^i,\wt{\zh}_i)\,,$$
but the form of $\zw$ is more sophisticated,
$$\zw=\xd t\xd z-t\left(\xd \mathfrak{p}_a\xd x^a+\xd\wt{\zh}_i\xd\zx^i\right)-
\xd t\left(\mathfrak{p}_a\xd x^a+\wt{\zh}_i\xd\zx^i\right)=\xd\left(t\left(\xd z-\mathfrak{p}_a\xd
x^a-\wt{\zh}_i\xd\zx^i\right)\right)\,,$$ and the $\R^\ti$-homogeneous pseudo-Euclidean form on
$E[1]=\R^\ti\ti_M\ti V[1]\ti V^*[1]$ is represented by the Poisson tensor
$-\frac{1}{t}\pa_{\wt{\zh}_i}\pa_{\zx^i}$ of weight -1 with respect to the vector field $t\pa_t$. On the other
hand, $\zw$ is manifestly associated with the contact form $\za=\xd z-\mathfrak{p}_a\xd
x^a-\wt{\zh}_i\xd\zx^i$ of weight 2 on
$$\sJ^1\left(\uP_1^e\ti_MF_0[1]\right)=\R[2]\ti M\ti V[1]\ti V^*[1]\,.$$
\end{example}
The above example shows that, in general, if $F$ is a purely even linear principal $\R^\ti$-bundle over $\cM$
with the linear $\R^\ti$ action $h^0$, then $F\oplus_\cM F^*$ is canonically a pseudo-Euclidean principal
$\R^\ti$-bundle with respect to action $\wh{h}^0_t=h^0_t\oplus t.\left(h^0_{t^{-1}}\right)^*$ and the
pseudo-Euclidean product induced from the canonical pairing between $F$ and $F^*$,
$$\la X+\zm,Y+\zn\ran =\frac{1}{2}\left( i_X\zn+i_Y\zm\right)\,.$$
This structure can be viewed as a reductions of the symplectic principal $\R^\ti$-bundle of degree 2
$\sT^*[2]F[1]$. If $F$ is a pseudo-Euclidean principal $R^\ti$-bundle, then the isometric bundle embedding
$$F[1]\to (F\oplus_\cM F^*)[1]\,,\quad X\mapsto X+\la X,\cdot\ran\,,$$
is simultaneously a morphism of linear principal $\R^\ti$-bundles. We then let $P$ be the symplectic principal
bundle of degree 2 being the pull-back of $\sT^*[2]F[1]$, i.e. completing the commutative diagram of morphisms
of linear principal Poisson $\R^\ti$-bundles
\be\label{c2m}\xymatrix{
P\ar[rr] \ar[d] &&
\sT^*[2]F[1]\ar[d] \\
F[1]\ar[rr] && (F\oplus_\cM F^*)[1] }
\ee
This leads to the following `contact variant' of the Roytenberg's result \cite[Theorem 3.3]{Roy}.
\begin{theorem}\label{cN2} Contact 2-manifolds are in one-to-one correspondence with pseudo-Euclidean principal $\R^\ti$-bundles.
The correspondence is given by the above construction.
\end{theorem}
\begin{example}
A canonical example of a contact 2-vector bundle is $P=\sT^*[2]\sT[1](L^*)^\ti$ for a line bundle $L$ over $M$ with the diagram of vector and principal $\R^\ti$-bundle morphisms

\be\label{bigdiagram}
\xymatrix{
               &  \dts[1](L^*)^\ti \ar[dd]\ar[dl]   &                     &  \dts[2]\sT[1](L^*)^\ti \ar[ll] \ar[dd] \ar[dl]
  \\
\sJ^1L[1] \ar[dd] &            & \sJ^1[1]C_L^* \ar[ll]\ar[dd]&
  \\
               &  (L^*)^\ti \ar[ld] &                 &  \sT[1](L^*)^\ti \ar[ll] \ar[ld]
  \\
M        &                   &    \DO^1[1](L,L)\ar[ll] & }
\ee
Here $\sJ^1C_L^*$ is the first jet bundle of the dual line bundle of $C_L\to\sJ^1L$, where $C_L\subset\sT^*\sJ^1L$ is the canonical contact bundle.
\end{example}
\section{Contact Courant algebroids}
Courant algebroids, originally defined in terms of even differential geometry \cite{LWX}, have been recognized
by Roytenberg \cite{Roy0,Roy} as derived brackets on symplectic 2-manifolds.

To be more precise, let us start with a symplectic 2-manifold $(\cM,\zw,h^1)$ with $\cM_{(1)}=F[1]$, so that
functions from $\A^1(\cM)$ (of weight 1) represent linear functions on $F$ and any section $e$ of $F$
represents a linear function $\zi_e$ on $F^*$.

Let $\BB{\cdot}{\cdot}$ be the Poisson bracket associated with the symplectic form $\zw$ of weight 2. The
symplectic Poisson bracket induces a nondegenerate pairing
\be\label{pairing}
\BB{\cdot}{\cdot}:\A^1(\cM)\ot_{\A^0(\cM)}\A^1(\cM)\to \A^0(\cM)
\ee
inducing an identification
$$F\simeq F^*\,, \quad e\mapsto \zi(e)\,,$$
i.e. an identification of the module $\Sec(F)$ of sections of $F$ and the
$\A^0(\cM)=C^\infty(\cM_{(0)})$-module $\A^1(\cM)$ of homogeneous functions of degree 1  by
$e=\BB{\zi(e)}{\cdot}$, and a pseudo-Euclidean product on $F$ defined by
\be\label{pEp}\la e,e'\ran=\BB{\zi(e)}{\zi(e')}\,.
\ee
If $H\in\A^3(\cM)$ is a homological cubic Hamiltonian, $\BB{H}{H}=0$, then the derived bracket,
\be\label{db}\{ \zf_1,\zf_2\}=\BB{\BB{\zf_1}{H}}{\zf_2}\,,
\ee
is even and satisfies the Jacobi identity (\ref{s2}) with $k=0$. Brackets (not necessarily skew-symmetric)
satisfying the Jacobi identity are called {\it Loday} (or {\it Leibniz}) {\it brackets}. Being of degree -2,
the bracket is closed on $\A^1(\cM)$, so that it can be viewed as a Loday bracket, the {\it Courant-Dorfman
bracket}, on sections of the vector bundle $F$ which will be denoted, with some abuse of notation, also
$\{\cdot,\cdot\}$, i.e. $\{ e,e'\}=\{\zi(e),\zi(e')\}$. The bracket $\{ \zi(e),f\}$ between functions
$\zi(e)\in\A^1(\cM)$ and basic functions $f\in\A^0(\cM)=C^\infty(\cM_{(0)})$ is a derivative with respect to
$f$ and corresponds to a vector bundle morphism $\zr:F\to\sT \cM_{(0)}$ (the anchor map). The structures
$\{\cdot,\cdot\}$ and $\zr$, enriched with the pseudo-Euclidean product (\ref{pEp}), form a so called {\it
Courant algebroid} structure on $F$. Moreover, maps $\{\cdot,\cdot\}$, $\zr$, and $\la\cdot,\cdot\ran$  arise
in this way if and only if the following conditions hold true (cf. \cite{GM2}):
\bea\label{mc1}
&\la \{ e,e'\},e'\ran=\la e,\{ e',e'\}\ran\,,\\
&\zr(e)\la e',e'\ran=2\la\{ e,e'\},e'\ran\,,\label{mc2}
\eea
for all section $e,e'$ of $F$. Note that (\ref{mc2}) implies that $\zr$ is the (left) anchor of the
Courant-Dorfman bracket $\{\cdot,\cdot\}$ in the usual sense, i.e.
\be\label{anchor} \{ e,fe'\}=f\{ e,e'\}+\zr(e)(f)e'\,,
\ee
for all sections $e,e'$ of $F$ and all $f\in C^\infty(\cM_{(0)})$. Indeed, the polarization of (\ref{mc2})
gives
\be\label{pol}\zr(e)\la
e',e''\ran=\la\{ e,e'\},e''\ran +\la\{ e,e''\},e'\ran\,.
\ee
Replacing $e'$ in the above with $fe'$ we conclude that
$$\zr(e)(f)\la e',e''\ran=\la\{ e,fe'\},e''\ran\,,
$$
whence (\ref{anchor}), as the product $\la\cdot,\cdot\ran$ is nondegenerate. The anchor property
(\ref{anchor}) implies, in turn, that the anchor map is a bracket homomorphism (cf. \cite{Gr}),
\be\label{hom}
\zr\left(\{ e,e'\}\right)=[\zr(e),\zr(e')]\,,
\ee
where the bracket on the right-hand side is clearly the bracket of vector fields. The latter follows also
directly from the graded Jacobi identity applied to the bracket $\BB{\cdot}{\cdot}$. To write a formula
similar to (\ref{anchor}) with respect to the multiplication by function in the first argument, let us first
polarize (\ref{mc1}) to get
$$\la \{ e,e'\},e''\ran+\la \{ e,e''\},e'\ran=\la e,\{ e',e''\}+\{ e'',e'\}\ran
$$
which, combined with (\ref{pol}), yields
\be\label{pol1}
\zr(e)\la e',e''\ran=\la e,\{ e',e''\}+\{ e'',e'\}\ran\,.
\ee
Defining a derivation $\sD:C^\infty(M)\ra\Sec(F)$ by means of the pseudo-Euclidean product by
\be\label{D}\la\sD(f),e\ran=\zr(e)(f)\,, \ee
we can rewrite (\ref{pol1}) as
\be\label{sym}
\{ e',e''\}+\{ e'',e'\}=\sD\la e',e''\ran\,.
\ee
Replacing $e'$ with $fe'$ in the latter and using (\ref{anchor}), we get
\be\label{anchor2}
\{ fe',e''\}=f\{ e',e''\}-\zr(e'')(f)e'+\la e',e''\ran\sD(f)\,.
\ee

If we start not with a bare symplectic 2-manifold but with a contact 2-manifold, i.e. with a symplectic
principal \ $\R^\ti$-bundle $(P,\zw,h^0,h)$ of degree 2, we can repeat the above construction. Note that in
this case $P$, $P_{(1)}=F[1]$, and $P_{(0)}$ are additionally principal $\R^\ti$-bundles. As our structure is
reacher, e.g. the bracket $\BBP{\cdot}{\cdot}=\BB{\cdot}{\cdot}$ is homogeneous of degree -1 with respect to
the $\R^\ti$-action $h^0$, we can consider the space $\A^{(1,\bullet)}(P)=\oplus_{i=0}^\infty \A^{(1,i)}(P)$
which is closed with respect to this Poisson bracket, where $\A^{(1,i)}(P)$ is the subspace in the space of
functions of weight $i$ on $P$ which are simultaneously 1-homogeneous with respect to $h^0$. In this notation,
elements of $\A^{(1,0)}(P)$ represent homogeneous basic functions, i.e. functions on $P_{(0)}$ which are
1-homogeneous with respect to the reduced $\R^\ti$-action on $P_{(0)}$, and elements of $\A^{(1,1)}(P)$
represent homogeneous sections of $F$.

If we now choose an $\R^\ti$-homogeneous cubic homological Hamiltonian $H\in\A^{(1,3)}(P)$, $\BBP{H}{H}=0$,
the derived bracket (\ref{db}) is of degree $(-1,-1)$. In particular,
$$\{\A^{(1,1)}(P),\A^{(1,1)}(P)\}\subset \A^{(1,1)}(P)\,,$$
so it is closed on homogeneous sections of $F$. This suggests the following definition {\it a la} Roytenberg.
\begin{definition}
A {\it contact Courant algebroid} is a contact 2-manifold $P$ equipped with a an $\R^\ti$-homogeneous cubic
homological Hamiltonian  $H\in\A^{(1,3)}(P)$, $\BBP{H}{H}=0$.
\end{definition}
In this sense a contact Courant algebroid is an $\R^ti$-homogeneous Courant algebroid. To find a more
`classical' description of the contact Courant algebroid, note that the space $\A^{(1,1)}(P)$, as well as
the space $\A^{(1,0)}(P)$ of homogeneous functions on $P_{(0)}$, is an $\A^{(0,0)}(P)$-module, where elements
of $\A^{(0,0)}(P)$ represent functions on the base $M=P_{(0)}/\R^\ti$ of the principal $\R^\ti$-bundle
$P_{(0)}\to M$. Homogeneous functions on $P_{(0)}$ represent, in turn, sections of the line bundle
$L=P_{(0)}^e$. Since
$$\{\A^{(1,1)}(P),\A^{(1,0)}(P)\}\subset \A^{(1,0)}(P)\,,$$
with any homogeneous linear section $e$ of $E$ we can associate the first-order differential operator $\zr^{\!
1}(e)$ from $L$ into $L$ represented by
$$\{ \zi(e),\cdot\}:\A^{(1,0)}(P)\to \A^{(1,0)}(P).$$
Similarly to the decomposition (\ref{pom}) one can prove that $F=P_{(0)}\ti_M F_1$ for a vector bundle $F_1$
over $M$, so that elements of $\A^{(1,1)}(P)$ represent, {\it via} the identification $e\mapsto \zi(e)$,
sections of the vector bundle $\cE(F)=L\ot_MF_1$ and $\zr^{\! 1}$ can be viewed as a vector bundle morphism
\be\label{uu}\zr^{\! 1}:\cE(F)\to\DO^1(L,L)\,.
\ee
As
$$\BBP{\A^{(1,1)}(P)}{\A^{(1,1)}(P)}\subset \A^{(1,0)}(P)\,,$$
the symplectic Poisson bracket induces an $L$-valued pseudo-Euclidean product on $\cE(F)$ defined by
\be\label{pEp1}\la X,Y\ran^{\! 1}=\BBP{\zi(X)}{\zi(Y)}\,.
\ee
The Loday bracket on sections of $\cE(F)$ induced by $\{\cdot,\cdot\}$, the {\it Courant-Jacobi bracket}, we
will denote with $\{\cdot,\cdot\}^{\! 1}$ The graded Jacobi identity for $\BBP{\cdot}{\cdot}$ immediately
implies the identities
\bea\label{mc1a}
&\la \{ X,Y\}^{\! 1},Y\ran^{\! 1}=\la X,\{ Y,Y\}^{\! 1}\ran^{\! 1}\,,\\
&\zr^{\! 1}(X)\la Y,Y\ran^{\! 1}=2\la\{ X,Y\}^{\! 1},Y\ran^{\! 1}\,.\label{mc2a}
\eea this time valid for sections $X,Y$ of $\cE(F)$ and $L$-valued pseudo-Euclidean product
$\la\cdot,\cdot\ran^{\! 1}$ with values in the line bundle $L$. Property (\ref{mc2a}) implies also that
\be\zr^{\! 1}\left(\{ X,Y\}^{\! 1}\right)=[\zr^{\! 1}(X),\zr^{\! 1}(Y)]\,.\label{homa}
\ee
Here, of course, the right-hand side bracket in (\ref{homa}) is the commutator bracket of first-order
differential operators. To prove (\ref{homa}), consider the polarization of (\ref{mc2a}),
\be\label{pol2}\zr^{\! 1}(X)\la
Y,Z\ran^{\! 1}=\la\{ X,Y\}^{\! 1},Z\ran^{\! 1} +\la\{ X,Z\}^{\! 1},Y\ran^{\! 1}\,.
\ee
This implies
\bea\nn\zr^{\! 1}(X')\zr^{\! 1}(X)\la
Y,Z\ran^{\! 1}&=&\la\{ X',\{ X,Y\}^{\! 1}\}^{\! 1},Z\ran^{\! 1} +\la\{ X,Y\}^{\! 1},\{ X',Z\}^{\! 1}\ran^{\! 1} \\
&&+\la\{ X',\{ X,Z\}^{\! 1}\}^{\! 1},Y\ran^{\! 1}+\la\{ X,Z\}^{\! 1},\{ X'Y\}^{\! 1}\ran^{\! 1}\,,\label{pol3}
\eea
so
\bea\nn [\zr^{\! 1}(X'),\zr^{\! 1}(X)]\la
Y,Z\ran^{\! 1}&=&\la\{\{ X', X\}^{\! 1},Y\}^{\! 1},Z\ran^{\! 1} +\la\{ \{ X', X\}^{\! 1},Z\}^{\! 1},Y\ran^{\! 1}\\
&=&\zr^{\! 1}(\{ X',X\}^{\! 1})\la Y,Z\ran^{\! 1}\label{pol4}
\eea
and (\ref{homa}) follows, since $\la\cdot,\cdot,\ran^{\! 1}$ is nondegenerate and $\la Y,Z\ran^{\! 1}$ can be
arbitrary.

Like in the case of a linear principal Poisson $\R^\ti$-bundle, we can consider the linear principal
$\R^\ti$-bundle $E=E(F)=(L^*)^\ti\ti_M\cE(F)$ and the canonical isomorphism $\mathfrak{H}:E(F)\to F^*$ given
by (\ref{recc1}). In particular, any section $\ze$ of $\cE(F)$ defines an invariant section $\bar{\ze}$ of $E$
and, {\it via} the isomorphism $\mathfrak{H}$, a 1-homogeneous section of $F^*$, thus again a function
$\zi_\ze\in\A^{(1,1)}(F)$.

The pseudo-Euclidean product on the vector bundle $F\simeq F^*$ can be reconstructed from the $L$-valued
pseudo-Euclidean product on $\cE(F)$ as the unique one satisfying the formula
\be\label{recon1}\la\wt{\ze},\wt{\ze}'\ran(l^*)=l^*\left(\la\ze,\ze'\ran^{\! 1}\right)\,.
\ee
We can obtain the anchor $\zr:E\to\sT P_{(0)}=(L^*)^\ti$ using a natural identification of first-order
differential operators $D$ from $L$ to $L$ with invariant vector fields $\bar{D}$ in the principal
$\R^\ti$-bundle $(L^*)^\ti$,
\be\label{recon3a} \zr(\wt{\ze})=\overline{\zr^{\! 1}(\ze)}\,.
\ee
and hence the derivation $D:C^\infty(P_{(0)})\to\Sec(F)$ which is defined in terms of $\zr$ and
$\la\cdot,\cdot\ran$. We can also reconstruct the bracket $\{\cdot,\cdot\}$ on sections of $F$. First of all,
\be\label{recon2} \{\wt{\ze},\wt{\ze}'\}=\wt{\{\ze,\ze'\}^{\! 1}}\,.
\ee
The sections of the form $\wt{\ze}$ generate the vector bundle $E$, so the above expression uniquely defines
$\{\cdot,\cdot\}$ in view of identities (\ref{anchor}) and (\ref{anchor2}) which will give
$$\{ f\wt{\ze},g\wt{\ze}'\}=fg\{ \wt{\ze},\wt{\ze}'\}+f\zr(\wt{\ze})(g)e'-g\zr(\wt{\ze'})e +g\ra e,e'\ran \sD(f)\,.$$

\medskip
Conversely, suppose that we have a structure $(\cE,L,\{\cdot,\cdot\}^{\! 1},\la\cdot,\cdot\ran^{\! 1},\zr^{\!
1})$ such that
\begin{itemize}
\item[(a)] $\cE$ is a vector bundle  over $M$ and $L$ is a line bundle over $M$,
\item[(b)] $\{\cdot,\cdot\}^{\! 1}$ is a Loday bracket  on sections of $\cE$,
\item[(c)] $\la\cdot,\cdot\ran^{\! 1}$ is a pseudo-Euclidean product  with values in $L$,
\item[(d)] $\zr^{\! 1}:\cE\ra\DO^1(L,L)$ is a vector bundle morphism  associating to any section $X$ of $\cE$ a
first-order differential operator $\zr^{\! 1}(X)$ from sections of $L$ into sections of $L$,
\item[(e)] identities (\ref{mc1a}) and (\ref{mc2a})
are satisfied for all sections $X,Y$ of $\cE$.
\end{itemize}

Let us define a linear principal $\R^\ti$-bundle $E=(L^*)^\ti\ti_M\cE$, an $\R^\ti$-homogeneous
pseudo-Euclidean product $\la\cdot,\cdot\ran$ on $E$ by (\ref{recon1}), the anchor map $\zr:E\to\sT M$ by
(\ref{recon3}), and a Loday bracket $\{\cdot,\cdot\}$ on sections of $E$ satisfying (\ref{anchor}) and
(\ref{anchor2}) by (\ref{recon2}). It can be checked directly in local coordinates that the constructions are
correct and give us a Courant algebroid structure on $E$. As in this case $E$ is simultaneously a principal
$\R^\ti$-bundle, it is easy to see that all these structures are $\R^\ti$-homogeneous, so the cubic
Hamiltonian on the minimal symplectic realization defining the derived bracket is homogeneous as well. Summing
up, we can propose the following.
\begin{theorem}
There is a one-to-one correspondence between contact Courant algebroids and structures
$(\cE,L,\{\cdot,\cdot\}^{\! 1},\la\cdot,\cdot\ran^{\! 1},\zr^{\! 1})$ described by items (a)-(d) above.
\end{theorem}
\begin{remark} If the line bundle is trivial, $L=\R\ti M$, the above
description of Courant-Jacobi algebroids coincides with the one in \cite[Definition 1]{GM2} if we take into
account that identity (\ref{homa}) follows from (\ref{mc2a}), so is superfluous in the definition.
\end{remark}

\begin{example}
Consider a contact Courant algebroid $(\cE,L,\{\cdot,\cdot\},\la\cdot,\cdot\ran,\zr)$ defined as above. We
can always find a local trivializations of $L$ and $\cE$ with associated local affine coordinates $(x^a,z)$
and $(x^a,\zvy^i)$, respectively. The pseudo-Euclidean product with values in $L$ is in the local
trivialization of $L$ a standard product with values in $\R$. Moreover, a basis $(\ze_i)$ of local sections of
$\cE$ can always be chose such that $\la\ze_i,\ze_j\ran=g_{ij}$ and the basic functions $g_{ij}$ are
constants. Since by construction $\zvy^i(\ze_j)=\zd^i_j$, our identification of sections with linear functions
{\it via} the pseudo-Euclidean product yields $\ze_i=g_{ij}\zvy^j$. Let us write in local coordinates
$$\zr(e_i)=r^a_i(x)\pa_{x^a}+r_i(x)$$
and
$$\Theta(\ze_i,\ze_j,\ze_k)=\la\{ \ze_i,\ze_j\},\ze_k\ran=A_{ijk}(x)\,,$$
Since the product $\la\cdot,\cdot\ran$ is nondegenerate, the above formula defines the brackets $\{
\ze_i,\ze_j\}$ uniquely. Note that $A_{ijk}(x)$ i totally skew-symmetric in $(i,j,k)$. Indeed, combining
(\ref{mc1a}) and (\ref{mc2a}), we get
\be\label{pol1a}
\zr(\ze_k)\la \ze_i,\ze_j\ran=\la \ze_k,\{ \ze_i,\ze_j\}+\{ \ze_j,\ze_j\}\ran=A_{ijk}(x)+A_{jik}(x)\,.
\ee
Since $\la \ze_i,\ze_j\ran$ is constant, the left-hand side is 0, thus $A_{ijk}(x)$ is skew-symmetric with
respect to the first two arguments. Similarly, (\ref{pol2}) gives
$$\zr(\ze_i)\la
\ze_j,\ze_k\ran=\la\{ \ze_i,\ze_j\},\ze_k\ran +\la\{ \ze_i,\ze_k\},\ze_j\ran=A_{ijk}(x)+A_{ikj}(x)=0\,,
$$
so $A_{ijk}(x)$ is skew-symmetric also with respect to the last two arguments, hence totally skew-symmetric.

The linear principal $\R^\ti$-bundle $E=(L^*)^\ti\ti_M\cE$ has local coordinates $(t,x^a,\zvy^i)$, $t\ne 0$,
and the $\R^\ti$-action reduces to the regular action on the first argument,
$h^0_s(t,x^a,\zvy^i)=(st,x^a,\zvy^i)$. Homogeneous sections $e_i=\wt{\ze}_i$ of $E$ read $e_i(t,x)=t\ze_i(x)$
and correspond to linear homogeneous functions $tg_{ij}\zvy^j$. The pseudo-Euclidean product is determined by
$$\la e_i,e_j\ran=tg_{ij}$$
and corresponds to the odd Poisson structure
$$\zL=-\frac{1}{2t}\sum_{i,j}g^{ij}\pa_{\zvy^i}\pa_{\zvy^j}$$
on $E$. Here, of course $(g^{ij})=(g_{ij})^{-1}$. The minimal symplectic realization $P$ is a symplectic
principal $\R^\ti$ bundle of degree 2. In local Darboux coordinates $(t,x^a,\zvy^i,z,p_a)$ the symplectic form
reads as
\be\label{sf1} \zw=\xd z\xd t + \xd p_a\xd x^a+\frac{t}{2}g_{ij}\xd\zvy^i\xd\zvy^j\,.
\ee
Note that, among local coordinates, $t,x^a$ are of weight 0 thus even, $\zvy^i$ are of weight 1 thus odd, and
$z,p_a$ are of weight thus even, so that $\zw$ is of weight 2 and even. Moreover, with respect to the
$\R^\ti$-action, $x^a,\zvy^i, z$ are invariant, and $t,p_a$ are homogeneous of degree 1, so that $\zw$ is
homogeneous of degree 1. The corresponding Poisson tensor $\cJ=\zw^{-1}$ of weight 2 takes the form (cf.
\ref{lfe2})
\be\label{lfe2a}{\cJ}=\pa_t\pa_z+\pa_{x^a}\pa_{p_a}-\frac{1}{2t}g^{ij}\pa_{\zvy^i}\pa_{\zvy^j}\,.\ee
If we use invariant coordinates $\mathfrak{p}_a$ instead of $p_a=t\mathfrak{p}_a$, the Poisson tensor reads
(cf. \ref{lfe2})
\be\label{lfe2b}{\cJ}=\pa_t\pa_z+\frac{1}{t}\left[\pa_z\left(\mathfrak{p}_a\pa_{\mathfrak{p}_a}+\zvy^j\pa_{\zvy^j}\right)+
\pa_{x^a}\pa_{\mathfrak{p}_a}-\frac{1}{2}g^{ij}\pa_{\zvy^i}\pa_{\zvy^j}\right]\,,\ee

The homogeneous Courant-Dorfman bracket on sections of $E$ satisfies clearly
\be\label{cdn}
\la\{ e_i,e_j\},e_k\ran(t,x)=tA_{ijk}(x)
\ee
and its anchor map, $\zr:E\to(L^*)^\ti$,
\be\label{cdn1} \zr(e_i)=r_i^a(x)\pa_{x^a}+r_i(x)t\pa_t\,.
\ee
All these data imply that the the corresponding homogeneous cubic Hamiltonian must be of the form
\be\label{cH} H=\zvy^i\left(r^a_i(x)p_a+r_i(x)tz\right)-\frac{t}{6}A_{ijk}(x)\zvy^i\zvy^j\zvy^k\,.
\ee
indeed, since $\BBP{e_i}{\zvy^j}=\zd^j_i$, we have
$$\BBP{e_i}{H}=r_i^a(x)p_a+r_i(x)tz-\frac{t}{2}A_{ijk}(x)\zvy^j\zvy^k\,,$$
so that
$$\la\{ e_i,e_j\},e_k\ran=\BBP{\BBP{\BBP{e_i}{H}}{e_j}}{e_k}=tA_{ijk}(x)$$
and
$$\zr(e_i)(f)=\BBP{\BBP{e_i}{H}}{f}=r_i^a(x)\frac{\pa f}{\pa x^a}+r_i(x)t\frac{\pa f}{\pa t}
$$
for any basic function $f=f(t,x)$. Of course, the Jacobi identity for the derived bracket $\{\cdot,\cdot\}$,
together with (\ref{mc1}) and (\ref{mc2}), is equivalent with the homological condition $\BBP{H}{H}=0$.
\end{example}
\begin{example}
Consider the contact 2-manifold $\sT^*[2]\sT[1](\R^\ti\ti M)$ for a purely even manifold $M$, like in Example
\ref{eks1}. As the cubic Hamiltonian $H$ associated with the canonical vector field on $\sT[1](\R\ti M)$ being the de Rham derivative is 1-homogeneous, we obtain a homogeneous Courant bracket on the linear principal $\R^\ti$-bundle $F=\sT(\R^\ti\ti M)\oplus_{\R^\ti\ti M}\sT^*(\R^\ti\ti M)$. As we can interpret $\cE(F)$ as the bundle $\cE=(\R\ti\sT M)\oplus_M(\R^*\ti\sT^*M)$ whose sections are $(X,f)+(\za,g)$, where $f,g\in\C^\infty(M)$, $X$ is a vector field, and $\za$ is a one-form on $M$, the contact Courant algebroid structure on $E$ consists of
\begin{itemize}
\item[(a)] the Loday bracket  on sections of $\cE$ of the form
\bea\label{edirac}
\{(X_1,f_1)+(\za_1,g_1),(X_2,f_2)+(\za_2,g_2)\}^1=
\left([X_1,X_2],X_1(f_2)-X_2(f_1)\right)&&\\
+\left(\Li_{X_1}\za_2-i_{X_2}\xd\za_1+f_1\za_2- f_2\za_1
+f_2\xd g_1+g_2\xd f_1,
X_1(g_2)-X_2(g_1)+i_{X_2}\za_1+f_1g_2\right)\,,&&\nn
\eea

\item[(b)] the pseudo-Euclidean product of the form
$$\la(X,f)+(\za,g),(X,f)+(\za,g)\ran^1=\la X,\za\ran+fg\,,$$

\item[(c)] the vector bundle morphism $\zr^{\! 1}:\cE\ra\DO^1(M)$ of the form
$$\zr^{\! 1}\left((X,f)+(\za,g)\right)=X+f\,.$$
\end{itemize}
This is the Dorfman form of the bracket whose skew-symmetrization gives exactly the bracket introduced by Wade \cite{Wa} to define so called {\it $\cE^1(M)$-Dirac structures}.
\end{example}

\bigskip
\noindent \noindent Janusz Grabowski\\Institute of Mathematics, Polish Academy of Sciences\\\'Sniadeckich 8,
P.O. Box 21, 00-956 Warszawa,
Poland\\{\tt jagrab@impan.pl}\\\\

\end{document}